\pdfoutput=1
\date{} 
\title{The scaling limits of near-critical and dynamical percolation}

\author{Christophe Garban \and G\'abor Pete \and Oded Schramm}

\documentclass[11pt,preprint]{article}

\usepackage[vmargin={2.5cm, 3cm},hmargin=2.5cm]{geometry}

\usepackage[pdfpagelabels]{hyperref}

\usepackage{latexsym}
\usepackage{amsmath}
\usepackage{amssymb}
\usepackage{amsthm}
\usepackage{mathrsfs}
\usepackage{amsfonts}
\usepackage{graphicx}
\usepackage{url}

\usepackage{color}
\usepackage{bbm,verbatim}

\usepackage{minitoc} 

  \ifx\LabelFigloaded\MYundefined\relax
  \else
   \fi

  \def\LabelFigloaded{\relax}

  \chardef\LabelFigCatAt\the\catcode`\@
  \catcode`\@=11

 \let\LabelFigwlog@ld\wlog
 \def\wlog#1{\relax}

 \ifx\\\MYundefined@
    \let\\\relax
 \fi

  \def\ms@g{\immediate\write16}

 \def\N@wif{\csname newif\endcsname }
 \def\Temp@ {\N@wif\ifIN@}
 \ifx\INN@\MYundefined@
    \else \let\Temp@\relax
 \fi
 \Temp@

  \def\IN@{\expandafter\INN@\expandafter}
  \long\def\INN@0#1@#2@{\long\def\NI@##1#1##2##3\ENDNI@
    {\ifx\m@rker##2\IN@false\else\IN@true\fi}%
     \expandafter\NI@#2@@#1\m@rker\ENDNI@}
  \def\m@rker{\m@@rker}
 
  \newtoks\Initialtoks@  \newtoks\Terminaltoks@
  \def\SPLIT@{\expandafter\SPLITT@\expandafter}
  \def\SPLITT@0#1@#2@{\def\TTILPS@##1#1##2@{%
     \Initialtoks@{##1}\Terminaltoks@{##2}}\expandafter\TTILPS@#2@}

 \def\Shifted@@#1#2#3{\setbox0=\hbox{#3}%
   \raise -\dp0\vbox {\kern-#2%
       \hbox {\kern#1\unhbox0\kern-#1}%
           \kern#2}}

 \newcount\gridcount
 \newbox\auxGridbox@ \newbox\hGridbox@ \newbox\vGridbox@
 \newbox\Labelbox@ \newbox\auxLabelbox@
 \newbox\Coordinatebox@
 \newtoks\Labeltoks@
 \newdimen\Wdd@ \newdimen\Htt@
 \newdimen\Wddd@ \newdimen\Httt@
 
 \def\Wr@{\immediate\write16}

 \newdimen\GL@wd
 \def\GridLineWidth#1{\GL@wd=#1}

 \def\gobble#1{}
 \def\EdgeErr@{\Wr@{}%
      \Wr@{\string\Edges\space argument
      1, 10, 100 or 1000 please\string!}%
      }

 \newcount\Edgect@

 \def\Sweepup#1\endSweepup{}

 \def\SetEdges@{%
    \edef\Zr@@s{\expandafter\gobble\number\Edgect@\empty}%
        \count255=0\Zr@@s\relax
        \ifnum\count255=\z@\else\EdgeErr@\show\tailtest\fi
        \count255=1\Zr@@s\relax
        \ifnum\count255=\Edgect@\relax\else\EdgeErr@\show\leadtest\fi
    \EdgGl@b\edef\Zr@s{\expandafter\gobble\Zr@@s\empty}
    \ifnum\Edgect@>\@ne\relax\EdgGl@b\let\L@Dc\empty
        \else\EdgGl@b\edef\L@Dc{\string.}\fi
    \ifnum\Edgect@>\@ne\relax
        \EdgGl@b\edef\Edgescale@##1{\divide##1 by \Edgect@}%
        \else\EdgGl@b\edef\Edgescale@##1{}\fi
    }

 \def\Edges#1{\Edgect@=#1\relax
     \let\EdgGl@b\global \SetEdges@}

 \Edges{1}

 \def\hhrule{\hrule height \GL@wd\vskip-.\GL@wd}

 \def\hRule@{%
   \advance\gridcount -2%
   \vfil\hhrule\vfil
   \llap{\smash{\raise -2.5pt
     \hbox{\L@Dc\number\gridcount\Zr@s\kern2pt}}}%
   \hhrule
   }

\def\vvrule{\vrule width \GL@wd \kern-\GL@wd}

 \def\vRule@{\advance\gridcount 2%
   \hfil\vvrule\hfil
   \setbox\auxGridbox@=\vbox to 0pt
      {\vskip \Htt@\vskip 2pt
        \hbox to 0pt{\hss\L@Dc\number\gridcount\Zr@s\hss}\vss}%
      \wd\auxGridbox@=0pt \box\auxGridbox@
   \vvrule
   }

 \def\PlaceGrid@@{\gridcount=10 
  \setbox\hGridbox@=\hbox{%
        \hbox{%
             \hskip-.4pt\vrule
             \vbox to \Htt@{%
               \offinterlineskip\parindent=\z@\relax
               \hbox to \Wdd@{\hfil}
               \hRule@\hRule@\hRule@\hRule@
               \vfil\hhrule\vfil}%
             \vrule\hskip-.4pt}
    }%
  \gridcount=0%
  \setbox\vGridbox@=\hbox{%
      \vbox{\offinterlineskip\parindent=0pt\hsize=0pt
         \vskip-.4pt\hrule%
         \hbox to \Wdd@{%
                 \vtop to \Htt@{\vfil}%
                 \vRule@\vRule@\vRule@\vRule@
                 \hfil\vvrule\hfil}%
         \hrule\vskip-.4pt}}%
  \wd\hGridbox@=0pt\ht\hGridbox@=0pt
  \wd\vGridbox@=0pt\ht\vGridbox@=0pt
  \hbox{\box\hGridbox@\box\vGridbox@}%
  }

 \def\LabelsGlobal{\def\LabGl@b{\global}}
 \def\LabelsLocal{\def\LabGl@b{}}
 \LabelsGlobal 

 \def\SetLabels#1\endSetLabels{%
   \LabGl@b\Labeltoks@={#1()\\}%
   }

 \LabGl@b\Labeltoks@={()\\}

 \def\ShowGrid{\LabGl@b\let\PlaceGrid@\PlaceGrid@@}
 \def\HideGrid{\LabGl@b\let\PlaceGrid@\relax}
 \def\Grids{\ShowGrid\LabGl@b\let\GridSwitch@\ShowGrid}
 \def\noGrids{\HideGrid\LabGl@b\let\GridSwitch@\HideGrid}

 \noGrids

 \def\bAdjust@@{%
     \setbox\auxLabelbox@=\hbox{\raise \dp\auxLabelbox@
            \box\auxLabelbox@}}
 \def\bAdjust@{\let\vAdjust@\bAdjust@@}

 \def\eAdjust@@{\dimen0=-.5\ht\auxLabelbox@
     \advance\dimen0 by .5\dp\auxLabelbox@
     \setbox\auxLabelbox@=
            \hbox{\raise\dimen0\box\auxLabelbox@}}
 \def\eAdjust@{\let\vAdjust@\eAdjust@@}

 \def\tAdjust@@{%
     \setbox\auxLabelbox@=\hbox{\raise-\ht\auxLabelbox@
            \box\auxLabelbox@}}
 \def\tAdjust@{\let\vAdjust@\tAdjust@@}

 \let\vAdjust@\relax

 \def\lAdjust@{\let\hAdjust@\rlap}
 \def\rAdjust@{\let\hAdjust@\llap}

 \let\hAdjust@\relax\let\vAdjust@\relax

 \def\FetchLabel@#1(#2)#3\\{%
     \IN@0#2@@\ifIN@
        \setbox0=\hbox{\ignorespaces#1#3\unskip}%
        \ifdim\wd0>0pt
           \ms@g{}%
           \ms@g{ !!! Bad label(s)? !!!}%
           \message{ #1(#2)#3}%
        \fi
        \def\LabelMole@##1\endFetchLabel@{%
            \IN@0()\\@##1@%
            \ifIN@\def\Temp@{\FetchLabel@##1\endFetchLabel@}%
            \else\def\Temp@{}%
            \fi
            \Temp@
           }%
     \else
       \ignorespaces#1\unskip
       \setbox\auxLabelbox@=%
         \hbox to 0pt{\hss\ignorespaces\hAdjust@
          {\ignorespaces#3\unskip}\hss}%
       \vAdjust@
       \let\hAdjust@\relax\let\vAdjust@\relax
       \AugmentLabelBox@@{#2}%
       \ht\Labelbox@=0pt\dp\Labelbox@=0pt
       \let\LabelMole@\FetchLabel@%
     \fi\LabelMole@}

 \newtoks\XYSep@ 
 \def\SetXYSeparator#1{%
     \IN@0#1@@\ifIN@\XYSep@{*}%
     \else
     \XYSep@{#1}%
     \fi
     }

 \SetXYSeparator*

 \def\AugmentLabelBox@@#1{%
     \IN@0\the\XYSep@ @#1@\ifIN@
       \SPLIT@0\the\XYSep@ @#1@%
       \setbox\Labelbox@=\hbox to 0pt{%
         \unhbox\Labelbox@
         \Shifted@@{\the\Initialtoks@\Wddd@}%
         {\the\Terminaltoks@\Httt@}%
         {\box\auxLabelbox@}}%
     \else
         \ms@g{}%
         \ms@g{ !!! Bad insertion point. !!!}%
         \message{ (#1\ this point was rejected.)}%
     \fi
    }

 \def\FetchOption@#1[#2]#3\endFetchOption@{%
    \def\temp{#1}
    \ifx\temp\empty
       \Edgect@=#2\relax
       \let\EdgGl@b\relax
       \SetEdges@
       \Cleaner@#3%
    \fi}

 \def\Cleaner@#1[@]{\Labeltoks@{#1}}
     
 \def\PlaceLabels@@{\mathsurround=0pt
     \def\Cr@{\\}%
     \let\L\lAdjust@\let\R\rAdjust@
     \let\B\bAdjust@\let\E\eAdjust@\let\T\tAdjust@
     \expandafter\FetchOption@\the\Labeltoks@[@]\endFetchOption@
     \Wddd@=\Wdd@ \Edgescale@\Wddd@ 
     \Httt@=\Htt@ \Edgescale@\Httt@
     \expandafter\FetchLabel@\the\Labeltoks@\endFetchLabel@
     \box\Labelbox@
     }%

 \let \PlaceLabels@\PlaceLabels@@

 \def\AffixLabels#1{\setbox\Coordinatebox@=\hbox{#1}%
      \Wdd@=\wd\Coordinatebox@ \Htt@=\ht\Coordinatebox@
      \advance\Htt@ \dp\Coordinatebox@
      \hbox{\copy\Coordinatebox@\kern-\Wdd@ 
           \Shifted@@{0pt}{-\dp\Coordinatebox@}%
           {\PlaceLabels@\PlaceGrid@}%
           \kern\Wdd@}%
      \GridSwitch@ 
      \LabGl@b\Labeltoks@{()\\}%
      }
 
   \let\wlog\LabelFigwlog@ld   
   \catcode`\@=\LabelFigCatAt  

                                By

              Raymond S\'eroul <A18645@FRCCSC21.BITNET>
                                and 
              Laurent Siebenmann <lcs@topo.math.u-psud.fr>
    
              VERSIONS: July 1991, Oct 1991, Jan 1992, July 1992

INTRODUCTION

      This labelling package is intended for TeX users who
rely on non-TeX sources for for their graphics inserts.  It
provides means for adding TeX labels to such inserts with a
minimum of fuss. 

       For most labels, TeX users have in the past found it
reasonably convenient to rely on non-TeX sources. Typical
occasions when an inescapable need for TeX labels seemed to
arise are

 (a) when the graphics program lacks certain exotic or complex
mathematical symbols

 (b) when the very highest typographical quality is wanted for the
labels

 (c) when labels included with the graphics fail to print, 
 and you cannot figure out why (cf. boxedeps.doc).  The labels
 provided by labelfig.tex are 100

       Since this package first appeared, many users, who in the
past scarcely dreamed of using TeX labels, have come to use
nothing but.  So it is now appropriate to add

Intoxication Warning:  TeX labels may be addictive and expensive. 

     If you have a fast preview you may disagree, and even find
that this package provides an agreeable paste-up environment; see
extra applications at end.

     Note to publishers: It is possible and convenient to ultimately
export the TeX labels produced by labelfig.tex to become an integral
part of the EPS file. This is often desired by a publisher who typically
uses an "upmarket" graphics or page layout program, with which the
staff is skilled in perfecting figures.  See Appendix I for
a recipe.

     The authors are grateful to Patrick Ion of Math Reviews for
helpful comments and encouragement.

BASIC INSTRUCTIONS

    After reading in the macro file using

preview or proof your figure with a coordinate grid printed on
top, by typing the following:

    \ShowGrid  
    \AffixLabels{<the graphics insertion>}

Here <the graphics insertion> is what you would type to insert
the graphics object alone without the grid.  This must provide
for the space around it. For example <the graphics insertion>
might well be \BoxedEPSF{MyFigure scaled 700} using the
boxedeps.tex macro package (from same source); this provides a
TeX box containing the encapsulated PostScript insert specified by
the file MyFigure. \AffixLabels{...} provides the grid (supposing
\ShowGrid is present) and later, once you have specified labels
using the grid, it will "tack on" the labels.

     The grid is a sort of (usually elongated) checkerboard of
ten rows and ten columns and its (internal) partitions are by
default numbered  .1, ... ,.9  both horizontally (X-coordinate
running left to right) and vertically (Y-coordinate running bottom
to top).  Thus the points enclosed by the grid correspond to the
points of the unit square in the cartesian "X-Y" plane, the lower
left corner corresponding to the origin (0,0).  By extrapolation,
the full page corresponds to a larger rectangle in the plane.

     These coordinates serve to position labels as follows.
Before the \AffixLabels{...} command type label specifications:

  \SetLabels
   (<X-coordinate>*<Y-coordinate>) <first label> \\
   .
   .
   .
   (<X-coordinate>*<Y-coordinate>)  <last label> \\
  \endSetLabels

Each row specifies one label and is terminated by \\.  In each
row, the position indicator comes first; it is written as a
standard cartesian point except that the X- and Y- coordinates
are separated by * rather than a comma because TeX allows a
comma as decimal point. There are no dimension units to specify
as the unit is the grid itself.

     By default, this cartesian point specifies where the middle
of the baseline of the label will be located.  However if you precede
the point by \L [or \R] the left [or right] edge of the baseline will
be located there. Similarly you may also precede the point by \T, \E,
or \B to vertically align the top equator or bottom of the label box
at the specified point.  This gives nine standard positions of
the label with respect to the insertion point --- corresponding to
the eight principle points of the compas and the center

                     \L\T     \T      \R\T

                     \L\E     \E      \R\E

                     \L\B     \B      \R\B

But this neglects the default "baseline" level of TeX,
giving potentially three more positions

                     \L    <no tag>   \R

For text, the baseline level is often the preferred. Its relation to
the others is variable. It will often coincide with the bottom level,
as happens for "X".  But it is often distinct, as for "g", in which
case you have in all 12 distinct positions rather than 9.

     It is convenient to think of this specification of label
position as attaching the label by a thumb-tack to the coordinate
grid. There are up to twelve positions of the thumb-tack on the
label, while the position of the thumb-tack on the coordinate grid is
arbitrary.  Normally, one choses the position of the thumb-tack on
the label to be the one that is the closest to the item being
labeled.  There are good reasons for this "rule of thumb":

   (a)  It facilitates correct positioning at first try.

   (b)  If the scale of the figure must be altered after labels
have been affixed, the labels have a good chance of remaining well
positioned.

   (c)  The visible grid need not extend beyond the "bounding box"
for the figure, because the best preferred position is always
(at least almost) within the bounding box .

The second reason is particularly important. Indeed it often
happens that scale has to be altered after labelling begins, in
order to either provide space for the labels, or to adjust
proportions between the labels and the figure.  (The size of labels
is unaffected by scaling.)

     Here is an artificial but self-contained test which uses
TeX rules to make a graphics object.

TEST

    Do not skip this!

 \def\FrameIt#1{\hbox{\vrule$\vcenter {\hrule\kern3pt%
             \hbox {\kern3pt #1\kern3pt}%
               \kern3pt\hrule}$\relax\vrule}}

 \def\Caption#1#2{\FrameIt{%
       \vtop {\hsize=#1\relax \parindent=0pt
         \leftskip=0pt \rightskip=0pt plus15pt
         \parfillskip=0pt
         \lineskip=1pt\baselineskip=0pt
         #2}}}

 \def\FirstQuadrant{\hbox to 100pt{\vrule\vbox to 100pt{%
        \hbox to 100pt{\hfil}\vfil\hrule}\hss}}

  \SetLabels
    \R(.5*.2) $\zeta\,\cdot$\\
    (.9*-.10) $\xi$\\
    \R(-.03*.9) $\eta$\\
    \T(.5*.9) \Caption{70pt}{%
          \it The norm of
          $g(\xi+i\eta)$ is indicated on
          contours of this invisible surface.}\\
  \endSetLabels

  \AffixLabels{\FirstQuadrant} \end

  Note that the coordinates to use for labels are indicated on the
edges of the grid (when visible) corresponding to the conventional
x- and y- axes of the Cartesian plane. By default the grid is
1-by-1. However, by the command \Edges{100}, you can change this
to 100-by-100 and many users find this alternative most
convenient. Place the command \Edges{...} in your style file (or
header) since its effect is is global. Other possible edge values
are 10 and 1000.

  If you use the command \Edges{...} at all, do so with care.  For
if you accidentally delete an \Edges{...} command your labels will
abruptly be badly misplaced and may logically but mysteriously
generate "dimension too big" errors under TeX and "off page" errors
under your driver.  

  You can dictate the edgescale for an individual figure by giving
the scale in brackets immediately after \SetLabels.  Thus, to
import into an article using say \Edge{100} a figure labelled using
another edgescale, say the original 1-by-1 default, you can use
\SetLabels[1]...\endSetLabels.

GETTING IT DOWN PAT

     Complicated labeling deserves the same respect as
complicated mathematics.  Do not expect it to come out perfect the
first time!  What is needed in either case is a mechanism to
repeatedly typeset troublesome pieces.

     One mechanism is always available.  One does complicated
labelling in a separate "test" file involving just the figure being
labelled;  a texpert will know how to \dump TeX's current state as
a temporary format that restarts rapidly at each retry.  Usually,
one then pastes the completed labelled figure back into the main
TeX file, but, of course, one can also \input it as an auxiliary
file.

     If you do not have a TeXpert at handy, here is a first
approximation to an efficient setup. By deletions reduce a copy
of your article to just a few lines before and after the figure.
Now label the figure, and finally, copy and paste the labelled
figure to the original article. Then copy the next figure to label
into this testbed and repeat. The TeXpert can improve the  speed
at which TeX starts up, by compiling a format specifically for
your article; just one caution: best NOT include in the format
ephemeral details of setup like \Set<mydriver>ArtSpecials (from
boxedeps.tex because this reads  figure dimensions which you may
change during your work session.

     An improved mechanism to repeatedly typeset troublesome
pieces is now available on the Macintosh; it is called LinoTeX;
see the same ftp sources.  It could be set up on many types
of computer.

     Before using labelfig.tex to attach labels to a graphics
object inserted using boxedeps.tex or BoxedArt.tex, make it a
firm rule to carefully adjust the bounding box using the trimming
commands of these packages, and also at least tentatively scale
and position the object. Beware of changing the grid inadvertently
after the labels have been positioned.  For example, correcting
the bounding box of a PostScript graphics object can foul up the
labels by changing the coordinate grid to which the labels are
attached. This is particularly true for the trimming  commands of
boxedeps.tex and BoxedArt.tex. However, as noted already, change
of scale is much less disruptive, and modest adjustments should be
well tolerated.

     Sometimes the labels protrude so far from the bounding box
of a figure that the figure has to be repositioned.  Best do this
by ad hoc spacing, say using \hglue and \vglue; altering the
bounding box would create a vicious circle.

     Remember that you are responsible for preventing labels
from overlapping. You are responsible for all label typography
including size and style. A label is really just about anything
that can be put in a TeX box. Note that spaces at the beginning
and end of labels will normally be suppressed; if you really want
them you must protect them with TeX braces.

     This package temporarily sets the \mathsurround parameter
of TeX to zero  while the labels are being affixed. This is done
because nonzero \mathsurround space would influence the position
of left and right aligned labels; then, when a texpert or printer
modifies mathsurround, diagram labeling might be disastrously
altered. There is a small price to pay involving labels that are
formatted as caption boxes including mathematics: you  may want or
need to specify an explicit mathsurround space within the caption
box; it will not influence anything outside.

     Those hostile to the use of * as separator between
the X and Y coordinates of label insertion points, are free to
impose another using \SetXYSeparator{<the new separator>}.  
Americans may prefer "," to "*" since they never use a 
comma as a decimal point; on the other hand, * may be more visible.

APPENDIX (I)  MERGING labelfig.tex LABELS INTO AN EPSF GRAPHICS OBJECT.

     As promised in the introduction, here is a recipe useful for
publishers. It works at least on Macintosh and at least for vectorized
graphics and Adobe type1 fonts.  (There is surely a similar recipe for
PCs under MSWindows.)

 (a)  Use boxedeps.tex utility to integrate the figure given by the eps
file, "x.eps" say, with a visible frame around it.  See
\ShowDisplacementBoxes command in boxedeps.tex.  To get precise results
automatically it is important to use the \Trim... commands of
boxedeps.tex making the "DisplacementBox" neatly fit the figure.

 (b)  Use the TeX printer driver and LaserWriter (versions >= 8.1.1) to
export to an EPSF the DVI page containing the integrated, labelled
figure. You now have an EPS file  "xx.eps"  that contains too much, and at
the wrong scale, and at wrong position.

 (c)  Convert the EPSF to an Adode Illustrator format EPSF using
the shareware utility called epsConvert by Sam Weiss
1993-- (currently $25).

 (d)  In Illustrator (or a compatible program), group the labels and the
"DisplacementBox"; copy them to the clipboard and paste them into "x.ps".
This step requires that all the label fonts be "visible to the Macintosh.

 (e)  Translate and scale the pasted group consisting of the labels plus
the "DisplacementBox" so as to make the "DisplacementBox" the bounding
box of (labelless) figure represented by "x.eps".  At this point the
labels will be correctly placed on the figure "x.eps".

 (f)  Ungroup and delete the "DisplacementBox".  The result is the
desired single EPS file, "x+.eps" say, It contains the original figure
plus its labels.  

     Using grouping and ungrouping appropriately in "x+.eps", a
publisher's staff can very efficiently improve label positions etc.

APPENDIX II)  SOME EXOTIC APPLICATIONS

     The grid of labelfig.tex is analogous to a light-table in
classical page makeup with wax or latex glue.  In principle, you
can use it to compose any page from its indivisible parts.  This
even has some of the artisanal charm of classical paste-up
provided you have a fast screen preview to make the process
"interactive".

     In practice labelfig.tex is a tool for nonstandard jobs.
Here are a few going beyond the labelling already discussed.

(I)  GRAPHICS INTEGRATION.

     This is accomplished by treating the imported graphics
objects as labels.  The underlying graphics object is then
typically an empty  \vbox to <dimension>{\vfill} in a TeX
\midinsert...\endinsert construction.  A label line
might be of the form

   (.1*.1) \\

The exact form of the special command varies from driver to
driver.  However, in the case of encapsulated PostScript graphics
(EPSF norm), by relying on boxedeps.tex, one can have the
following standard syntax (independant of driver  (see
boxedeps.doc for details.
  
  (.1*.1) \BoxedEPSF{MyFigure scaled <scale in mils>}\\

This may be slow since it requires TeX to read the PostScript
file to read bounding box using many complex macros.  So you
may want to try

  (.1*.1) \EPSFSpecial{MyFigure}{<scale in mils>}\\

which is fast and driver independant, but it squashes the
bounding box, normally to its lower left corner.

     Similarly for graphics of the Macintosh PICT norm ---
using BoxedArt.tex (same sources) in place of boxedeps.tex.

     This approach to integration is to be recommended when
one is assembling a composite graphics object.

 (II)  COMMUTATIVE DIAGRAM ENHANCEMENT

     Commutative diagrams or arrays of mathematical objects
connected by arrows of various sorts are common in mathematics.
The mathematical objects require the use of TeX.  Recently TeX
acquired a good collection of arrows of all slopes --- that of
LamSTeX --- plus pwerful macros to build the diagrams.

     However, even the LamSTeX collection is often
inadequate; it lacks for example double shafted arrows, dotted
arrows and curved arrows. Fortunately it is possible to produce
such arrows on an individual basis using sophisticated graphics
programs such as Illustrator and AldusFreehand (both serving
the EPSF norm) or using Metafont (with its public domain norm).
Since the creation of each new arrow is a work of love, you
probably want to limit the number of arrows by using LamSTeX
for most arrows. The 40K commutative diagram module of LamSTeX
has been adapted to work with AmSTeX and a copy may be posted
with LabelFig and related files. Unfortunately no one has yet
offered a version that works with Plain TeX or LaTeX.

       Suffice it here to say that when the exotic arrow has
been somehow imported into TeX, labelfig.tex treats it as a
label that one affixes to the commutative diagram.  Two other
steps will be treated in separate notes, namely the matter of
extracting the dimension specifications for the arrow and the
construction of the arrow --- for these steps are far from
unique and often depend intimately on your computer environment. 
Notes for the Macintosh-Textures-Illustrator combination are
found in the file ExoticArrows.doc.

 (III) NESTING 

Ingenuity pays off in exploiting labelfig.tex. One can
mix graphics and typography quite freely.  labelfig.tex is good
for freeform or overlapping arrangements, while boxedeps.tex (or
BoxedArt.tex) is best for regimented non-overlapping
arrangements --- and the two can be combined.

     The default behavior of labelfig.tex is not ideal 
for nesting objects, because to prevent trouble for beginners
the register for labels is globally cleared when \AffixLabels
concludes.  But there are switches available

      \LabelsGlobal      \LabelsLocal

which change this.  To understand this, extend the above test 
by something like:

 \LabelsLocal

 \SetLabels
    (.5*.5) AAA\\
 \endSetLabels

 {
 \SetLabels
    (.5*.5) ZZZ\\
 \endSetLabels
   \AffixLabels{\FirstQuadrant}
 }

   \AffixLabels{\FirstQuadrant}

     There are however potential pitfalls.  Neither
labelfig.tex nor boxedeps.tex has been tested under extreme
conditions. Problems may occur if their procedures are
indiscriminately nested. For boxedeps.tex (not labelfig.tex)
there is a precise cause for worry, namely many of its
variables are "global", which means that TeX braces will not
provide the protection one might expect.

COMMAND SUMMARY FOR labelfig.tex

  Here [...] means optional (one or zero)
       [...]* means any number of such constructs

  \SetLabels
    [[<P>](<X><Sep><Y>) <label> \\]*
  \endSetLabels
  \ShowGrid  
  \AffixLabels{<the figure>}

   --- <P> is tack position, one of eleven or empty
              order irrelevant

                   \L\T      \T      \R\T

                   \L\E      \E      \R\E

                     \L               \R

                   \L\B      \B      \R\B

   --- (<X><Sep><Y>) insertion point;
  <Sep> is separator, = * by default;
  \SetXYSeparator{<Sep>} changes it.
   <X> and <Y> are real numbers

  --- <label> a label to attach 

  --- <the figure> the figure to label 

  \GlobalLabels (default)     
  \LocalLabels  setting for nested constructs.

 \Grids makes ALL grids appear; \HideGrid then makes just next disappear.
 \noGrids returns to default.  The commands are always global.

 \GridLineWidth{<dimension>} adjusts width of grid lines. Default is very
small, to give "hairline" effect. If your grid lines are missing try
setting \GridLineWidth{1pt}.

 \Edges#1 globally changes the edge size of all grids to the numerical 
value #1, which must be 1, 10, 100, or 1000.  The default is 1.

VERSION HISTORY.
 --- Jan 1993: \Edges#1 and [??] option after \SetLabels
 --- July 1992: \Grids, \noGrids, \HideGrid;
       Gridlines become hairlines; \GridLineWidth{<dimension>}.
 --- Oct 1991, Jan 1992: \SetXYSeparator{<Sep>},  \LabelsGlobal,
       \LabelsLocal.
 --- July 1991: first release

Address for bugs and other feedback:

        Raymond S\'eroul
        IREM and Lab. de Typographie Informatise
        Univ. Rene Descartes
        Strasbourg

    Tel 33-88-41-63-45
    Email:  A18645@FRCCSC21.BITNET

        Laurent Siebenmann
        Mathematique, Bat. 425,
        Univ de Paris-Sud,
        91405-Orsay,
        France

    Tel 33-1-6941-7949; 
    Email: lcs@topo.math.u-psud.fr  

\usepackage{multirow} 
\usepackage{array}
\newcolumntype{M}[1]{>{\centering}m{#1}}

\newcommand{\margin}[1]{\textcolor{magenta}{*}\marginpar{ \vskip -1cm \textcolor{magenta} {\it #1 }  }}
\newcommand{\note}[2]{ \hskip 2cm  \textcolor{blue}{\large \bf #1 :}   \vline\,\vline \hskip 0.5 cm \parbox{10 cm}{ #2}  }

\renewcommand{\margin}[1]{}
\renewcommand{\note}[2]{}{}

\numberwithin{equation}{section}
\numberwithin{figure}{section}
\newtheorem{theorem}{Theorem}
\numberwithin{theorem}{section}
\newtheorem{corollary}[theorem]{Corollary}
\newtheorem{lemma}[theorem]{Lemma}
\newtheorem{proposition}[theorem]{Proposition}
\newtheorem{conjecture}[theorem]{Conjecture}

\newtheorem{question}[theorem]{Question}

\theoremstyle{remark}\newtheorem{definition}[theorem]{Definition}   
\theoremstyle{remark}\newtheorem{remark}[theorem]{Remark}
\def\eqref#1{(\ref{#1})}
\let\qqed=\qed
\def\SLE{\mathrm{SLE}}
\def\gr{\mathrm{gr}}

\def\bl{\begin{lemma}}
\def\el{\end{lemma}}
\def\bth{\begin{theorem}}
\def\eth{\end{theorem}}
\def\bc{\begin{corollary}}
\def\ec{\end{corollary}}
\def\bcj{\begin{conjecture}}
\def\ecj{\end{conjecture}}
\def\bpr{\begin{proposition}}
\def\epr{\end{proposition}}
\def\bde{\begin{definition}}
\def\ede{\end{definition}}
\newcommand{\be}{\begin{eqnarray}}
\newcommand{\ee}{\end{eqnarray}}
\newcommand{\bes}{\begin{eqnarray*}}
\newcommand{\ees}{\end{eqnarray*}}

\def\nn{\nonumber}
\def\ni{\noindent}
\def\bi{\begin{itemize}}
\def\ei{\end{itemize}}
\def\bnum{\begin{enumerate}}
\def\enum{\end{enumerate}}

\def\QED{\qqed\medskip}
\let\qed=\QED

\newcommand{\R}{\mathbb{R}}

\newcommand{\DD}{\mathcal{D}}
\newcommand{\Q}{\mathbb{Q}}
\newcommand{\C}{\mathbb{C}}
\newcommand{\Z}{\mathbb{Z}}
\newcommand{\N}{\mathbb{N}}
\def\Hyp{\mathbb{H}}

\def\diam{\mathrm{diam}}

\def\dist{\mathrm{dist}}

\def\Im{{\rm Im}\,}

\def\cov{\mathrm{Cov}}
\def \eps {\epsilon}
\def \P {{\bf P}}
\def\md{\mid}
\def\Bb#1#2{{\def\md{\bigm| }#1\bigl[#2\bigr]}}
\def\BB#1#2{{\def\md{\Bigm| }#1\Bigl[#2\Bigr]}}
\def\Bs#1#2{{\def\md{\mid}#1[#2]}}
\def\Pb{\Bb\P}
\def\Eb{\Bb\E}
\def\PB{\BB\P}

\def\Ps{\Bs\P}

\def \p {{\partial}}
\def \E {{\bf E}}

\def\closure{\overline}
\def\ev#1{{\mathcal{#1}}}

\def\bl{\bigl}

\def\dbox{B_\eta}

\def\1{1\hspace{-2.55 mm}{1}}

\def\A{\mathcal{A}}
\def\Cov{\mathcal{C}}
\def\O{\mathcal{O}}
\def\r{\mathfrak{r}} 
\def\k{\mathfrak{k}}

\def\Tg{\mathbb{T}} 
\def\boxup{\boxdot}
\def\Quad{Q} 
\def\QUAD{\mathcal{Q}} 
\def\HH{\mathscr{H}}  
\def\T{\mathcal{T}}  
\def\nc{\mathsf{nc}}
\def\PPP{\mathsf{PPP}}
\def\switch{\mathcal{S}}

\def\M{\mathfrak{M}}
\def\Net{\mathsf{N}}
\def\rot#1{\overset{\curvearrowright}{#1} }

\def\cadlag{{c\`adl\`ag }}
\def\Sk{\mathsf{Sk}}

\def\Piv{\mathcal{P}}
\def\<#1{\langle #1\rangle}

\def\EK{MR838085}
\def\Proho{MR0084896}

\def\Tsirelson{MR2079671}

\def\SS{SSblacknoise}
\def\SchrammSmirnovNoise{\SS}
\def\SmirnovWerner{MR1879816}

\def\CamiaNewmanFull{MR2249794}
\def\CN{\CamiaNewmanFull}

\def\BeffaraDim6{MR2078552}

\def\GPS{arXiv:0803.3750}

\def\GPSa{GPS2a}

\begin{document}
\maketitle

\begin{abstract}
We prove that near-critical percolation and dynamical percolation on the triangular lattice $\eta \Tg$ have a scaling limit as the mesh $\eta \to 0$, 
in the ``quad-crossing'' space $\HH$ of percolation configurations introduced by Schramm and Smirnov. The proof essentially proceeds by ``perturbing'' the scaling limit of the critical model, using the pivotal measures studied in our earlier paper. Markovianity and conformal covariance of these new limiting objects are also established. 
\end{abstract}

\tableofcontents

\section{Introduction}

\subsection{Motivation}
 
Percolation is a central model of statistical physics, exceptionally simple and rich at the same time.
Indeed, edge-percolation on the graph $\Z^d$ is simply defined as follows: each edge $e\in \mathbb{E}^d$ (the set of edges $e=(x,y)$ such that $\|x-y\|_2=1$) is kept with probability $p\in[0,1]$ and is removed with probability $1-p$ independently of the other edges. This way, one obtains a random configuration $\omega_p \sim \P_p$ in $\{0,1\}^{\mathbb{E}^d}$. It is well-known (see, for example, \cite{Grimmett}) that for each dimension $d\geq 2$, there is a {\bf phase transition} at some critical point $0<p_c(\Z^d)<1$. One of the main focuses of percolation theory is on the typical behaviour of percolation at and near its phase transition. This problem is still far from being understood; for example, it is a long-standing conjecture that the phase transition for percolation is continuous in $\Z^3$. 

When the percolation model is {\bf planar}, a lot more is known on the phase transition. A celebrated theorem by Kesten \cite{KestenZ2} is that for edge-percolation on $\Z^2$, $p_c(\Z^2)=1/2$. Furthermore, the corresponding phase transition is continuous (to be more precise, it falls into the class of {second-order phase transitions}): the {\bf density function}
\[
\theta_{\Z^2}(p):= \P_p [0\text{ is connected to infinity}]
\] 
is continuous on $[0,1]$.

When one deals with a statistical physics model which undergoes such a continuous phase transition, it is natural to understand the nature of its phase transition by studying the behaviour of the system near its critical point, at $p=p_c+ \Delta p$.
In the case of percolation on $\Z^2$, it is proved in \cite{KestenZhang} that there exists an $\eps>0$ so that, as $p\to p_c(\Z^2)=\frac 1 2$, one has 
\begin{align}\label{e.K1}
\theta_{\Z^2}(p) \geq (p-p_c)^{1-\eps} 1_{p>p_c}\,.
\end{align}

In order to study such systems near their critical point, it is very useful to introduce the concept of {\bf correlation length} $L(p)$ for  $p\approx p_c$. Roughly speaking, $p\mapsto L(p)$ is defined in such a way that, for $p\neq p_c$, the system ``looks critical'' on scales smaller than $L(p)$, while the non-critical behaviour becomes ``striking'' above this scale $L(p)$. See for example \cite{PC,Nolin,Kesten} for a precise definition and discussion of $L(p)$ in the case of percolation (see also Subsection~\ref{ss.CL} in this paper). Kesten proved in \cite{Kesten} that the correlation length $L(p)$ in planar percolation is given in terms of the probability of the {\bf alternating four-arm} event {\bf at the critical point}: 
\begin{align}\label{e.K2}
L(p) \asymp \inf \left\{ R\geq 1,\, \text{s.t. } R^2 \alpha_{4}(R) \geq \frac 1 {|p-p_c|} \right\}\,,  
\end{align}
where $\alpha_4(R)=\alpha_{4,p_c}(R)$ stands for the probability of the alternating four-arm event up to radius $R$ at the critical point of the planar percolation model considered. See for example \cite{PC,Buzios}. In particular, the scale whose aim is to separate critical from non-critical effects at $p\approx p_c$ can be computed just by studying the {\bf critical geometry} of the system (here, the quantity $\alpha_4(R)$).  A detailed study of the near-critical system below its correlation length was given in \cite{BCKS}. Furthermore, Kesten's notion of correlation length enabled him to prove in \cite{Kesten} that, as $p>p_c$ tends to $p_c$, one has 
\begin{align}
\theta(p) &\asymp \P_p[ 0 \text{ is connected to } \p B(0,L(p))]  \nn \\
& \asymp  \P_{p_c}[ 0 \text{ is connected to } \p B(0,L(p))] \nn \\
&:= \alpha_1(L(p))\,. \label{e.K3}
\end{align}
In particular, the {density} $\theta(p)$ of the infinite cluster near its critical point can be evaluated just using quantities which describe the critical system: $\alpha_1(R)$ and $\alpha_4(R)$.

Such critical quantities are not yet fully understood on $\Z^2$ at $p_c(\Z^2)=1/2$ (this is why the behaviour for $\theta_{\Z^2}(p)$ given in~\eqref{e.K1} remains unprecise), but there is one planar percolation model for which such quantities can be precisely estimated: {\bf site percolation on the triangular grid} $\Tg$, for which one also has $p_c(\Tg)=\frac 1 2$. (See \cite{PC} for self-contained lecture notes on critical site percolation on $\Tg$). Indeed, a celebrated theorem by Smirnov  \cite{Smirnov} states that if one considers critical site percolation on $\eta \Tg$, the triangular grid with small mesh $\eta>0$, and lets $\eta\to 0$, then 
the limiting probabilities of crossing events are {\bf conformally invariant}. This conformal invariance enables one to rely on the so-called Stochastic Loewner Evolution (or SLE) processes introduced by the third author in \cite{Schramm}, which then can be used to obtain the following estimates:
\bi
\item[(i)] $\alpha_1(R)= R^{-5/48+o(1)}$ obtained in \cite{LSW}\,,
\item[(ii)] $\alpha_4(R)=R^{-5/4+o(1)}$ obtained in \cite{\SmirnovWerner}\,,
\item[(iii)] $L(p) = \left|\frac{1}{p-p_c}\right|^{4/3+o(1)}$ obtained in \cite{\SmirnovWerner}\,,
\item[(iv)] $\theta_\Tg(p) = (p-p_c)^{5/36+o(1)}1_{p>p_c}$ obtained in \cite{\SmirnovWerner}\,,
\ei
where the $o(1)$ are understood as $R\to \infty$ and $p\to p_c$, respectively. It is straightforward to check that items (iii) and (iv) follow from items (i), (ii) together with equations~\eqref{e.K2} and~\eqref{e.K3}.

Items (iii) and (iv) are exactly the type of estimates which describe the so-called {\bf near-critical} behaviour of a statistical physics model. To give another well-known example in this vein: for the Ising model on the lattice $\Z^2$, it is known since Onsager \cite{Onsager} that $\theta(\beta):=\P_{\beta}^+\bigl[ \sigma_0 = +\bigr] \asymp (\beta-\beta_c)^{1/8}1_{\beta>\beta_c}$, which is a direct analog of Item~(iv) if one interprets $\theta(\beta)$ in terms of its associated FK percolation ($q=2$). Also the correlation length $\beta \mapsto L(\beta)$ defined in the spirit of Kesten's paper \cite{Kesten} is known to be of order $\frac 1 {|\beta-\beta_c|}$, see \cite{DGP}. 

The question we wish to address in this paper is the following one: how does the system look below its correlation length $L(p)$? 
More precisely, let us redefine $L(p)$ to be exactly the above quantity $\inf \left\{ R\geq 1,\, \text{s.t. } R^2 \alpha_{4}(R) \geq \frac 1 {|p-p_c|} \right\}$; of course, the exact choice of the constant factor in $1/|p-p_c|$ is arbitrary here. Then, for each $p\neq p_c$, one may consider the percolation configuration $\omega_p$ in the domain $[-L(p),L(p)]^2$ and rescale it to fit in the compact window $[-1,1]^2$ (one thus obtains a percolation configuration on the lattice $L(p)^{-1}\Tg$ with parameter $p\neq p_c$). A natural question is to prove that as $p\neq p_c$ tends to $p_c$, one obtains a nontrivial scaling limit: the {\bf near-critical scaling limit}. Prior to this paper, subsequential scaling limits were known to exist. As such, the status for near-critical percolation was the same as for critical percolation on $\Z^2$, where subsequential scaling limits (in the space $\HH$ to be defined in Section~\ref{s.HH}) are also known to exist. The existence of such subsequential scaling limits is basically a consequence of the RSW theorem. Obtaining a (unique) scaling limit is in general a much harder task (for example, proved by \cite{Smirnov, MR2249794} for critical percolation on $\Tg$), and this is one of the main contributions of this paper:  we prove the existence of the scaling limit (again in the space $\HH$) for near-critical site percolation on the triangular grid $\Tg$ below its correlation length. We will state a proper result later; in particular, Corollary~\ref{c.BCL} says that one obtains two different scaling limits in the above setting as $p\to p_c$: $\omega_\infty^+$ and $\omega_\infty^-$ depending whether $p>p_c$ or not. One might think at this point that these near-critical scaling limits should be identical to the critical scaling limit $\omega_\infty$, since the correlation length $L(p)$ was defined in such a way that the system ``looks'' critical below $L(p)$. But, as it is shown in \cite{NolinWerner}, although any subsequential scaling limit of near-critical percolation indeed ``resembles'' $\omega_\infty$ (the interfaces have the same Hausdorff dimension $7/4$ for example), it is nevertheless singular w.r.t. $\omega_\infty$. (See also Subsection~\ref{ss.sing}). 
  
Now, in order to describe our results in more detail, we introduce {near-critical percolation} in a slightly different manner, via the so-called monotone couplings.

%
%
%

\subsection{Near-critical coupling}

It is a classical fact that one can couple site-percolation configurations $\{\omega_p\}_{p\in [0,1]}$ on $\Tg$ in such a way that for any $p_1<p_2$, one has $\omega_{p_1}\leq \omega_{p_2}$ with the obvious partial order on $\{0,1\}^\Tg$. 
One way to achieve such a coupling is to sample independently on each site $x\in \Tg$ a uniform random variable $u_x\sim \mathcal{U}([0,1])$, and then define 
$\omega_p(x):= 1_{u_x\leq p}$. 

\begin{remark}
Note that defined this way, the process $p\in [0,1] \mapsto \omega_p$ is a.s.~a \cadlag path in $\{0,1\}^\Tg$ endowed with the product topology. This remark already hints why we will later consider the Skorohod space on the Schramm-Smirnov space $\HH$. 
\end{remark}

One would like to rescale this monotone coupling on a grid $\eta \Tg$ with small mesh $\eta>0$ in order to obtain an interesting limiting coupling. If one just rescales space without rescaling the parameter $p$ around $p_c$, it is easy to see that the monotone coupling $\{\eta \omega_p\}_{p\in[0,1]}$ on $\eta \Tg$ converges as a coupling to a trivial limit except for the slice corresponding to $p=p_c$ where one obtains the Schramm-Smirnov scaling limit of critical percolation (see Section \ref{s.HH}). Thus, one should look for a monotone coupling $\{\omega_\eta^\nc(\lambda)\}_{\lambda\in\R}$, where $\omega_\eta^\nc(\lambda)=\eta\omega_p$ with $p=p_c + \lambda r(\eta)$,  and where the zooming factor $r(\eta)$ goes to zero with the mesh. On the other hand, if it tends to zero too quickly, it is easy to check that $\{\omega_\eta^\nc(\lambda)\}_\lambda$
will also converge to a trivial coupling where all the slices are identical to the $\lambda=0$ slice, i.e., the Schramm-Smirnov limit $\omega_\infty$.  From the work of Kesten \cite{Kesten} (see also \cite{NolinWerner} and \cite{\GPSa}), it is natural to fix once and for all the zooming factor to be:
\begin{align}\label{e.r}
r(\eta):= \eta^2 \alpha_4^\eta(\eta,1)^{-1}\,,
\end{align}
where $\alpha_4^\eta(r,R)$ stands for the probability of the {\bf alternating four-arm} event for critical percolation on $\eta \Tg$ from radius $r$ to $R$. See also \cite{\GPSa} where the same notation is used throughout. One disadvantage of the present definition of $\omega_\eta^\nc(\lambda)$ is that $\lambda \in \R \mapsto \omega_\eta^\nc(\lambda)$ is a time-{\it in}homogeneous Markov process. To overcome this, we change sightly the definition of $\omega_\eta^\nc(\lambda)$ as follows:

\begin{definition}\label{d.NCP}
In the rest of this paper, the {\bf near-critical coupling} $(\omega_\eta^\nc(\lambda))_{\lambda\in \R}$ will denote the following process: 
\bi
\item[(i)] Sample $\omega_\eta^\nc(\lambda=0)$ according to $\P_\eta$, the  law of critical percolation on $\eta \Tg$. We will sometimes represent this as a black-and-white colouring of the faces of the dual hexagonal lattice. 
\item[(ii)] As $\lambda$ increases, closed sites (white hexagons) switch to open (black) at exponential rate $r(\eta)$.
\item[(iii)] As $\lambda$ decreases, black hexagons switch to white at rate $r(\eta)$.
\ei
As such, for any $\lambda\in \R$, the near-critical percolation $\omega_\eta^\nc(\lambda)$ corresponds exactly to a percolation configuration on $\eta \Tg$ with parameter
\[
\begin{cases}
p=1-(1-p_c)e^{-\lambda\, r(\eta) } \quad \text{if } \lambda\geq 0 \\
p=p_c \, e^{-|\lambda|\, r(\eta) } \quad \text{if } \lambda< 0\,,  
\end{cases}
\]
thus making the link with our initial definition of $\omega_\eta^\nc(\lambda)$.

Let us note that the symmetry in (ii) and (iii) between increasing and decreasing $\lambda$ values is natural and leads to a time-homogeneous Markov process only because we have $p_c=1/2$ now. For general $p_c$, the correct definition would have different rates in (ii) and (iii), with a ratio of $p_c/(1-p_c)$.
\end{definition}

In this setting of monotone couplings, our goal in this paper is to prove the convergence of the monotone family $\{\omega_\eta^\nc(\lambda)\}_{\lambda\in \R}$  as $\eta\to 0$ to a limiting coupling $\{ \omega_\infty^\nc(\lambda) \}_{\lambda\in \R}$. See Theorems \ref{th.main1} and \ref{th.main2} for precise statements. In some sense, this limiting object captures the birth of the infinite cluster seen from the scaling limit. (Indeed, as we shall see in Theorem \ref{th.CL}, as soon as $\lambda>0$, there is a.s.~an infinite cluster in $\omega_\infty^\nc(\lambda)$).  

\subsection{Rescaled dynamical percolation} 

In \cite{HPS}, the authors introduced a natural reversible dynamics on percolation configurations called {\bf dynamical percolation}. This dynamics is very simple: each site (or bond in the case of bond-percolation) is updated independently of the other sites at rate one, according to the Bernoulli law $p \delta_1 + (1-p)\delta_0$. As such, the law $\P_p$ on $\{0,1\}^\Tg$ is invariant under the dynamics. Several intriguing properties like existence of exceptional times at $p=p_c$ where infinite clusters suddenly arise have been proved lately; see \cite{SchrammSteif, \GPS, HammPS}. 
It is a natural desire to define a similar dynamics for the Schramm-Smirnov scaling limit of critical percolation $\omega_\infty\sim \P_\infty$, i.e., a process $t\mapsto \omega_\infty(t)$ which would preserve the measure $\P_\infty$ of Section~\ref{s.HH}. Defining such a process is a much more difficult task and a natural approach is to build this process as the scaling limit of dynamical percolation on $\eta \Tg$ properly rescaled (in space as well as in time). Using similar arguments as for near-critical percolation (see the detailed discussion in \cite{\GPSa}), the right way of rescaling dynamical percolation is as follows:

\begin{definition}
\label{d.DP}
In the rest of this paper, for each $\eta>0$, the rescaled {\bf dynamical percolation} $t\mapsto \omega_\eta(t)$ will correspond to the following process: 
\bi
\item[(i)] Sample the initial configuration $\omega_\eta(t=0)$ according to $\P_\eta$, the law of critical site percolation on $\eta \Tg$. 
\item[(ii)] As time $t$ increases, each hexagon is updated independently of the other sites at exponential rate $r(\eta)$ (defined in equation~\eqref{e.r}). When an exponential clock rings, the state of the corresponding hexagon becomes either white with probability 1/2 or black with probability 1/2. (Hence the measure $\P_\eta$ is invariant). 
\ei
\end{definition}

Note the similarity between the processes $\lambda\mapsto \omega_\eta^\nc(\lambda)$ and $t\mapsto \omega_\eta(t)$. In particular, the second main goal of this paper is to prove that the rescaled dynamical percolation process $t\mapsto \omega_\eta(t)$, seen as a \cadlag process in the Schramm-Smirnov space $\HH$ has a scaling limit as the mesh $\eta\to 0$. See Theorem \ref{th.main3}. This answers Question~5.3 in~\cite{ICM}. 

\subsection{Links to the existing literature}\label{ss.links}

In this subsection, we wish to list a few related works in the literature.
\bi
\item As mentioned earlier, the near-critical coupling $\omega_\eta^\nc(\lambda)$ has been studied in \cite{NolinWerner}. They do not prove a scaling limit result for $\omega_\eta^\nc(\lambda)$ as $\eta\to 0$, but they show that any subsequential scaling limits (with $\lambda\neq 0$ fixed) for the interfaces $\gamma_\eta(\lambda)$ of near-critical percolation $\omega_\eta^\nc(\lambda)$ are {\bf singular} w.r.t. the $\SLE_6$ measure. This result was very inspiring to us at the early stage of this work since it revealed that if a near-critical scaling limit $\omega_\infty^\nc(\lambda)$ existed, then it would lead to a very different (and thus very interesting) object compared to the Schramm-Smirnov scaling limit of critical percolation $\omega_\infty\sim \P_\infty$ (which is defined in Section~\ref{s.HH}).

\item In \cite{CFN}, the authors suggested a conceptual framework to construct a candidate for the scaling limit of $\omega_\eta^\nc(\lambda)$ (their rescaling procedure is slightly different from our Definition \ref{d.NCP} as it does not take into account possible logarithmic corrections in quantities like $\alpha_4^\eta(\eta,1)$). In this work, we thus answer the two main problems raised by \cite{CFN}. First, we prove that their framework indeed leads to an object $\omega_\infty^\nc(\lambda)$ and second, we prove that this object is indeed the scaling limit of $\omega_\eta^\nc(\lambda)$ as $\eta\to 0$. 

\item In the announcement \cite{MS}, the authors discuss what should be the scaling limit of interfaces of near-critical models. They identify a family of processes called the {\bf massive SLEs} which are the candidates for such near-critical scaling limits. However, they have concrete candidates only for the special cases $\kappa=2,3,4,16/3,8$, where the models are related to harmonic functions directly or through fermionic observables.  For the case of percolation, we make a conjecture for massive $\SLE_6$'s in Subsection~\ref{ss.off}. Let us note here that massive SLE's are expected to be absolutely continuous w.r.t.~their standard version for $\kappa\leq 4$, and singular for $\kappa>4$.

\item In \cite{CGN2}, a similar kind of near-critical scaling limit is considered: namely the {\bf Ising model} on the rescaled lattice $\eta \Z^2$ at the critical inverse temperature $\beta_c$ and with exterior magnetic field $h_\eta:= h\,\eta^{15/8}$ with $h>0$ fixed. As the mesh $\eta$ tends to zero, it is proved using the limit of the magnetic field obtained in \cite{CGN1} (which also relies on \cite{CHI}) that this near-critical Ising model has a scaling limit. Obtaining such a near-critical scaling limit in that case is in some sense easier than here, since in compact domains, the above near-critical scaling limit of Ising model with vanishing magnetic field happens to be {\bf absolutely continuous} w.r.t.~the critical scaling limit (as opposed to what happens with near-critical percolation, see Subsection \ref{ss.sing}). In particular, in order to obtain the above existence of the near-critical scaling limit, it is enough to identify its Radon-Nikodym derivative w.r.t.~the critical measure.

\item It is well-known that there is a phase-transition at $p=1/n$ for the {\bf Erd{\H o}s-R\'enyi random graphs} $G(n,p)$. Similarly to the above case of planar percolation, it is a natural problem to study the geometry of these random graphs near the transition $p_c=1/n$. It turns out in this case that the meaningful rescaling is as follows: one considers near-critical random graphs with intensity $p=1/n + \lambda/ n^{4/3},\, \lambda\in \R$. Using notations similar to ours, if $R_n(\lambda)=(C_n^1(\lambda), C_n^2(\lambda),\ldots)$ denotes the sequence of clusters at $p=1/n + \lambda/ n^{4/3}$ (ordered in decreasing order of size, say), then it is proved in \cite{ABG} that as $n\to \infty$, the renormalized sequence $n^{-1/3}\, R_n(\lambda)$ converges in law to a limiting object $R_\infty(\lambda)$ for a certain topology on sequences of compact spaces which relies on the Gromov-Hausdorff distance.  This near-critical coupling $\{ R_\infty(\lambda)\}_{\lambda\in \R}$ has then been used in \cite{ABGM} in order to obtain a scaling limit as $n\to \infty$ (in the Gromov-Hausdorff sense) of the minimal spanning tree on the complete graph with $n$ vertices.  Our present paper is basically the Euclidean analog ($d=2$) of the mean-field case \cite{ABG}, and the near-critical coupling $\{ \omega_\infty^\nc(\lambda)\}_{\lambda\in \R}$ we build is used in the companion paper \cite{MSTreal}  to obtain the scaling limit of the Minimal Spanning Tree in the plane (see the report~\cite{MST} as well as ~\cite{CFN} where such a construction from the continuum had already been discussed). An important difference is  that in the mean-field case one is interested in the intrinsic metric properties (and hence works with the Gromov-Hausdorff distance between metric spaces), while in the Euclidean case one is first of all interested in how the graph is embedded in the plane.

\item In \cite{Wulff}, the author relies on our main result in his proof that the {\bf Wulff crystal} for supercritical percolation on the triangular lattice converges to a ball as $p>p_c$ tends to $p_c(\Tg)=1/2$. 

\item In \cite{Kiss} and \cite{vdBKN}, our results from the present paper and from \cite{GPS2a} are used to study Aldous' {\bf frozen percolation} for site percolation on the triangular lattice, where clusters are grown dynamically and get frozen when they reach large diameter or large volume, respectively.
\ei

\subsection{Main statements}

The first result we wish to state is that if $\lambda\in \R$ is fixed, then the near-critical percolation $\omega_\eta(\lambda)$ has a scaling limit as $\eta\to 0$. In order to state a proper theorem, one has to specify what the setup and the topology are. As it is discussed at the beginning of Section \ref{s.HH}, there are several very different manners to represent or ``encode'' what a percolation configuration is (see also the very good discussion on this in \cite{SSblacknoise}). In this paper, we shall follow the approach by the third author and Smirnov, which is explained in detail in Section~\ref{s.HH}. In this approach, each percolation configuration $\omega_\eta \in \{0,1\}^{\eta \Tg}$ corresponds to a point in the Schramm-Smirnov topological space $(\HH,\T)$ which has the advantage to be compact (see Theorem \ref{th.compact}) and Polish. From \cite{SSblacknoise} and \cite{\CN}, it follows that $\omega_\eta \sim \P_\eta$ (critical percolation on $\eta\Tg$) has a scaling limit in $(\HH,\T)$: i.e., it converges in law as $\eta\to 0$ under the topology $\T$ to a ``continuum'' percolation $\omega_\infty \sim \P_\infty$, where $\P_\infty$ is a Borel probability measure on $(\HH, \T)$. See Subsection~\ref{ss.scalinglimitSSsense}. We may now state our first main result.

\begin{theorem}\label{th.main1}
Let $\lambda\in \R$ be fixed. Then as $\eta\to 0$, the near-critical percolation $\omega_\eta^\nc(\lambda)$ converges in law (in the topological space $(\HH,\T)$) to a limiting random percolation configuration, which we will denote by $\omega_\infty^\nc(\lambda)\in \HH$. 
\end{theorem}

As pointed out earlier, the process $\lambda \in \R \mapsto \omega_\eta^\nc(\lambda)$ is a
\cadlag process in $(\HH, \T)$. One may thus wonder if it converges as $\eta\to 0$ to a limiting random \cadlag path. There is a well-known and very convenient functional setup for \cadlag paths with values in  a Polish metric spaces $(X,d)$: the Skorohod space introduced in Proposition \ref{pr.sko}. Fortunately, we know from Theorem \ref{th.compact} that the Schramm-Smirnov space $(\HH,\T)$ is metrizable. In particular, one can introduce a Skorohod space of \cadlag paths with values in $(\HH,d_\HH)$ where $d_\HH$ is some fixed distance compatible with the topology $\T$. This Skorohod space is defined in Lemma \ref{l.SKR} and is denoted by $(\Sk,d_\Sk)$. We have the following theorem:

\begin{theorem}\label{th.main2}
As the mesh $\eta\to 0$, the \cadlag process $\lambda \mapsto \omega_\eta^\nc(\lambda)$ converges in law under the topology of $d_\Sk$ to a limiting random \cadlag process $\lambda \mapsto \omega_\infty^\nc(\lambda)$. 
\end{theorem}

\begin{remark}\label{}
Due to the topology given by $d_\Sk$, it is not a priori obvious  that the slice $\omega_\infty^\nc(\lambda)$ obtained from Theorem \ref{th.main2}  is the same object as the scaling limit $\omega_\infty^\nc(\lambda)$ obtained in Theorem \ref{th.main1}. Nonetheless, it is proved in Theorem \ref{th.MARGIN} that these two objects indeed coincide. 
\end{remark}

From the above theorem, it is easy to extract the following corollary which answers our initial motivation by describing how percolation looks below its correlation length; see Subsection~\ref{ss.CL} for the proof:

\begin{corollary}\label{c.BCL}
For any $p\neq p_c$, let 
\[
L(p):= \inf \left\{ R\geq 1,\, \text{s.t. } R^2 \alpha_{4}(R) \geq \frac 1 {|p-p_c|} \right\}\,.
\]
Recall that for any $p\in[0,1]$, $\omega_p$ stands for percolation on $\Tg$ with intensity $p$. Then as $p-p_c>0$ tends to zero, $L(p)^{-1} \omega_p$ converges in law in $(\HH,d_\HH)$ to $\omega_\infty^\nc(\lambda=2)$ while as $p-p_c<0$ tends to $0$, $L(p)^{-1} \omega_p$ converges in law in $(\HH,d_\HH)$ to $\omega_\infty^\nc(\lambda=-2)$. 
\end{corollary}

We defined another \cadlag process of interest in Definition \ref{d.DP}: the rescaled dynamical percolation process $t\mapsto \omega_\eta(t)$. This process also lives in the Skorohod space $\Sk$ and we have the following scaling limit result:

\begin{theorem}\label{th.main3}
As the mesh $\eta\to 0$, rescaled dyamical percolation converges in law (in $(\Sk,d_\Sk)$) to a limiting stochastic process in $\HH$ denoted by $t\mapsto \omega_\infty(t)$. 
\end{theorem}

By construction, $t \mapsto \omega_\eta(t)$ and $\lambda\mapsto \omega_\eta^\nc(\lambda)$ are Markov processes in $\HH$. There is no reason arising from general theory that the Markov property survives at the scaling limit. 
Our strategy of proof  for Theorems \ref{th.main2} and \ref{th.main3} (see below) in fact enables us to prove the following result (see Section \ref{s.markov}).

\begin{theorem}\label{th.main4}
\ni\bi
\item The process $t\mapsto \omega_\infty(t)$ is a {\bf Markov} process which is reversible w.r.t. the measure $\P_\infty$,  the scaling limit of critical percolation.
\item The process $\lambda \mapsto \omega_\infty^\nc(\lambda)$ is a {\bf time-homogeneous (but non-reversible) Markov}  process in $(\HH,d_\HH)$. 
\ei
\end{theorem}

\begin{remark}
Thus we obtain a natural {\bf diffusion} on the Schramm-Smirnov space $\HH$. Interestingly, it can be seen that this diffusion is {\bf non-Feller}. See Remark \ref{r.feller}. 
\end{remark}

As we shall see in Section~\ref{s.CL}, the processes $\lambda \mapsto \omega_\infty^\nc(\lambda)$ and $t \mapsto \omega_\infty(t)$ are conformally covariant under the action of conformal maps. See Theorem \ref{th.cc} for a precise statement.  Roughly speaking, if $\tilde\omega_\infty(t) = \phi\cdot \omega_\infty(t)$ is the conformal mapping of a continuum dynamical percolation from a domain $D$ to a domain $\tilde D$, then the process $t\mapsto \tilde \omega_\infty(t)$ evolves very quickly (in a precise quantitative manner) in regions of $D'$ where $|\phi'|$ is large and very slowly in regions of $D'$ where $|\phi'|$ is small. This type of invariance  was conjectured in \cite{ICM}; it was even coined a ``relativistic'' invariance due to the space-time dependency. When the conformal map is a scaling $z\in \C\mapsto \alpha\cdot z\in \C$, the conformal covariance reads as follows (see Corollary \ref{c.SCALE}):

\begin{theorem}
For any scaling parameter $\alpha>0$ and any $\omega\in \HH$, we will denote by $\alpha\cdot \omega$ the image by $z\mapsto \alpha\, z$ of the configuration $\omega$. With these notations, we have the following identities in law:
\bnum
\item \[
\Bigl( \lambda \mapsto \alpha\cdot \omega_\infty^\nc(\lambda) \Bigr) \overset{(d)}{=} \Bigl( \lambda \mapsto \omega_\infty^\nc(\alpha^{-3/4} \lambda) \Bigr)
\]
\item
\[
\Bigl( t\geq 0 \mapsto \alpha\cdot \omega_\infty(t)\Bigr) \overset{(d)}{=} \Bigl( t \mapsto \omega_\infty(\alpha^{-3/4} t) \Bigr)
\]
\enum
\end{theorem}

Note that this theorem is very interesting from a {\bf renormalization group} perspective. (For background, see the monographs \cite{Cardy, ZJ}, as well as the recent \cite{FKrenorm} which sheds new light on this topic). Indeed, the mapping
$F : \HH \to \HH$ which associates to a configuration $\omega\in \HH$ the ``renormalized'' configuration $\frac 1 2 \cdot \omega\in \HH$ is a very natural renormalization transformation on $\HH$. It is easy to check that the law $\P_\infty$ is a fixed point for this transformation. The above theorem shows that the one-dimensional line given by $\{ \P_{\lambda,\infty} \}_{\lambda\in \R}$, where $\P_{\lambda,\infty}$ denotes the law of $\omega_\infty^\nc(\lambda)$, provides an {\bf unstable variety} for the transformation $\omega\in \HH \mapsto \frac 1 2\cdot \omega\in \HH$. 

Finally, in Sections \ref{s.CL}, \ref{s.ET} and \ref{s.MIS}, we establish some interesting properties of the scaling limits of near-critical and dynamical processes as well as some related models like gradient percolation. Here is a concise list summarizing these results.
\bnum
\item In Theorem \ref{th.CL}, we show that if $\lambda>0$, then $\omega_\infty^\nc(\lambda)$ a.s.~has an infinite cluster and for each $\lambda > 0$, one can define a natural notion of {\bf correlation length} $L(\lambda)$, which is shown to satisfy $L(\lambda)=c \lambda^{-4/3}$. 
\item In Section \ref{s.ET}, we prove that the dynamics $t\mapsto \omega_\infty(t)$ is {\bf noise sensitive} and that there a.s.~exist exceptional times with an infinite cluster. This extends the results from \cite{\GPS} to the scaling limit of dynamical percolation. We wish to point out that this property is the only link with \cite{\GPS} throughout the whole paper. (In other words, all the sections besides Section \ref{s.ET} are completely independent of \cite{\GPS}).  
\item In Subsection \ref{ss.gradient}, we prove that the model of {\bf gradient percolation} considered in \cite{gradient} has a scaling limit, denoted by $\omega_\infty^\gr$. 
\item In Subsection \ref{ss.sing}, we prove that if $\lambda\neq 0$, then $\omega_\infty^\nc(\lambda)\sim \P_{\lambda,\infty}$ is singular w.r.t. $\omega_\infty\sim \P_\infty$, confirming in a weaker sense the main singularity result from \cite{NolinWerner}.
\enum

\subsection{Strategy of proof}
 
Let us end this introduction by explaining what will  be our strategy to build the processes  $\lambda \mapsto \omega_\infty^\nc(\lambda)$ and $t \mapsto \omega_\infty(t)$ and to show that they are the scaling limits of their discrete $\eta$-analogs. We will focus on the near-critical case, the dynamical case being handled similarly. Also, before giving a rather detailed strategy, let us start with a very rough one: in order to build a random \cadlag process $\lambda\mapsto \omega_\infty^\nc(\lambda)$, our strategy will be to start with the critical slice, i.e., the Schramm-Smirnov limit $\omega_\infty=\omega_\infty(\lambda=0)\sim \P_\infty$ and then as $\lambda$ will increase, we will randomly add in an appropriate manner some ``infinitesimal'' mass to $\omega_\infty(0)$. In the other direction, as $\lambda$ will decrease below 0, we will randomly remove some ``infinitesimal'' mass to $\omega_\infty(0)$. Before passing to the limit, when one still has discrete configurations $\omega_\eta$ on a lattice $\eta \Tg$, this procedure of adding or removing mass is straightforward and is given by the Poisson point process induced by Definition \ref{d.NCP}. At the scaling limit, there are no sites or hexagons any more, hence one has to find a proper way to perturb the slice $\omega_\infty(0)$. Even though there are no black or white hexagons anymore, there are some specific points in $\omega_\infty(0)$ that should play a significant role and are measurable w.r.t.~$\omega_\infty$: namely, the set of all {\bf pivotal points} of $\omega_\infty$. We shall denote this set by $\bar\Piv = \bar \Piv(\omega_\infty)$, which could indeed be proved to be measurable w.r.t.~$\omega_\infty$ using the methods of \cite[Section 2]{GPS2a}, but we will not actually need this. The ``infinitesimal'' mass we will add to the configuration $\omega_\infty(0)$ will be a certain random subset of $\bar \Piv$. Roughly speaking, one would like to define a mass measure $\bar \mu$ on $\bar \Piv$ and the infinitesimal mass should be given by a Poisson point process $\PPP$ on $(x,\lambda)\in \C\times \R$ with intensity measure $d\bar \mu \times d\lambda$. We would then build our limiting process $\lambda \mapsto \omega_\infty^\nc(\lambda)$ by ``updating'' the initial slice $\omega_\infty(0)$ according to the changes induced by the point process $\PPP$.  So far, the strategy we just outlined corresponds more-or-less to the conceptual framework from \cite{CFN}.

\medskip
The main difficulty with this strategy is the fact that the set of pivotal points $\bar \Piv(\omega_\infty)$ is a.s. a dense subset of the plane of Hausdorff dimension $3/4$ and that the appropriate mass measure $\bar \mu$ on $\bar \Piv$ would be of infinite mass everywhere. This makes the above strategy  too degenerate to work with. To overcome this, one introduces a small spatial cut-off $\eps>0$ which will ultimately tend to zero. Instead of considering the set of all pivotal points, the idea is to focus only on the set of pivotal points which are initially pivotal up to scale $\eps$. Let us denote by $\bar \Piv^\eps=\bar \Piv^\eps(\omega_\infty(\lambda=0))$ this set of $\eps$-pivotal points. The purpose of the companion paper \cite{\GPSa} is to introduce a measure $\bar \mu^\eps=\bar \mu^\eps(\omega_\infty)$ on this set of $\eps$-pivotal points. This limit corresponds to the weak limit of renormalized (by $r(\eta)$) counting measures on the set $\bar \Piv^\eps(\omega_\eta)$, and it can be seen as a ``local time'' measure on the pivotal points of percolation and is called the {\bf pivotal measure}. See Subsection~\ref{ss.pivmeasure}  or \cite{\GPSa} for more detail. (In fact, as we shall recall in Subsection \ref{ss.pivmeasure}, we will consider for technical reasons a slightly different set $\Piv^\eps$, with the corresponding measure $\mu^\eps$). Once such a spatial cut-off $\eps$ is introduced, the idea is to ``perturb'' $\omega_\infty(\lambda=0)$ using a Poisson point process $\PPP=\PPP(\mu^\eps)$ of intensity measure $d\mu^\eps \times d\lambda$ (we now switch to the actual measure $\mu^\eps$ used throughout and which is introduced in Definition~\ref{d.IP}). This will enable us to define a cut-off trajectory $\lambda \mapsto \omega_\infty^{\nc,\eps}(\lambda)$. (In fact the construction of this process will already require a lot of work, see the more detailed outline below). The main problem then is to show that this procedure in some sense stabilizes as the cut-off $\eps \to 0$. This is far from being obvious since there could exist ``cascades'' from the microscopic world which would have macroscopic effects as is illustrated in Figure \ref{f.cascade}.

\begin{figure}[htbp]
\begin{center}
\includegraphics[width=\textwidth]{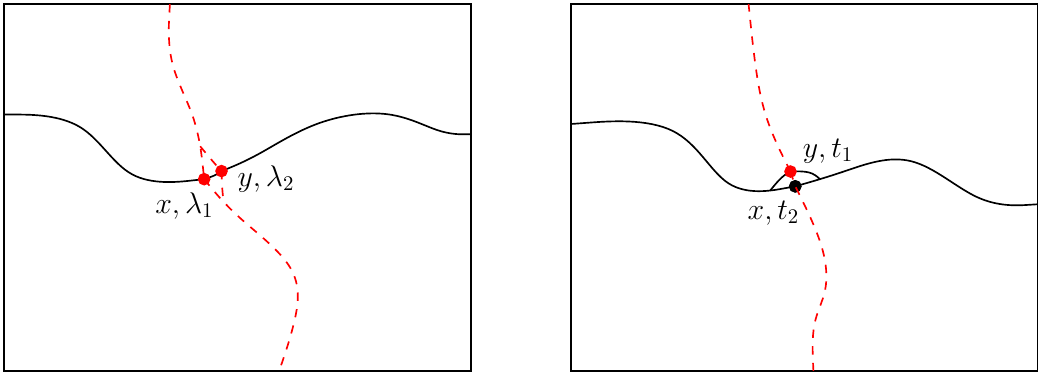}
\end{center}
\caption{Two ``cascade'' configurations: on the left at $\lambda=0$, there is no left-right crossing and both 
points $x$ and $y$ have low importance, but at the level $\lambda_2>\lambda_1$ there is a left-right crossing that we could not predict if we are not looking at low important points. Similarly, in dynamical percolation (on the right), with $t_1< t_2$, the low important point $y$ switches first, followed by the important one. If one does not look at low important points, one would wrongly predict that the left-right crossing has ceased to exist, although it still exists thanks to $y$. Note that the second configuration could occur only for dynamical percolation, which does not have monotonicity in its dynamics.}\label{f.cascade}
\end{figure}

Let us now outline our global strategy in more detail and with pointers to the rest of the text. We will from time to time use notations which will be properly introduced later in the text.
\bnum
\item {\bf Cut-off processes $\lambda\mapsto \omega_\infty^{\nc,\eps}(\lambda)$}

\ni
As discussed above, the first step is to define for each cut-off parameter $\eps>0$ a random cut-off process $\lambda\mapsto \omega_\infty^{\nc,\eps}(\lambda)$ out of a sample of $\omega_\infty\sim \P_\infty$. To achieve this, we will rely on the {\bf pivotal measure} $\mu^\eps=\mu^\eps(\omega_\infty)$ built in the companion paper \cite{\GPSa}. Then, we sample a Poisson point process $\PPP=\PPP(\mu^\eps)$ on $\C\times \R$ with intensity measure $d\mu^\eps \times d\lambda$. This Poisson point process is now a.s.~locally finite. Given $\omega_\infty(0)$ and $\PPP$, we would like to ``update'' $\omega_\infty(0)$ as $\lambda$ increases (or decreases) according to the information provided by the Poisson point process $\PPP$. This step which is straightforward in the discrete setting is more delicate at the scaling limit: indeed, due to the definition of the Schramm-Smirnov space $\HH$ (see Section \ref{s.HH}), updating a point $\omega_\infty(\lambda=0)$ in this space $\HH$ requires in principle to follow the status of all crossing events $\boxminus_Q$ for all quads $Q$ (see Section \ref{s.HH} for these notations). Under the consistency conditions of Lemma~\ref{l.unique} (which will indeed be satisfied), it is enough to follow the status of countably many quads $Q\in \QUAD_\N$. We are thus left with the following problem: given a fixed quad $Q\in \QUAD_\N$ and a level $\lambda \in \R$, can one decide based on $\omega_\infty(0)$ and $\PPP$ whether the process one is building should cross or not the quad $Q$ at level $\lambda$?  

\item {\bf Networks $\Net_Q=\Net_Q(\omega_\infty,\PPP)$}

\ni
To answer the above problem, to each level $\lambda\in \R$ and each quad $Q\in \QUAD_\N$, we will define a kind of graph structure (with  two types of edges, primal and dual ones), called a {\bf network}, whose vertices will be the points in $\PPP_\lambda = \PPP\cap (\C\times [0,\lambda])$ (we assume here that $\lambda>0$), see Definition \ref{d.network}. The purpose of this network is to represent the connectivity properties of the configuration $\omega_\infty(\lambda=0)$ within $Q \setminus \PPP_\lambda$.  This network $\Net_{Q,\lambda}=\Net_{Q,\lambda}(\omega_\infty, \PPP_\lambda)$ is obtained as a limit of {\bf mesoscopic networks} defined in Definition \ref{d.mesoN}. See Theorem \ref{th.BC}.  Once we have at our disposal such a structure $\Net_{Q,\lambda}$ which is furthermore measurable w.r.t. $(\omega_\infty, \PPP)$, one can answer the above question and obtain (assuming that the conditions of Lemma~\ref{l.unique} hold) a well-defined process $\lambda\mapsto \omega_\infty^{\nc,\eps}(\lambda)$ in the space $\HH$.

\item {\bf Convergence in law of $\omega_\eta^{\nc,\eps}(\cdot)$ towards $\omega_\infty^{\nc,\eps}(\cdot)$}

\ni
The convergence in law we wish to prove is under the topology of the Skorohod space $\Sk$ on the space $\HH$ introduced in Lemma \ref{l.SKR}. To prove this convergence, we couple the pairs $(\omega_\eta, \mu^\eps(\omega_\eta))$ and $(\omega_\infty, \mu^\eps(\omega_\infty))$ together so that with high probability the Poisson point processes $\PPP_\eta$ and $\PPP_\infty$ are sufficiently ``close''  so that we get identical networks for macroscopic quads $Q$. See Subsection \ref{ss.couplingPPP} for the coupling. To conclude that $d_\Sk(\omega_\eta^{\nc,\eps(\cdot)}, \omega_\infty^{\nc,\eps}(\cdot))$ tends to zero under this coupling as $\eta\to0$, there is one additional technicality which lies in the fact that $d_\Sk$ relies on the {\it metric} space $(\HH,d_\HH)$  and that the distance $d_\HH$ compatible with the {quad-crossing} topology is non-explicit. To overcome this, we introduce an explicit {\bf uniform structure} in Section~\ref{s.unif}. 

\item {\bf There are no cascades from the microscopic world: $\omega_\eta^\nc(\cdot) \approx \omega_\eta^{\nc,\eps}(\cdot)$}

\ni
This is the step which proves that scenarios like the ones highlighted in Figure~\ref{f.cascade} are unlikely to happen. This type of ``stability'' result is done in Section~\ref{s.stability}. In particular, it is proved that $\Eb{d_\Sk(\omega_\eta^\nc(\cdot), \omega_\eta^{\nc,\eps}(\cdot))}$ goes to zero uniformly as $0<\eta<\eps$ goes to zero. See Proposition~\ref{pr.stab}.

\item {\bf Existence of the limiting process $\lambda \mapsto \omega_\infty^\nc(\lambda)$ and weak convergence of $\omega_\eta^\nc(\cdot)$ to it}

\ni
Once the above steps are established, this last one is more of a routine work. It is handled in Section \ref{s.proof}.  
\enum

For the sake of simplicity, we will assume in most of this paper that our percolation configurations are defined in a bounded smooth simply-connected domain $D$ of the plane (i.e., we will consider the lattice $\eta \Tg \cap D$). Only in Section~\ref{s.proof} will we highlight how to extend our main results to the case of the whole plane, which will consist of a routine compactification technique.  
  
Finally, to make the reading easier,  we include a short list of notations which should be useful:

\section*{Notations}

\begin{small}

\begin{tabular}{ll}

\hspace{4 cm} & \\

$D$ & 
any fixed bounded, smooth domain in $\C$ \\

$\eta \Tg$ & 
the triangular grid with mesh $\eta>0$ \\

$\HH=\HH_D$ 
& Schramm-Smirnov space of percolation configurations in $D$ \\

\vspace{0.3 cm}

$d_\HH$ 
& a distance on $\HH$ compatible with the {quad-crossing} topology $\T$ \\

$\omega_\eta \in \HH \sim \P_\eta$ 
& a critical configuration on $\eta \Tg$ \\

$\omega_\infty \in \HH \sim \P_\infty$ 
& a continuum percolation in the sense of Schramm-Smirnov \\

$\omega_\eta(t),\,
 \omega_\infty(t)$ & 
rescaled dynamical percolation on $\eta \Tg$ and its scaling limit \\

$\omega_\eta^\nc(\lambda), \omega_\infty^\nc(\lambda)$ &
near-critical percolation on $\eta \Tg$ and its scaling limit
\end{tabular}

\ni
\begin{tabular}{ll}
\hspace{4 cm} & \\

$Q$ & a {quad} \\
$\QUAD_D$ & the space of all quads in $D$ \\
$d_\QUAD$ & a distance on $\QUAD_D$ \\ 
$\QUAD_\N= \bigcup \QUAD^k$&  a countable dense subset of $\QUAD_D$ \\
$\boxminus_Q$ & event of crossing the quad $Q$ \\
\end{tabular}

\ni
\begin{tabular}{ll}
\hspace{4 cm} & \\
$\alpha^\eta_1(r,R)$, $\alpha_4^\eta(r,R)$ & 
probabilities of the one and four-arms events for $\omega_\eta\sim\P_\eta$ \\
$r(\eta)$ & the renormalized rate $r(\eta):=\eta^2 \alpha_4^\eta(\eta,1)^{-1}=\eta^{3/4+o(1)}$
\end{tabular}

\ni
\begin{tabular}{ll}
\hspace{4 cm} & \\

$\eps$ & cut-off parameter in space \\
$\mu^\eps=\mu^\eps(\omega_\infty)$ &
pivotal measure on the $\eps$-pivotal points constructed in \cite{\GPSa} \\
$\PPP_T=\PPP_T(\mu^\eps)$ & Poisson point process on $D\times [0,T]$ with intensity measure $d\mu^\eps \times dt$ \\
$\omega_\eta^\eps(t),\, \omega_\infty^\eps(t)$ &
cut-off dynamical percolation and its scaling limit \\

$\omega_\eta^{\nc,\eps}(\lambda),\, \omega_\infty^{\nc,\eps}(\lambda)$ &
cut-off near-critical percolation and its scaling limit \\
$\Piv^\eps(\omega_\eta)$ & 
$\eps$-pivotal points for an $\eps\Z^2$ grid in the sense of Definition \ref{d.IP} \\

$\Sk$ &
space of \cadlag trajectories on $\HH$ \\

$d_\Sk$ & 
Skorohod distance on $\Sk$  \\


\end{tabular}

\end{small}

\section*{Acknowledgements}

We wish to thank Juhan Aru, Vincent Beffara, Itai Benjamini, Nathanael Berestycki, C\'edric Bernardin, Federico Camia, Hugo Duminil-Copin, 
Alan Hammond, Emmanuel Jacob, Milton Jara, Gr\'egory Miermont, Pierre Nolin, Leonardo Rolla, Charles Newman, Mika\"el de la Salle, Stanislav Smirnov, Jeff Steif, Vincent Tassion and Wendelin Werner for many fruitful discussions through the long elaboration of this work. We also wish to thank Alan Hammond as well as an anonymous referee for their very detailed comments.

\section{Space and topology for percolation configurations}\label{s.HH}

There are several ways to {\it represent} a configuration of percolation $\omega_\eta$. 
Historically, the first topological setup appeared with Aizenman in \cite{Aizenman} where the author introduced the concept of {\it percolation web}. The rough idea there is to think of a percolation configuration as the set of all its possible open paths and then to rely on a kind of Hausdorff distance on the space of collections of paths. 

Later on, in the setup introduced in \cite{AiBu,\CN}, one considers  a percolation configuration $\omega_\eta$ as a set of oriented loops (the loops represent interfaces between primal and dual clusters). 
The topology used in \cite{\CN} for the scaling limit of critical percolation $\omega_\eta$ on $\eta\,\Tg$ as the mesh $\eta$ goes to zero, is thus based on the way {\it macroscopic loops} look. See \cite{\CN} for more details. In this work, we will rely on a different representation of percolation configurations which leads to a different topology of convergence. The setup we will use was introduced by the third author and Smirnov in \cite{\SS}. It is now known as the {\it quad-crossing topology}.
We will only recall some of the main aspects of this setup here, so we refer to \cite{\CN, \SS, \GPSa} for a complete description.
(In fact some of the explanations below are borrowed from our previous work \cite{\GPSa}).

\subsection{The space of percolation configurations $\HH$}

The idea in \cite{\SS} is in some sense to consider a percolation configuration $\omega_\eta$ as the set of all the 
{\it quads} that are crossed (or traversed) by the configuration $\omega_\eta$. 
Let us then start by defining properly what we mean by a {\it quad}.


\begin{definition}\label{d.quad}
Let $D\subset \C$ be a bounded domain. 
A {\bf quad} in the domain $D$ can be considered as a homeomorphism $\Quad$ from $[0,1]^2$ into $D$.  A {\bf crossing} of a quad $\Quad$ is a connected closed subset of $[\Quad]:=\Quad([0,1]^2)$ that intersects both $\p_1 Q:= \Quad(\{0\}\times[0,1])$ and $\p_3 Q:=\Quad(\{1\}\times[0,1])$ (let us also define $\p_2 Q:=\Quad([0,1]\times \{0\})$ and $\p_4 Q:= \Quad([0,1]\times \{1\})$). The space of all quads in $D$, denoted by $\QUAD_D$, is equipped with the following metric: $d_Q(\Quad_1,\Quad_2):=\inf_{\phi }\sup_{z\in \p [0,1]^2} |\Quad_1(z)-\Quad_2(\phi(z))|$, where the infimum is over all homeomorphisms $\phi: [0,1]^2 \to [0,1]^2$ which preserve the 4 corners of the square. Note that we use a slightly different metric here as in \cite{\SS,\GPSa}, yet the results from \cite{\SS} still hold. 
\end{definition}

From the point of view of crossings, there is a natural partial order on $\QUAD_D$: we write $\Quad_1 \leq \Quad_2$ if any crossing of $\Quad_2$ contains a crossing of $\Quad_1$.  See Figure~\ref{f.poset}. Furthermore, we write $\Quad_1 < \Quad_2$ if there are open neighborhoods $\mathcal{N}_i$ of $\Quad_i$ (in the uniform metric) such that $ N_1\leq N_2$ holds for any $N_i\in \mathcal{N}_i$.  A subset $S\subset \QUAD_D$ is called {\bf hereditary} if whenever $\Quad\in S$ and $\Quad'\in\QUAD_D$ satisfies $\Quad' < \Quad$, we also have $\Quad'\in S$.

\begin{definition}[The space $\HH$]\label{d.spaceH}
We define the space $\HH=\HH_D$ to be the collection of all {\bf closed} hereditary subsets of $\QUAD_D$.
\end{definition}

Now, notice that any discrete percolation configuration $\omega_\eta$ of mesh $\eta>0$ can be viewed as a point in $\HH$ in the following manner.  Consider $\omega_\eta$ as a union of the topologically closed percolation-wise open hexagons in the plane. It thus naturally defines an element $S(\omega_\eta)$ of $\HH_D$: the set of all quads for which $\omega_\eta$ contains a crossing. By a slight abuse of notation, we will still denote by $\omega_\eta$ the point in $\HH$ corresponding to the configuration $\omega_\eta$.

Since configurations $\omega_\eta$ in the domain $D$ are now identified as points in the space $\HH=\HH_D$,
it follows that critical percolation induces a probability measure on $\HH_D$, which will be denoted by $\P_\eta$.

\begin{figure}[htbp]
\SetLabels
(.17*.02)$\Quad_1$\\
(.97*.74)$\Quad_2$\\
(-.05*.5)$\p_1\Quad_2$\\
(1.07*.4)$\p_3\Quad_2$\\
(.25*.45)$\p_1\Quad_1$\\
(.75*.6)$\p_3\Quad_1$\\
\endSetLabels
\centerline{
\AffixLabels{
\includegraphics[height=1.5 in]{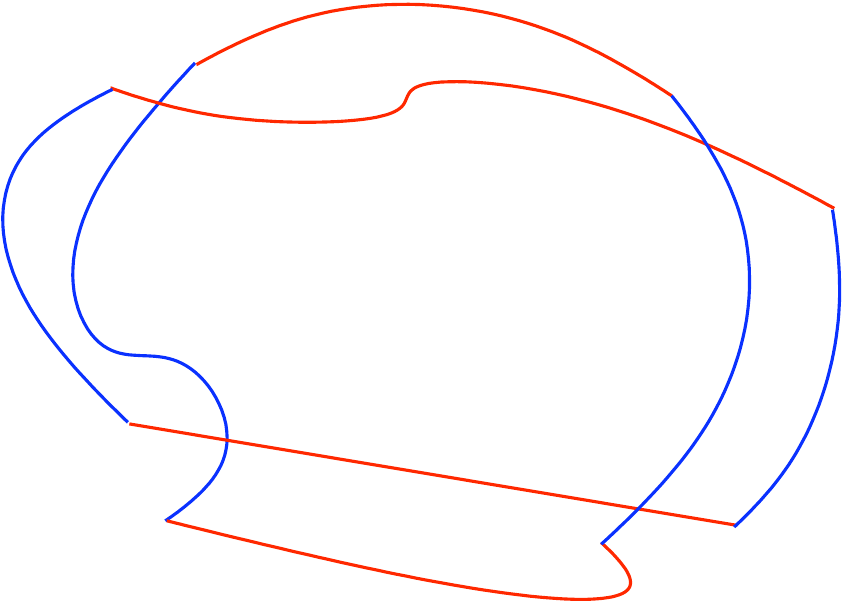}
}}
\caption{Two quads, $\Quad_1 \leq \Quad_2$.}
\label{f.poset}
\end{figure}

In order to study the {\it scaling limit} of $\omega_\eta\sim \P_\eta$, we need to define a topology on the space $\HH$ 
for which the measures $\P_\eta$ will converge weakly as $\eta \to 0$.

\subsection{A topology on percolation configurations: the quad-crossing topology $\T$}

Hereditary subsets can be thought of as Dedekind cuts in the setting of partially ordered sets (instead of totally ordered sets, as usual). It can be therefore hoped that by introducing a natural topology, $\HH_D$ can be made into a compact metric space. Indeed, let us consider the following subsets of $\HH_D$. For any quad $\Quad\in \QUAD_D$, let
\begin{equation}\label{e.boxminus}
\boxminus_\Quad:=\{\omega \in\HH_D:\Quad\in \omega\}\,,
\end{equation}
and for any open $U\subset \QUAD_D$, let 
\begin{equation}\label{e.boxup}
\boxup_U:=\{\omega \in\HH_D:  \omega \cap U=\emptyset\}\,.
\end{equation}

It is easy to see that these sets have to be considered closed if we want $\HH_D$ to be compact.
It motivates the following definition from \cite{\SS}.

\begin{definition}[The quad-crossing topology, \cite{\SS}]\label{d.quadtopology}
We define $\T=\T_D$ to be the minimal topology on $\HH$ that contains every $\boxminus_\Quad^c$ and $\boxup_U^c$ as open sets.
\end{definition}

The following result is proved in \cite{\SS}.
\begin{theorem}[Theorem 3.10 \cite{\SS}]\label{th.compact}

For any nonempty domain $D$, the topological space $(\HH_D,\T_D)$ is compact, Hausdorff, and metrizable.

Furthermore, for any dense $\QUAD_0 \subset \QUAD_D$, the events $\{\boxminus_\Quad : \Quad\in\QUAD_0\}$ generate the Borel $\sigma$-field of $\HH_D$. 

In particular, the space $\HH$ is a Polish space. 
\end{theorem}

This compactness property is very convenient since it implies right-away the existence of
subsequential scaling limits. Similarly the fact $(\HH,\T)$ is Polish will enable us to study the weak convergence of measures on $(\HH,\T)$ in the classical framework of probability measures on Polish spaces. We will come back to this 
in Subsection \ref{ss.scalinglimitSSsense}, but before this,  we discuss 
the {\it metrizable} aspect of $\HH$.

\subsection{On  the metrizability of the topological space $(\HH,\T)$}

As we discussed above, it is stated in \cite{\SS} that the topological space $(\HH,\T)$ is metrizable. 
It would be convenient for our later purposes to have at our disposal a natural and explicit metric on $\HH$ which would induce the topology $\T$. 
The following one, $\tilde d_{\HH}$, seems to be a good candidate since it is ``invariant'' under translations.

For any $\omega,\omega' \in \HH (=\HH_D)$, define
\begin{equation*}\label{e.tildemetric}
\tilde d_\HH(\omega,\omega'):= 
\inf_{\eps>0\text{ such that }}\left\{ \begin{array}{l}
\forall Q\in\omega,\,\exists Q'\in\omega'\,\textrm{ with }\, d_{\QUAD}(Q,Q')<\epsilon\\
\textrm{and}\\
\forall Q'\in\omega',\,\exists Q\in\omega\,\textrm{ with }\, d_{\QUAD}(Q,Q')<\epsilon
\end{array}\right\}\,.
\end{equation*}

As such $(\HH,,\tilde d_\HH)$ is clearly a metric space. It is not hard to check that the topology on $\HH$ induced by $\tilde d_\HH$ is finer than the topology $\T$, but unfortunately, it turns out to be {\it strictly } finer.

After careful investigations, we did not succeed in finding a natural and explicit metric compatible with the topology $\T$. (One possible way is to go through Urysohn's metrization theorem proof, but that does not lead to a nice and explicit metric).
We will thus rely in the remainder of this text on some non-explicit metric $d_\HH$.

\begin{definition}\label{d.metric}
We thus fix once and for all a metric $d_\HH$ on $\HH$ which induces the topology $\T$ on $\HH$. 
In particular, the space $(\HH,d_\HH)$ is a compact metric space. It is also a Polish metric space. 
Since by compactness, $\diam(\HH)<\infty$, we will assume without loss of generality that $\diam_{d_\HH}(\HH)=1$. 
\end{definition}

Since $d_\HH$ is not explicit, we will need to find some explicit and quantitative criteria which will tell us whenever two configurations $\omega, \omega'\in \HH$ are $d_\HH$-close or not. 
This will bring us to the notion of defining explicit {\bf uniform structures} on the topological space $(\HH, \T)$. This will be the purpose of the next section (Section \ref{s.unif}). 
\medskip

But before that, let us review some useful results from \cite{\SchrammSmirnovNoise} and \cite{\GPSa}. 

\subsection{Scaling limit of percolation in the sense of Schramm-Smirnov}\label{ss.scalinglimitSSsense}

This setup we just described allows us to think of $\omega_\eta \sim \P_\eta$ as a random point in the compact metric space $(\HH,d_\HH)$.  Now, since Borel probability measures on a compact metric space are always tight, we have subsequential scaling limits of $\P_\eta$ on $\HH$, as the mesh $\eta_k\to 0$, denoted by $\P_\infty=\P_\infty(\{\eta_k\})$.

One of the main results proved in \cite{\SS} is the fact that any subsequential scaling limit $\P_\infty$ is a {\it noise}
in the sense of Tsirelson (see \cite{\Tsirelson}). But it is not proved in \cite{\SS} that there is a {\bf unique} such subsequential scaling limit. As it is explained in Section 2.3 in \cite{\GPSa}, the uniqueness property follows from 
the work \cite{\CN}. More precisely, \cite{\CN} proves the uniqueness of subsequential scaling limits in a different topological space than $(\HH,d_\HH)$, but it follows from their proof that $\omega\in \HH$ is {\it measurable} with respect to their notion of scaling limit (where a percolation configuration, instead of being seen as a collection of {\it quads}, is seen as a collection of nested {\it loops}). See \cite{\GPSa}, Section 2.3, for a more thorough discussion.

\begin{definition}\label{}
In what follows, we will denote by $\omega_\infty \sim \P_\infty$ the scaling limit of discrete mesh percolations $\omega_\eta \sim \P_\eta$. (Recall $\P_\eta$ denotes the law of critical site percolation on $\eta \Tg$).
\end{definition}

Of course, as explained carefully in \cite{\SchrammSmirnovNoise, \GPSa}, the choice of the space $\HH=\HH_D$ (or any other setup for the scaling limit) already poses restrictions on what events one can work with.
Note, for instance, that $\A:=\{ \exists$ neighborhood $U$ of the origin $0\in\C$ s.t.~all quads $\Quad \subset U$ are crossed$\}$ is clearly in the Borel $\sigma$-field of $(\HH_D,\T_D)$, and it is easy to see that $\P_\infty[\A]=0$, but if the sequence of $\eta$-lattices is such that $0$ is always the center of an hexagonal tile, then $\P_\eta[\A]=1/2$. 


With such an example in mind, it is natural to wonder how to effectively measure crossing events, multi-arms events and so on. Since the crossing event $\boxminus_Q$ is a Borel set, it is measurable and $\P_\infty[\boxminus_Q]$ is thus well-defined. Yet, one still has to check that 
\[
\P_\eta[\boxminus_Q] \to \P_\infty[\boxminus_Q]\,, \text{ as }\eta \to 0\,,
\]
which will ensure that $\P_\infty[\boxminus_Q]$ is given by Cardy's formula. 
This property was proved in \cite{\SchrammSmirnovNoise}. More precisely they prove the following result. 

\begin{theorem}[\cite{\SchrammSmirnovNoise}, Corollary 5.2]\label{th.SScardy}
For any quad $Q\in \QUAD_D$, 
\[
\P_\infty[\p \boxminus_Q] =0\,.
\]
In particular, one indeed has 
\[
\P_\eta[\boxminus_Q] \to \P_\infty[\boxminus_Q]\,, 
\]
as $\eta\to 0$, by weak convergence of $\P_\eta$ to $\P_\infty$. 
\end{theorem}

In the next subsection, we define Borel sets in $(\HH,d_\HH)$ which correspond to the so called mutli-arms events.
They were introduced and studied in \cite{\GPSa} where an analog of the above Theorem \ref{th.SScardy} was proved. See Lemma \ref{l.meas4arm} below.

\subsection{Measurability of arms events (\cite{\GPSa})}\label{ss.armevents}


Following \cite{\GPSa}, if $A=(\p_1 A, \p_2 A)$ is any non-degenerate smooth annulus of the plane (see \cite{\GPSa}), one can define events $\A_1,\A_2, \A_3, \A_4, \ldots, \A_j$ which belong to the Borel sigma-field of $(\HH,d_\HH)$
and which are such that for the discrete percolation configurations $\omega_\eta\sim\P_\eta \in (\HH,d_\HH)$, $1_{\A_i}(\omega_\eta)$ coincides with the indicator function that $\omega_\eta$ has $j$ (alternate) arms in the annulus $A$.

We recall below the precise definition from \cite{\GPSa} in the case where $j=4$ (which is the most relevant case in this paper).

\begin{definition}[Definition of the 4-arm event]\label{d.fourarms}
Let $A=(\p_1 A, \p_2 A)\subset D$ be a piecewise smooth annulus.  We define the {\bf alternating 4-arm event} in $A$ as $\A_4=\A_4(A)=\bigcup_{\delta>0} \A_4^\delta$, where $\A_4^\delta$ is the existence of quads $\Quad_i \subset D$, $i=1,2,3,4$, with the following properties (See figure \ref{f.fourarm}):

\bi
\item[(i)] $Q_1$ and $Q_3$ (resp. $Q_2$ and $Q_4$) are disjoint and are at distance at least $\delta$ from each other. 
\item[(ii)] 
For all $i\in \{1,\ldots,4\}$, the paths $Q_i(\{0\}\times [0,1])$ (resp. $Q_i(\{1\}\times [0,1]) $ lie inside (resp. outside) $\p_1 A$ (resp. $\p_2 A$) and are at distance at least $\delta$ from the annulus $A$ and from the other $Q_j$'s.
\item[(iii)] The four quads are ordered cyclically around $A$ according to their indices.
\item[(iv)] For $i\in \{1,3\}$, $\omega\in \boxminus_{Q_i}$.
\item[(v)] For $i\in \{2,4\}$, $\omega\in  \boxminus_{\rot{Q_i}}^c$,
\ei
where if $Q$ is a {\it quad} in $D$ (i.e. an homeomorphism from, say, $[-1,1]^2$ into $D$), then $\rot{Q}$ denotes the rotated quad by $\pi/2$, i.e. 

\begin{equation}\label{}
\rot{Q}:=Q\circ e^{i\pi/2}\,.
\end{equation}
\end{definition}

\begin{remark}\label{r.open}
Note that by construction, $\A_4=\A_4(A)$ is a {\it measurable} event. In fact, it is easy to check that it is an open set for the quad-topology $\T$.  

Also, the definitions of general {\bf (mono-  or polychromatic) $k$-arm events} in $A$ are  analogous: see \cite{\GPSa} for more details. 

\end{remark}

\begin{figure}[!htp]
\begin{center}
\includegraphics[width=0.7\textwidth]{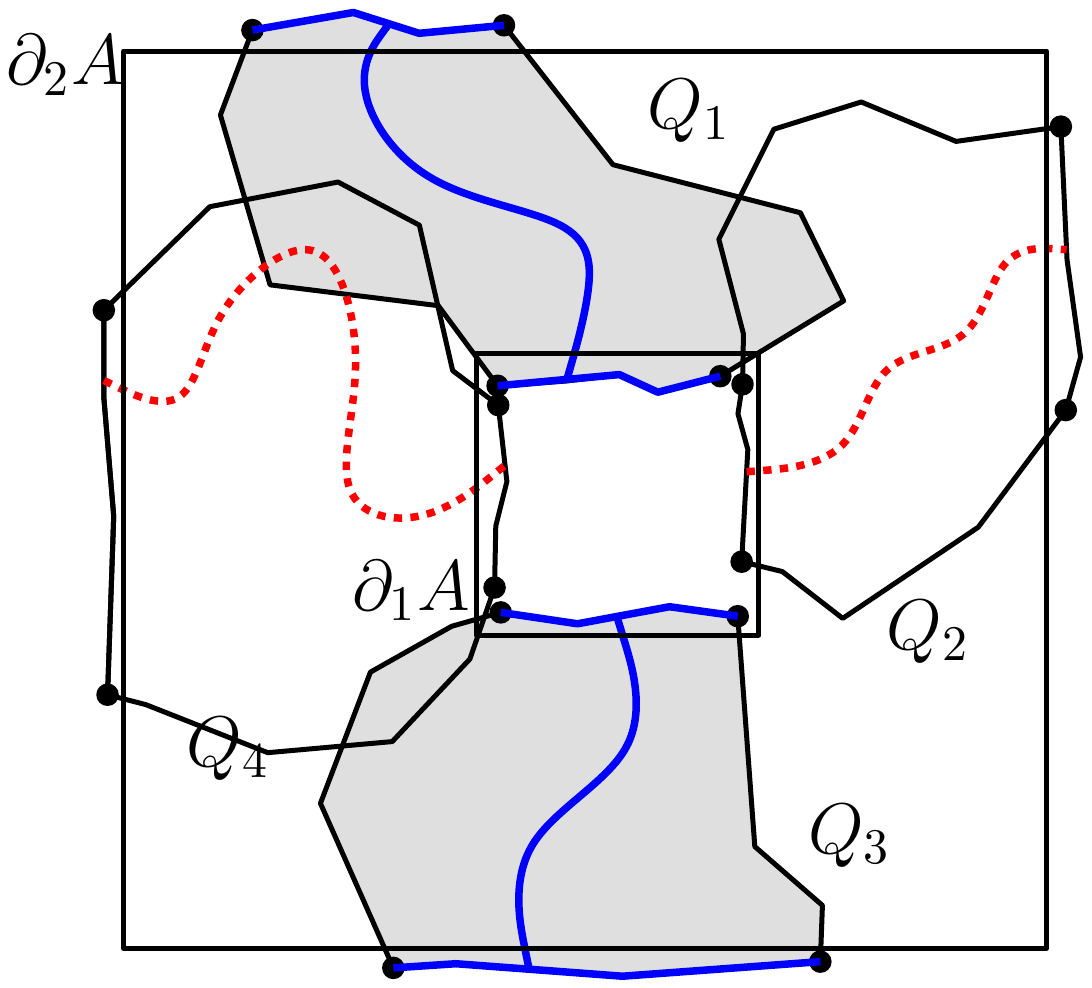}
\end{center}
\caption{Defining the alternating 4-arm event using quads crossed or not crossed.}\label{f.fourarm}
\end{figure}

We will need the following Lemma from \cite{\GPSa}, which is the analog of the above Theorem \ref{th.SScardy}:
 
\begin{lemma}[Lemma 2.4 and Corollary 2.10 in \cite{\GPSa}]\label{l.meas4arm}
Let $A\subset D$ be a piecewise smooth topological annulus (with finitely many non-smooth boundary points). Then the 1-arm, the alternating 4-arm and any polychromatic 6-arm event in $A$, denoted by $\A_1$, $\A_4$ and $\A_6$, respectively, are measurable w.r.t.~the scaling limit of critical percolation in $D$, and one has 
\[
\lim_{\eta \to 0}\P_\eta[\A_i]=\P_\infty[\A_i]\,.
\]
Moreover, in any coupling of the measures $\{\P_{\eta}\}$ and $\P_\infty$ on $(\HH_D,\T_D)$ in which $\omega_{\eta}\to\omega$ a.s.~as $\eta\to 0$,  we have 
\begin{align}\label{e.Delta}
\Pb{\{ \omega_\eta\in \A_i \} \Delta \{\omega \in \A_i\}} \to 0\qquad (\text{as }\eta\to 0)\,.
\end{align}

\ni
Finally, for any exponent $\gamma<1$, there is a constant $c=c_{A,\gamma}>0$ such that, 
for any $\delta>0$ and any $\eta>0$:
\begin{align}\label{e.delta}
\P_\eta \bigl[ \A_4^\delta \md \A_4\bigr] \geq 1- c\, \delta^\gamma\,.
\end{align}
\end{lemma}

\subsection{Pivotal measures on the set of pivotal points (\cite{\GPSa})}\label{ss.pivmeasure}

In what follows, $A=(\p_1 A, \p_2 A)$ will be a piecewise smooth annulus with inside face denoted by $\Delta$. The purpose of \cite{\GPSa} is to study the scaling limit of suitably renormalized counting measures on the set of $A$-important points where these latter points are defined as follows:
\begin{definition}\label{d.AIP}
 For any $\eta>0$, 
a point $x\in \eta \Tg \cap \Delta$ is $A$-important for the configuration $\omega_\eta$ if one can find four alternating arms in $\omega_\eta$ from $x$ to the exterior boundary $\p_2 A$. See figure \ref{f.PM}
\end{definition}

\begin{figure}[!htp]
\begin{center}
\includegraphics[width=\textwidth]{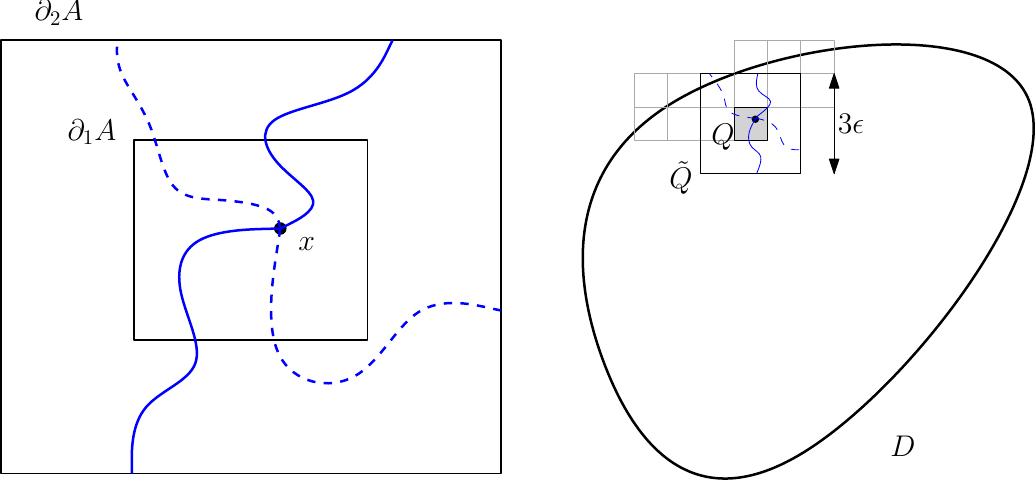}
\end{center}
\caption{On the left picture, a point $x$ which is $A$-important for the annulus $A=(\p_1 A, \p_2 A)$. On the right, a point which is $\eps$-important, i.e. in $\Piv^\eps$.}\label{f.PM}
\end{figure}

\begin{definition}[Pivotal measure $\mu^A$]\label{}
  Let us introduce the following counting measure on the set of $A$-important points:
 \[
 \mu^A=\mu^A(\omega_\eta):= \sum_{\begin{array}{c} x\in \eta \Tg \cap \Delta \\ x \text{ is $A$-important} \end{array}  } \,\, \delta_x \, \eta^2 \alpha_4^\eta(\eta,1)^{-1}\,. 
 \]
\end{definition}

The main Theorem in \cite{\GPSa} may be stated as follows:
\begin{theorem}[Theorem 1.1 in \cite{\GPSa}]\label{th.PM}
For any annulus $A$ as above, there is a measurable map $\mu^A$ from $(\HH,d_\HH)$ into the space $\M$ of finite Borel measures on $\bar \Delta$ such that  
\[
(\omega_\eta, \mu^A(\omega_\eta)) \overset{(d)}{\longrightarrow} (\omega_\infty, \mu^A(\omega_\infty))\,,
\]
as the mesh $\eta\to 0$. The topology on $\M$ is the topology of weak convergence (see the Prohorov metric $d_\M$ in~\eqref{e.Proho}) and the above convergence in law holds under the product topology of $d_\HH$ by $d_\M$. 
\end{theorem}

For each $\eps>0$, let us consider the grid $\eps \Z^2 \cap D$. Since the domain $D$ is assumed to be bounded, there are finitely many $\eps$-squares in this grid. To each such square $Q$ (if $Q$ intersects $\p D$, we still consider the entire $\eps$-square), we associate the square $\tilde Q$ of side-length $3\eps$ centered around $Q$ and we consider the annulus $A_Q$ so that $\p_1 A_Q = \p Q$ and  $\p_2 A_Q = \p \tilde Q$. See figure \ref{f.PM}. 

\begin{definition}\label{d.IP}
For any $\eta>0$, we define the set $\Piv^\eps=\Piv^\eps(\omega_\eta)$ to be the set of points $x\in \eta \Tg \cap D$, which are such that $x$ belongs to an $\eps$-square $Q$ in the grid $\eps \Z^2$ and $x$ is $A_Q$-important for the configuration $\omega_\eta$. The points in $\Piv^\eps$ are called {\bf $\eps$-important points}. 

Furthermore, we will denote by $\mu^\eps=\mu^\eps(\omega_\eta)$ the pivotal measure on these $\eps$-important points, namely:
\[
\mu^\eps=\mu^\eps(\omega_\eta):= \sum_{x\in \Piv^\eps(\omega_\eta)} \, \delta_x\, \eta^2 \alpha_4^\eta(\eta,1)^{-1}\,. 
\] 
\end{definition}


Theorem \ref{th.PM} above clearly implies the following result on the scaling limit of $\mu^\eps(\omega_\eta)$:
\begin{corollary}\label{c.PM}
For any $\eps>0$,
there is a measurable map $\mu^\eps$ from $(\HH,d_\HH)$ into the space of finite Borel measures on $\bar D$, such that 
\[
(\omega_\eta, \mu^\eps(\omega_\eta)) \overset{(d)}{\longrightarrow} (\omega_\infty, \mu^\eps(\omega_\infty))\,,
\]
under the above product topology. 
\end{corollary}

Furthermore, in the Proposition below, we list some properties on the pivotal measure $\mu^\eps$ from \cite{\GPSa}.

\begin{proposition}[Proposition 4.4 in \cite{\GPSa}]\label{th.mu}

There is a universal constant $C>0$ (which does not depend on $\eps>0$) such that 
%
\bi
\item[(i)] for any smooth bounded open set $U\subset \bar D$, 
\[
\Eb{\mu^\eps(U)}< C\, \eps^{-5/4}\, \mathrm{area}(U)\,,
\]

\item[(ii)] for any $r$-square $S_r=(x,y)+[0,r]^2$ included in $\bar D$, 
\[
\Eb{\mu^\eps(S_r)^2}< C\, \eps^{-5/4}\, r^{11/4} = C\, \eps^{-5/4} \mathrm{area}(S_r)^{11/8}\,.
\]
\ei
\ni
(In fact this second moment estimate does not hold for all shapes of open sets $U$). 

\end{proposition}

\begin{remark}\label{r.PIV}
As mentioned in the introduction, it may seem easier or more natural to consider the set $\bar{\Piv}^\eps=\bar{\Piv}^\eps(\omega_\eta)$ of all points $x\in \eta \Tg$ which are such that $\omega_\eta$ satisfies a  four-arms event in the euclidean
ball $B(x,\eps)$.  But the techniques in \cite{\GPSa} would not provide a scaling limit for the corresponding pivotal measures $\bar \mu^\eps$. Yet, it is easy to check that for any $\eps>0$, one always has
\[
\bar{\Piv}^{2\sqrt{2} \eps} \subset \Piv^\eps \subset \bar{\Piv}^\eps\,,
\] 
which will be a useful observation later in Section \ref{s.stability}. 
\end{remark}

%
%
%
%

 \section{Notion of uniformity on the space $\HH$}\label{s.unif}

 When dealing with \cadlag processes on a topological space $(X,\tau)$, one needs a way to compare two different \cadlag trajectories on $X$. But in general, just having a topology $\tau$ on $X$ is not enough for such a task. A notion of {\em uniformity} is needed and this brings us to the notion of {\bf uniform structure}.
(Part of this section, in particular item $(ii)$ and its proof in Proposition \ref{pr.CovrtoCovk} are borrowed from \cite{\GPSa} and are included here for completeness).
 
\subsection{Uniform structure on a topological space}

A {\bf uniform structure} on a topological space $(X, \tau)$ is a given family $\Phi$ of {\em entourages}, which are subsets of $X\times X$. The uniform structure $\Phi$ needs to satisfy a few properties (like symmetry, a certain type of associativity and so  on) and needs to generate in a certain sense the topology $\tau$. See \cite{Tukey} for example for an introduction on uniform spaces. 
If $\tau$ is generated by a metric $d_X$, then the canonical uniform structure on the metric space $(X,d_X)$ is generated by the {\em entourages} of the form $U_a:= \{ (x,y)\in X\times X, \, d_X(x,y)< a \}, a>0$. 
Furthermore, the following fact is known (see for example\cite{Tukey}).

\begin{proposition}\label{pr.Bourb}
If $(X,\tau)$ is a compact Hausdorff topological space, then there is a {\bf unique} uniform structure on $(X,\tau)$ compatible with the topology $\tau$. 
\end{proposition}

We will not rely explicitly on this Proposition nor on the exact definition of uniform structures, but we state these in order to show the intuition underlying the setup to come. 

\subsection{Two useful coverings of $(\HH,\T)$}

\subsubsection{The first covering : with metric balls}

For any radius $r>0$, one can cover $\HH$ by $\{ B_{d_\HH}(\omega,r),\, \omega\in\HH \}$. Since $(\HH,d_\HH)$ is compact,
there is a finite subcover 
\begin{equation}\label{}
\Cov_{d_\HH}^r:=\{ B(\omega_i^r,r),\, i=1 \ldots N_r\}\,.
\end{equation}

\subsubsection{The second covering : with $\boxup_U^c$ and $\boxminus_Q^c$ open sets}

In order to introduce an interesting covering of $\HH$ consisting of open sets as the complements of closed sets \eqref{e.boxminus}, \eqref{e.boxup}, let us first introduce one particular dense countable family of quads in $\QUAD(=\QUAD_D)$.
\medskip

\begin{definition}\label{d.kquads}(A dyadic family of quads)

For any $k\geq 1$, let $(Q_{n}^{k})_{1\le n \le N_{k}}$ be the family of all quads which are {\it polygonal quads} in $D\cap 2^{-k} \Z^{2}$, i.e. their boundaries $\p Q_n^k$ are included in $D\cap 2^{-k} \Z^{2}$ and the four marked vertices are vertices of $D\cap 2^{-k} \Z^{2}$. (For fixed $k$, there are finitely many such quads since the domain $D$ is assumed to be bounded). We will denote by $\QUAD^k=\QUAD^k_D$ this family of quads. Notice that $\QUAD^k\subset \QUAD^{k+1}$. 
\end{definition}

Clearly, the family $\QUAD_\N:= \bigcup_k \QUAD^k$ is dense in the space of quads $(\QUAD_D,d_\QUAD)$. In particular, Theorem~\ref{th.compact} implies that  the events $\{\boxminus_\Quad : \Quad\in \QUAD_\N\}$ generate the Borel $\sigma$-field of $\HH_D$.


In order to use the open sets of the form $\boxup_U^c$, where $U$ is an open set of $\QUAD_D$, we will associate to each $Q\in \QUAD^k,\, (k\geq 1)$, the open set $\hat Q_k:= B_{d_\QUAD}(Q,2^{-k-10})$, and with a slight abuse of notation, we will  write $\boxup_{\hat Q}^c$ for the open set $\boxup_{\hat Q_k}^c$. 

Also, to each quad $Q\in \QUAD^k$, we will associate the quad $\bar Q_k\in \QUAD^{k+10}$ which among all quads $Q'\in\QUAD^{k+10}$ satisfying $Q'>Q$ is the smallest one. Even though $>$ is not a total order, it is not hard to check that $\bar Q_k$ is uniquely defined. See Figure \ref{f.barQ} for an illustration. Note furthermore that $\bar Q_k$ satisfies $d_\QUAD(\bar Q_k, Q)\in [2^{-k-10},2^{-k-5}]$. Since, by definition $\bar Q_k > Q$, one has that $\boxminus_{Q}^c \subset \boxminus_{\bar Q_k}^c$. With a slight abuse of notation, we will write $\boxminus_{\bar Q}^c$ for the open set $\boxminus_{\bar Q_k}^c$.

\begin{figure}[!htp]
\begin{center}
\includegraphics[width=0.8\textwidth]{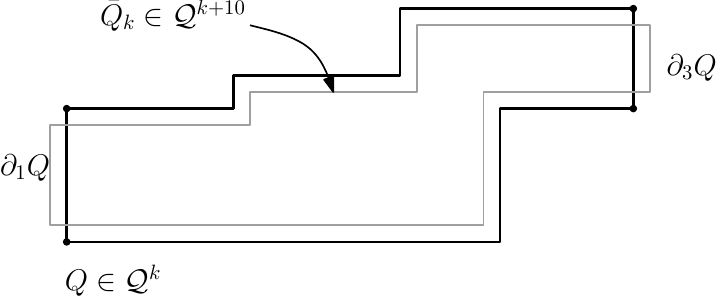}
\end{center}
\caption{Definition of $\bar Q_k>Q$}\label{f.barQ}
\end{figure}

\begin{definition}[A family of neighborhoods]\label{d.open}
For each $k\geq 1$ and each point $\omega \in \HH$, let $\O_k(\omega)$ be the following open set 
\begin{equation}\label{}
\O_k(\omega):= \big( \bigcap_{Q\in \QUAD^k,\,s.t. \, Q\notin \omega }\, \boxminus_{\bar Q}^c \big) \bigcap 
\big( \bigcap_{Q \in \QUAD^k \,s.t. \, Q\in \omega } \boxup_{\hat Q}^c\big)\,.
\end{equation}
Let also $\O_0(\omega)$ be the space $\HH$ for any $\omega\in \HH$. Intuitively, one should think of $\O_k(\omega)$ as the set of percolation configurations which share similar crossing properties with $\omega$ for all quads in $\QUAD^k$ possibly with some small $2^{-k}$-deformations. 
\end{definition}

\begin{remark}\label{}
Since for any $Q\in \QUAD^k$, $\boxminus_Q^c$ is already an open set, one might wonder why we have chosen here to relax $\boxminus_Q^c$ into $\boxminus_{\bar Q}^c$. This choice will make the statements and proofs to come more symmetric and easier to handle. 
\end{remark}

\begin{remark}\label{}
Let us point out that for any $\omega\in \HH$ and any $k\geq 0$, we have that $\O_{k+1}(\omega) \subset \O_{k}(\omega)$. 
This illustrates the fact that the {\bf finite} coverings $\Cov^k:= \{  \O_k(\omega)\, \omega\in \HH \}$ are finer and finer as $k\to \infty$. 
\end{remark}


This family of open sets is useful thanks to the following property.

\begin{lemma}\label{l.subbase}
The collection of open sets $\{\O_k(\omega),\, k\geq 1, \omega\in \HH \} $ is a (countable) subbase 
for the topological space $(\HH, \T)$.
\end{lemma}
The Lemma follows from the proof of Lemma 3.7 in \cite{\SS} and the fact that the collection of open sets $\{ \hat Q_k, \, Q\in \QUAD^k, k\geq 1 \}$ is a countable basis for the topological space $(\QUAD, d_\QUAD)$.

In some sense, the purpose of the next subsection is to see how these two different coverings of $(\HH,\T)$ relate to each other.

\subsection{Two uniform structures on $(\HH,\T)$}

The first natural uniform structure on $(\HH, \T)$ is of course given by the metric $d_\HH$. 
We now wish to give an {\em explicit} uniform structure on $(\HH,\T)$ which in the end shall be the same as the one given by $d_\HH$. 
This uniform structure will be defined using a {\bf pseudometric} on $\HH\times \HH$. 

\begin{definition}[An explicit uniform structure on $(\HH,\T)$]\label{d.K}
Let us start by introducing the following quantity on $\HH\times \HH$:
\begin{equation}\label{}
K_\HH(\omega,\omega'):= 
\sup \{k\geq 0  ,\, s.t.\, \omega' \in \O_k(\omega) \text{ or } \omega \in \O_k(\omega') \}
\end{equation}
It is easy to check that $(K_\HH)^{-1}$ defines a pseudometric on $\HH$. 
For each $k\geq 0$, let 
\[
U_k := \{ (\omega,\omega') \in \HH \times \HH, \,s.t.\, K_\HH(\omega,\omega')\geq  k\}\,.
\]
\end{definition}

Even though we will not need this fact, the following proposition holds.
\begin{proposition}\label{pr.K}
The family $\{ U_k, k\geq 0 \}$ defines a uniform structure on $\HH$ compatible with the topology $\T$.
\end{proposition}
The non-straightforward part in the proof of this proposition would be to show that the family $\{ U_k\}$ indeed satisfies to the transitivity condition needed for uniform structures. This step will be implicitly proved along the proof of Proposition \ref{pr.CovrtoCovk}.

\subsection{How these two different uniform structures relate to each other}

The purpose of this subsection is the following result.

\begin{proposition}\label{pr.CovrtoCovk}
One can define two functions 
\[
\begin{cases}
r>0\, &\mapsto \;\;  \k(r) \in \N^* \\
& \text{and} \\
k\in \N^* & \mapsto \;\;\r(k)>0
\end{cases}
\]

which are such that the following properties hold:
\bi
\item[(i)] For any $r>0$ and any $\omega,\omega'\in \HH$,
if $K_\HH(\omega,\omega')\geq \k(r)$ then 
\[
d_\HH(\omega,\omega')\le r\,.
\]
In other words, if two configurations share the same crossing properties for all quads in $\QUAD^{\k(r)}$ (up to a small perturbation of about $2^{-k-10}$), then these two configurations are necessarily $r$-close for the $d_\HH$ metric. 
\item[(ii)]
For any $k\geq 1$ and any $\omega,\omega'\in \HH$. If  $d_\HH(\omega,\omega')\le \r(k)$, then 
$\omega'\in \O_k(\omega)$ and $\omega\in \O_k(\omega')$. (Note that it implies $K_\HH(\omega,\omega')\geq k$.) 

In other words, if $\omega$ and $\omega'$ are sufficiently close ($\r(k)$-close), then (up to a small perturbation of $2^{-k-10}$) they share the same crossing properties for all quads in $\QUAD^k$.
\ei

\end{proposition}

Let us point out here that one could prove this Proposition by first proving Proposition \ref{pr.K} and the existence of these two functions would then follow from Proposition \ref{pr.Bourb}. Nevertheless, since it does not make the proof much longer, we will give a self-contained proof here which thus bypasses the notion of {\em uniform structure} as well as its axioms. Also, as mentioned at the beginning of this section, item $(ii)$ corresponds exactly to the content of Lemma 2.5. in \cite{\GPSa} but is included here for completeness.
 
\medskip
\ni
{\bf Proof of the proposition.}
Let us start with the proof of $(ii)$.
Let us fix some integer $k\geq 0$.
Let $k'$ be a slightly larger integer, say $k+20$.  
Notice that since there are finitely many quads in $\QUAD^{k'}$, there are only finitely many possible open sets of the form $\O_{k'}(\omega)$ and the union of these covers $\HH$. Note also that if $\omega\in \HH$ then necessarily the open set $\O_{k'}(\omega)$ is non-empty since it contains the point $\omega$.  It follows from these easy observations that to any point $\omega\in \HH$, one can associate a radius $r_\omega>0$ so that the ball $B_{d_\HH}(\omega,2 \,r_\omega)$  is included at least in one of the open sets $\O_{k'}(\omega)$. 
Consider now the covering $\{ B_{d_\HH}(\omega, r_\omega), \omega\in \HH \}$ from which one can extract a finite covering 
\[
\{ B_{d_\HH}(\omega_i, r_i), i=1,\ldots, N_{k'} \}\,.
\]
Let us define $\r(k):= \min_{1\le i \le N_{k'}} \{ r_i \}$ and let us check that it satisfies the desired properties.
Let $\omega,\omega'$ be any points in $\HH$ such that $d_\HH(\omega, \omega') \le \r(k)$. By our choice of $\r(k)$, one can find at least one ball $B_{d_\HH}(\omega_i, r_i)$ in the above covering such that both $\omega$ and $\omega'$ lie in the ball $B_{d_\HH}(\omega_i, 2\, r_i)$. In particular this means that one can find some $\bar \omega \in \HH$ such that both $\omega$ and $\omega'$ lie in $\O_{k'}(\bar \omega)$. Let us now prove that $\omega' \in \O_k(\omega)$, the other condition is proved similarly. 
Consider any quad $Q\in \QUAD^k$. We will distinguish the following cases:

\bi
\item[(a)] Suppose $Q \in \bar\omega$ and $Q\in \omega$. Since $\omega'\in \O_{k'}(\bar \omega)$, we have that $\omega'\in \boxup_{\hat Q_{k'}}^c \subset \boxup_{\hat Q_k}^c$.

\item[(b)] Suppose $Q \in \bar\omega$ and $Q\notin \omega$. We need to show that $\bar Q_k \notin \omega'$. For this, note that one can find a quad $R$ in $\QUAD^{k'}$ which is such that $\bar Q_k > \bar R_{k'}$ and $\hat R_{k'}>Q$ (in the sense that all the quads in the open set $\hat R_{k'}$ are larger than $Q$). If $R$ happened to be in $\bar \omega$, then since $\omega\in \O_{k'}(\bar \omega)$, $Q$ would necessarily belong to $\omega$. Hence $R\notin \bar \omega$ and thus $\bar R_{k'}\notin \omega'$ which implies $\bar Q_k \notin \omega'$. 

\item[(c)] Suppose $Q\notin \bar \omega$ and $Q\in \omega$. We need to show that $\omega' \in \boxup_{\hat Q_k}^c$. As in case $(b)$, note that one can find a quad $R\in \QUAD^{k'}$ such that $Q>\bar R_{k'}$ and $\hat R_{k'} \subset \hat Q_k$. If $R$ was not in $\bar \omega$, then $Q$ would not be in $\omega$ either. Hence $R\in \bar \omega$ and thus $\omega'\in \boxup_{\hat R_{k'}}^c \subset  \boxup_{\hat Q_k}^c$.

\item[(d)] Finally, suppose $Q \notin \bar \omega$ and $Q\notin \omega$. Note that  $\bar Q_k > \bar Q_{k'}$. Since $\omega' \in \O_{k'}(\bar \omega)$, $\bar Q_{k'}$ is not in $\omega'$ and thus $\omega' \in  \boxminus_{\bar Q_k}^c$, which ends the proof of $(ii)$. 
\ei

\medskip
Let us now turn to the proof of $(i)$. 
Fix some radius $r>0$ and let $\Cov_{d_\HH}^r:=\bigcup_{i=1}^{N_r} B(\omega_i^r,r/2)$ be a finite covering of $\HH$ by balls of radii $r/2$. For any point $\omega\in \HH$, we claim that there exists a large enough integer $k_\omega$ such that the open set $\O_{k_\omega}(\omega)$ is included in at least one of the open balls $B(\omega_i^r,r/2), 1\le i \le N_r$.  This follows from  Lemma \ref{l.subbase}.

Let us consider the following covering of $\HH$: $\{  \O_{k_\omega+10}(\omega), \, \omega\in \HH \}$ from which one can extract a finite covering $\{ \O_{k_j+10}(\omega_j),\, j=1,\ldots, M_r \}$. 
We define 
\[
\k(r):= \max_{1\le j \le M_r} k_j + 20 \,.
\]
Let us check that it satisfies the desired property. 
Let thus $\omega,\omega'$ be two configurations in $\HH$ such that $K(\omega,\omega')\geq \k(r)$ and 
suppose we are in the case where $\omega' \in \O_{\k(r)}(\omega)$. 
There is at least one $j\in [1, M_r]$ such that $\omega\in \O_{k_j+10}(\omega_j)\subset \O_{k_j}(\omega_j)$.
In order to conclude, we only need to check that $\omega' \in \O_{k_j}(\omega_j)$ since this would imply that both $\omega,\omega'$ belong to $\O_{k_j}(\omega_j)$ which itself is contained in a $d_\HH$-ball of radius $r/2$. 

Let $Q$ be any quad in $\QUAD^{k_j}$. We distinguish two cases:
\begin{enumerate}
\item Suppose $Q\notin \omega_j$. Since we assumed $\omega \in \O_{k_j+10}(\omega_j)$, we have that $R:=\bar Q_{k_j+10}\notin \omega$. Now notice that $R\in \QUAD^{k_j+20}\subset \QUAD^{\k(r)}$ and that $\bar Q_{k_j} > \bar R_{\k(r)}$ (this uses the fact that $2^{-k_j-10-5}+2^{-\k(r)-5}<2^{-k_j-10}$). Since $\omega'\in \O_{\k(r)}(\omega)$, we have that $\bar R_{\k(r)} \notin \omega'$ and thus $\bar Q_{k_j} \notin \omega'$. 
\item
 Suppose $Q\in \omega_j$.
 Similarly to the above cases, one can find a quad $R\in \QUAD^{\k(r)}$ such that $R< \hat Q_{k_j+10}$ and $\hat R_{\k(r)} \subset \hat Q_{k_j}$ (this uses the fact that $2^{-k_j-20}+ 2^{-\k(r)-10}<2^{-k_j-10}$).
 Since $\omega\in \O_{k_j+10}(\omega_j)$ and $R< \hat Q_{k_j+10}$, necessarily $R\in \omega$ and since $\omega'\in O_{\k(r)}(\omega)$, $\omega'\in \boxup^c_{\hat R_{\k(r)}}\subset \boxup^c_{\hat Q_{k_j}}$ which concludes our proof.
\end{enumerate}
\QED

\section{Space and topology for \cadlag paths of percolation configurations}\label{s.skorohod}

As we mentioned earlier, both in the cases of dynamical percolation ($t\mapsto \omega_\eta(t)$) and near-critical percolation ($\lambda \mapsto \omega_\eta^\nc(\lambda)$), our processes will be considered as \cadlag processes with values in the metric space $\HH$. 

Recall from Theorem \ref{th.compact} and Definition \ref{d.metric} that $(\HH,d_\HH)$ is a Polish space. 
It is a classical fact that if $(X,d)$ is a Polish space and if $\DD_X=\DD_X[0,1]$ denotes the space of \cadlag functions from $[0,1]$ to $X$, then one can define a metric $d_\Sk$ on $\DD_X$ for which $(\DD_X, d_\Sk)$ is a Polish space. This metric is usually known under the name of {\bf Skorohod metric}. Let us summarize these facts in the following Proposition.
\begin{proposition}[See for example \cite{\EK}, Chapter 3.5]\label{pr.sko}
Let $(X,d)$ be a Polish metric space (i.e. a complete separable metric space). Let $\DD_X=\DD_X[0,1]$ be the space of \cadlag functions $[0,1]\to X$. Then $\DD_X$ is a Polish metric space under the Skorohod metric $d_\Sk$ defined as follows: 
for any \cadlag processes $x,y : [0,1]\to X$, define
\begin{align}\label{}
d_\Sk(x,y)& : = \inf_{\lambda \in \Lambda} \left\{ \| \lambda \| \vee \sup_{0\le u \le 1} d_X(x(u), y(\lambda(u)))\right\}\,, \nn
\end{align}
where the infimum is over the set $\Lambda$ of all strictly increasing continuous mappings of $[0,1]$ onto itself and where 
\begin{equation}\label{}
\|\lambda\|: = \sup_{0\leq s<t\leq 1} |\log \frac{\lambda(t)-\lambda(s)}{t-s}|\,.
\end{equation}
\end{proposition}



This discussion motivates the following definition:
\begin{definition}\label{d.skorohod}
For any $T>0$, let $\Sk_T:= \DD_\HH [0,T]$ be the space of \cadlag processes from $[0,T]$ to $\HH$ and following Proposition \ref{pr.sko}, let 

\begin{align}\label{}
d_{\Sk_T}(\omega(t), \tilde \omega(t)) & : = \inf_{\lambda \in \Lambda_T} \left\{ \| \lambda \| \vee \sup_{0\le u \le T} d_\HH(\omega(u), \tilde \omega(\lambda(u)))\right\}\,, \nn
\end{align}
Here we used the same notations as in Proposition \ref{pr.sko}  (at least their natural extensions to $[0,T]$). When the context is clear, we will often omit the subscript $T$ in the  notation $d_{\Sk_T}$. 
\end{definition}

We will also need the following extension to $\R_+$ and $\R$:
\begin{lemma}\label{l.SKR}
Let $\Sk_{(-\infty,\infty)}$ (resp. $\Sk_{[0,\infty)}$) be the space of \cadlag  processes from $\R$ (resp. $[0,\infty)$) to $\HH$. 
Then, if we define 
\begin{equation}\label{}
d_{\Sk_{(-\infty,\infty)}}(\omega(\lambda), \tilde\omega(\lambda)) := 
\sum_{k\geq 1} \frac 1 {2^k} d_{\Sk_{[-k,k]}}(\omega, \tilde \omega) \,,
\end{equation}
this gives us a Polish space $(\Sk_{(-\infty,\infty)}, d_{\Sk_{(-\infty,\infty)}} )$ 
(and analogously for $\Sk_{[0,\infty)}$). 
\end{lemma}
The proof of this lemma is classical. Note that since $\HH$ is compact, one always has  $d_{\Sk_{[-T,T]}}(\omega,\tilde \omega) \leq \diam(\HH)$ for any $\omega,\tilde\omega\in \HH,\,T>0$. This way we do not need to rely on more classical expressions such as $\sum_k 2^{-k} \frac {d_k}{1+d_k}$.

\section{Poisson point processes on the set of pivotal points}\label{s.poisson}

In this section, we will fix the bounded domain $D$ as well as a cut-off scale $\eps>0$. 
Our aim in this section is to define a Poisson point process on $D$ with intensity measure $d \mu^\eps(x) \times dt$ and to study some of its properties.

\subsection{Definition}
Recall from subsection \ref{ss.pivmeasure} that in \cite{\GPSa}, we defined for any fixed  $\eps>0$, a measure 
$\mu^\eps=\mu^\eps(\omega_\infty)$ (in the sense that it is measurable w.r.t. $\omega_\infty \sim \P_\infty$) which is such that  
\begin{equation}\label{}
(\omega_\eta, \mu^\eps(\omega_\eta)) \overset{(d)}{\longrightarrow} (\omega_\infty, \mu^\eps(\omega_\infty))\,,
\end{equation}
as $\eta\to 0$. 
 
\begin{definition}\label{d.PPP}
Let $T>0$ and $\eps>0$ be fixed. We will denote by $\PPP_T=\PPP_T(\mu^\eps(\omega_\infty))$  the Poisson point process 
\[
\PPP_T = \{ (x_i,t_i), 1\le i \le N \}
\]
in $D\times [0,T]$ of intensity measure $d\mu^\eps(x)\times 1_{[0,T]}\, dt $.
Furthermore, we will denote by $\switch_T$ the random set of switching times $\switch_T:=\{t_1,\ldots, t_N\}\subset [0,T]$.
\end{definition}

\subsection{Properties of the point process $\PPP_T$}

We list below some useful a.s. properties for $\PPP_T$.

\begin{proposition}\label{pr.PPP}
Let $T>0$ and $\eps>0$ be fixed. 
Then the cloud of points $\PPP_T=\PPP_T(\mu^\eps(\omega_\infty))$ a.s. satisfies the following properties:
\bi
\item[(i)] The set of points $\PPP_T$ is finite. This justifies the notation $\PPP_T = \{ (x_i,t_i), 1\le i \le N \}$. 
In particular the set of switching times $\switch_T=\{t_1,\ldots,t_N\}\subset [0,T]$ is finite. 
\item[(ii)] The points in $\PPP_T$ are at positive distance from each other, i.e.
\[
\inf_{i\neq j} \{ |x_i-x_j| \} >0
\]
\item[(iii)] Similarly, the switching times in $\switch_T$ are at positive distance from each other. 
\item[(iv)] For any quad $Q\in \QUAD_\N$, the set $\PPP_T$ remains at a positive distance from $\p Q$, i.e. 
\[
\dist(\p Q, \PPP_T) >0
\]
\ei
\end{proposition}

%

\ni
{\bf Proof of the proposition:}

\ni
\bi
\item The first property $(i)$ follows directly from item $(i)$ in Proposition \ref{th.mu} applied to $U=D$ (we use here our assumption that $D$ is bounded). 
\item To prove the second property, notice that for any $k\in \N$, if $\inf_{i\neq j} \{ |x_i-x_j| \} <2^{-k}$, this means that one can find at least one dyadic square of the form $[i 2^{-k+2}, (i+1)2^{-k+2}]\times [j 2^{-k+2}, (j+1)2^{-k+2}]$ which contains at least two points of $\PPP_T$. 
Since $D$ is bounded, one can cover $D$ with $O(1) 2^{2k}$ such dyadic squares (where $O(1)$ depends on $D$). If $S=S_k$ is any dyadic square of the above form, one has 
\[
\Pb{|\PPP_T \cap S| \geq 2 \md \mu^\eps} \leq T^2 \, \mu^\eps(S)^2\,.
\]
Integrating w.r.t. $\mu^\eps$ gives 

\begin{align}\label{}
\Pb{|\PPP_T \cap S| \geq 2 } & \le T^2 \,  \Eb{\mu^\eps(S)^2} \\
&= T^2 \, \Eb{\mu^\eps([0,2^{-k+2}]^2)^2}\,,
\end{align}
by translation invariance of the measure $\mu^\eps$ in the plane. Now, from item $(ii)$ in Proposition \ref{th.mu}, we have 
\begin{align}\label{}
\Pb{|\PPP_T \cap S| \geq 2 } & \le O(1) T^2\,  2^{- 11k/4}\,.
\end{align}
By union bound, we thus obtain 
\begin{align}\label{}
\Pb{\inf_{i\neq j} \{ |x_i-x_j| \} <2^{-k}} & \leq O(1) T^2\, 2^{-3k/4}\,,
\end{align}
where $O(1)$ depends only on the size of $D$ as well as on $\eps>0$. This gives us the following estimate on the lower tail of the random variable $\rho:=\inf_{i\neq j} \{ |x_i-x_j| \}$. There is a constant $c=c(D,\eps,T)>0$ such that 
\[
\Pb{\rho<r}< c\, r^{3/4}\,.
\]

\item The third property $(iii)$ is proved in the same way.

\item For the fourth property $(iv)$, since $\QUAD_\N$ is countable, it is enough to check the property on any fixed quad $Q\in \QUAD_\N$. For such a quad $Q$,
there is a constant $K=K(U)<\infty$ such that for any $0<r<1$, if $U_r$ is the $r$-neighbourhood of $\p Q$, then $|U_r|< K  r$.  Now from item $(i)$ in Proposition \ref{th.mu}, one has $\Eb{\mu^\eps(U_r)}< C_\eps\,K\, r$ which readily implies 
\[
\Pb{\dist(\p Q, \PPP_T)<r} < 1-e^{-C_\eps\,K\, r\,T}<C_\eps\,K\,r\,T
\]
\ei
\qed


\section{Networks associated to marked percolation configurations}\label{s.network}

In this section, we will fix some quad $Q\in \QUAD_\N$. 
Roughly speaking, our goal in this section will be to associate to any configuration $\omega\in \HH$ and any \textbf{finite} set of points $X=\{x_1,\ldots, x_p\} \subset D$ a combinatorial object that we will call a \textbf{network} 
which will be designed in a such a way that it will represent the connectivity properties of the configurations $\omega$ within the domain $Q\setminus \{x_1,\ldots, x_n\}$.
 This network denoted by $\Net_Q=\Net_Q(\omega,X)$ will be a certain graph on the set of vertices $V= X \cup \{ \p_1, \p_2,\p_3, \p_4 \}$ with two types of edges connecting these vertices: \textbf{primal}  and \textbf{dual} edges.

Let us start with a formal definition of our combinatorial object (network), which for the moment will not depend on a configuration $\omega\in \HH$. 

\subsection{Formal definition of network}


\begin{definition}[\textbf{Network}]\label{d.network}
Suppose we are given a polygonal quad $Q\in \QUAD_\N=\QUAD_\N(D)$ and a finite set of points  $X=\{x_1,\ldots,x_p\}\subset D$.
(This subset $X$ will later correspond to the set of pivotal points in $\PPP_T$).
We will assume that the points in $X$ are all at positive distance from the boundary $\p Q$. 

A {\bf network} for the pair $(Q,X)$ will be an undirected graph $\Net=(V,E ,\tilde E)$ with vertex set $V= X\cup \{\p_1,\p_2,\p_3,\p_4\}$ and with two types of edges (the primal edges $e\in E$ and the dual edges $\tilde e\in \tilde E$) and which satisfies to the following constraints:
\begin{enumerate}
\item Any edge connected to $\p_1$ and/or $\p_3$ is a primal edge.
\item Any edge connected to $\p_2$ and/or $\p_4$ is a dual edge.
\item There are no multiple edges.
\end{enumerate}
\end{definition}

Here are a few properties on networks that we shall need:

\begin{definition}\label{d.connected}
We will say that a network  $\Net=(V,E ,\tilde E)$ is  {\bf connected} if there is a primal path connecting $\p_1$ to $\p_3$ or a dual path connecting $\p_2$ to $\p_4$.  
\end{definition}

\begin{definition}\label{d.boolean}
We will say that a network  $\Net=(V,E ,\tilde E)$ is  \textbf{Boolean} if for any assignment $\phi : X=\{x_1,\ldots, x_p \} \to \{0,1\}$ which induces a \textbf{site}-percolation on the graph $(V,E,\tilde E)$, the following is satisfied: 
\bi
\item[(i)] Either there is at least one open primal path from $\p_1$ to $\p_3$, (i.e. a primal path from $\p_1 $ to $\p_3$ which only uses sites $x\in X$ for which $\phi(x)=1$) and there are no closed dual paths from $\p_2$ to $\p_4$ (i.e. dual paths from  $\p_2$ to $\p_4$ which only use sites $x\in X$ for which $\phi(x)=0$).
\item[(ii)] Or, there is a closed dual path from $\p_2$ to $\p_4$ and there are no open primal path connecting $\p_1$ to $\p_3$.
\ei
As such one can associate a Boolean function $f_\Net: \{0,1\}^X \to \{0,1\}$ to a Boolean network $\Net$ as follows : for any $\phi \in \{0,1\}^X$, let $f_\Net(\phi):=1$ if we are in case $(i)$ and $0$ otherwise.  
\end{definition}

\begin{remark}\label{}
Note that a Boolean network is necessarily connected. 
\end{remark}

\subsection{Mesoscopic network}

In the previous subsection, we defined a combinatorial structure associated to a finite cloud of points (these points will later correspond to pivotal switches in $\PPP_T(\omega_\infty)$). When one deals with the continuum limit $\omega_\infty \sim \P_\infty$, it is easier to work with mesoscopic squares rather than points.  Whence the following definition.

\begin{definition}\label{d.mnetwork}
Let $Q$ be a fixed quad in $\QUAD_\N$. 
For any dyadic $r>0$ in $2^{-\N}$, and any family of disjoint $r$-squares $B=\{B_1^r,\ldots, B_p^r\}$ where the squares are taken from the grid $r\, \Z^2 - (\frac r 4, \frac r 4)$, a {\bf network} for the pair $(Q,B)=(Q,B^r_1,\ldots,B^r_p)$ will be the same combinatorial structure as in Definition \ref{d.network}, where the set of vertices is replaced here by $B \cup \{\p_1,\p_2,\p_3,\p_4\}$.
\end{definition}

The purpose of the next subsections will be to define a network attached to a cloud of points with the help of a nested sequence of mesoscopic networks, where the dyadic squares will shrink towards the cloud of points. This motivates the following definition. (See also Figure \ref{f.network} for an illustration of this procedure).

\begin{figure}[!htp]
\begin{center}
\includegraphics[width=0.8\textwidth]{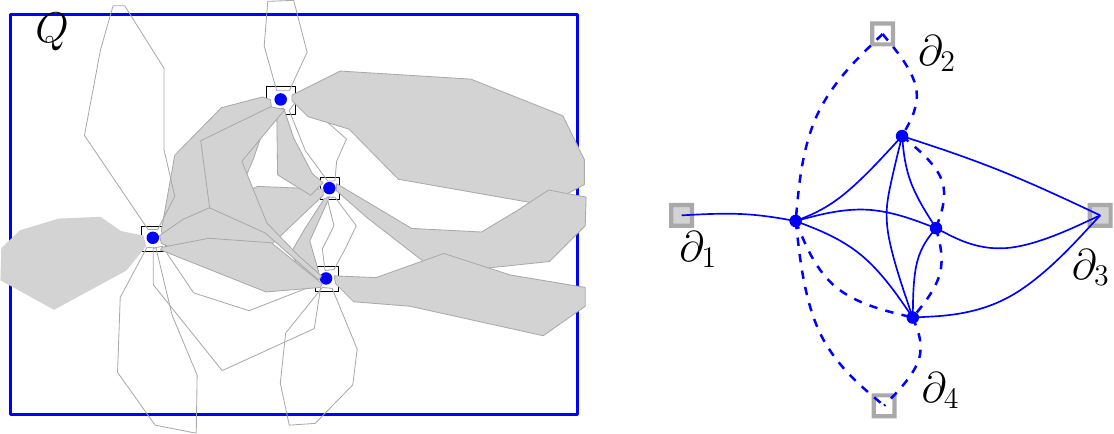}
\end{center}
\caption{This figure illustrates how to extract a network (Definition \ref{d.network}) out of a given configuration $\omega \in \HH$}\label{f.network}
\end{figure}

\begin{definition}[A nested family of dyadic coverings]\label{d.nested}
For any $b>0$ in $2^{-\N}$, let $G_b$ be a disjoint covering of $\R^2$ using $b$-squares of the form $[0,b)^2$ along the lattice $b\, \Z^2$. 
Now for any $r\in 2^{-\N}$ and any finite subset $X=\{x_1,\ldots, x_p\} \subset D$, one can associate uniquely $r$-squares 
$B_{x_1}^r,\ldots,B^r_{x_p}$ in the following manner: for all $1\le i \le p$, there is a unique square $\tilde B_{x_i} \in G_{r/2}$ which contains $x_i$ and we define $B_{x_i}^r$ to be the $r$-square in the grid $r\Z^2 - (r/4,r/4)$ centered around the $r/2$-square $\tilde B_{x_i}$. (This explains the above translation by $(r/4,r/4)$). 
We will denote by $B^r(X)$ this family of $r$-squares. The family of $r$-squares has the following two properties:
\bi
\item[(i)] The points $x_i$ are at distance at least $r/4$ from $\p B_{x_i}^r$. 
\item[(ii)] For any set $X$, $\{ B^r(X) \}_{r\in 2^{-\N}}$ forms a nested family of squares in the sense that for any $r_1<r_2$ in $2^{-\N}$, and any $x\in X$, we have
\[
B^{r_1}_{x} \subset B^{r_2}_{x}\,.
\]
\ei
\end{definition}

\subsection{How to associate a mesoscopic network to a configuration $\omega\in \HH$}
Given a quad $Q\in \QUAD_\N$ and a dyadic positive number $r\in 2^{-\N}$, the purpose of this subsection is to associate in a useful manner a mesoscopic network $\Net_Q^r$ to a finite set $X\subset D$ and a configuration $\omega\in \HH$. In other words, we wish to construct a map $\Net_Q^r : (\omega,X) \mapsto \Net_Q^r(\omega,X)$. 
\medskip

Let us start with a technical definition which quantifies by how much points in $X$ are away from each other and from $\p Q$. 

\begin{definition}\label{d.r*}
Let $Q\in \QUAD_\N$ be a fixed quad and $X=\{x_1,\ldots,x_p\}\subset D$ a finite subset. Let us define the quantity $r^*=r^*(X,Q)>0$ to be the supremum over all $u\geq 0$ such that for any $1\le i \le p$, $x_i$ is at distance at least $10\cdot u$ from the other points $x_j$ and from the boundary $\p Q$. 

In the particular case where the set $X$ is the random set $\PPP_T=\PPP_T(\mu^\eps(\omega_\infty))$ defined earlier in Definition \ref{d.PPP} we will consider, with a slight abuse of notation, the random variable $r^*=r^*(\PPP_T,Q)$. 

It follows from Proposition \ref{pr.PPP} that this random variable $r^*=r^*(\PPP_T,Q)$ is positive a.s. (more precisely one has from Proposition \ref{pr.PPP} that there is a constant $c=c_{Q,D,\eps,T}<\infty$ such that $\Pb{r^*<r}<c\,  r^{3/4}$. This follows form the fact that the contribution $r^{3/4}$ is much larger than the boundary contribution of order $r$). 
\end{definition}

%
%
%
%
%
%

We are now ready to define what a {\it mesoscopic network} is. 

\begin{definition}[Mesoscopic network]\label{d.mesoN}
Let $Q\in \QUAD_\N$ be a fixed quad. And let $X\subset D$ be a finite subset with $r^*(X,Q)>0$.
For any $r \in 2^{-\N}$, 
the $r$-{\bf mesoscopic network}
$\Net_Q^r=\Net_Q^r(\omega,X)$ associated to the set $X$ will be the following network:

\bi
\item The set of vertices of $\Net_Q^r$ will be $B^r(X)\cup \{\p_1,\ldots,\p_4\}$, where 
$B^r(X)$ is the family of $r$-squares defined in Definition \ref{d.nested}. With a slight abuse of notation, we will denote the vertices $B_{x_i}^r$ simply by $x_i$. 
\item If $r\geq r^*(\PPP_T,Q)$, for convenience we define the edge structure of $\Net_Q^r$ to be empty. Otherwise, if $r<r^*(\PPP_T,Q)$, the edge structure is defined as follows:
\item 
The primal edge $e=\<{x_i, x_j}$ will belong to $\Net_Q^{r}(\omega,X)$ if and only if one can find a quad $R$ such that:
\bi
\item $\p_1 R$ and $\p_3 R$ remain respectively strictly inside $B_{x_i}^r$ and $B_{x_j}^r$ \item $R$ remains strictly away from the squares $B_{x_k}^{r}, k \notin \{i,j\}$ as well as $r$-away from the boundary $\p Q$
\item  $\omega\in \boxminus_R$. 
\ei

\item The dual edge $\tilde e=\<{x_i,x_j}$ will belong to $\Net_Q^{r}(\omega,X)$ if and only if one can find a quad $R$ satisfying the same conditions as in the above item, except $\omega \in \boxminus_{\rot{R}}^c$.  

\item The primal edge $e=\<{\p_1,\p_3}$ will belong to $\Net_Q^{r}(\omega,X)$
if and only if one can find a quad $R$ which is larger than any quad in $B_{d_\Quad}(Q,r)$, which remains strictly away from all squares $B_{x_i}^r$ and for which $\omega\in \boxminus_R$. 
\item Idem for the dual edge $\tilde e = \<{\p_2, \p_4}$. 
\item The edge $e=\<{\p_1, x_i}$ will belong to $\Net_Q^r$ if and only if one can find a quad $R$ for which $\p_3 R$ remains inside $B_{x_i}^r$, for which $\p_1 R$ remains $r$-outside of $Q$, for which $R$ remains strictly away from the squares $B_{x_k}^r, k\neq i$ and $r$-away from $\p_2,\p_3,\p_4$ and for which $\omega\in \boxminus_R$. 
 (Analogous conditions for the edges $\<{x_j,\p_2}$ and their dual siblings). 
\ei  
Note that $\Net^r_Q$ is defined in such a way that for any combinatorial network $\Net\neq \emptyset$, the event $\{\Net_Q^r=\Net\}$ is an open set of the quad-topology $\T$ (this requires all the above conditions to be ``strict'' conditions as opposed to conditions of the type $\geq r$-away from the boundary).  
\end{definition}

It is easy to check that the above conditions define a mesoscopic network in the sense of Definition \ref{d.mnetwork}.

\subsection{Comparison of $\Net_Q^r(\omega_\infty)$ and $\Net_Q^r(\omega_\eta)$}\label{ss.comparison}

One has the following Proposition:
\begin{proposition}\label{p.discrete}
Assume $(\omega_\eta)_{\eta>0}$ and $\omega_\infty$ are coupled together so that 
$\omega_\eta \sim \P_\eta$ converge pointwise towards $\omega_\infty\sim \P_\infty$. 
For any $Q\in \QUAD_\N$, and for any subset $X\subset D$ with $r^*(X,Q)>0$, we have for all $r< r^*(X,Q)$ in $2^{-\N}$:
\begin{equation}\label{}
 \Pb{\Net_Q^r(\omega_\eta,X)=\Net_Q^r(\omega_\infty,X)} \to 1\,,
\end{equation}
as $\eta\to 0$.
\end{proposition}

\begin{corollary}\label{c.boolean}
In particular, for almost all $\omega_\infty\sim \P_\infty$, 
all the networks $\Net_Q^r(\omega_\infty,X)$ with $X\subset D$ and $0<r\in2^{-\N}<r^*(X,Q)$ are Boolean networks. 
(Note that this was not obvious at all to start with, since it is easy to construct points $\omega \in \HH$ which do not correspond to planar percolation configurations). 
\end{corollary}

\ni
{\bf Proof:}
\ni
This is proved in the same fashion as equation \ref{e.Delta} in Lemma \ref{l.meas4arm} (see \cite{\GPSa}), namely one direction uses the fact that $\{ \Net_Q^r = \Net\}$ is an open set, as mentioned above. The other direction is in fact easier than ~\eqref{e.Delta} which is proved in \cite{\GPSa} since one does not require the quads for different edges to be disjoint. Hence it is only a matter of controlling boundary effects as in the proof of Corollary 5.2 in \cite{\SchrammSmirnovNoise} (stated in Theorem \ref{th.SScardy}). \qed

%

\subsection{Almost sure stabilization as $r\to 0$ of the $r$-mesoscopic networks}

Let $\PPP_T = \PPP_T(\mu^\eps) \subset D$ be sampled according to the intensity measure $d\mu^\eps(x) \times 1_{[0,T]}dt$. As we have seen in Section \ref{s.poisson}, $\PPP_T$ is almost surely finite and is such that for all quads $Q\in \QUAD_\N$, one has a.s. $r^*=r^*(\PPP_T,Q)>0$ (see Definition \ref{d.r*}). 

\begin{definition}\label{}
For any $r\in 2^{-\N}$, we will denote by $\Net_Q^r(\omega_\infty, \PPP_T)$ the $r$-mesoscopic network associated to $\PPP_T$ as defined in Definition \ref{d.mesoN}. 
\end{definition}

We start with the following lemma on the measurability of this network.

\begin{lemma}\label{l.mesNET}
The network  $\Net^r_Q(\omega_\infty,\PPP_T)$ is measurable w.r.t. $(\omega_\infty, \PPP_T)$.
\end{lemma}

\ni
{\bf Proof:}
\ni
First of all, the scale $r^*(\PPP_T, Q)$ is measurable with respect to $\PPP_T$ as a deterministic geometric quantity which depends on the finite cloud $\PPP_T$. This implies that the event $\{ r\geq r^*\} \subset \{ \Net_Q^r =\emptyset \}$ is measurable (where by $\Net_Q^r=\emptyset$ we mean that the edge structure of $\Net_Q^r$ is empty).

In what follows $r\in 2^{-\N}$ is fixed. We already dealt with the case $\{ r\geq r^*\}$. Now, on the event $\{r< r^*\}$, notice that since $D$ is bounded, there are finitely many possible vertex sets $B^r(X), X\subset D$ since these consists in families of dyadic squares and the constraints $r^*(X,Q)>r$ imply that $|X|$ is bounded. In other words, the map $X\subset D \mapsto B^r(X) 1_{\{r^*(X,Q)>r\}}$ is piecewise constant and its image is a finite set. Letting $B^r$ be an arbitrary such family of $r$-squares. 
First, let us notice that the event $\{ B^r(\PPP_T) = B^r\}$ is clearly measurable w.r.t. $\PPP_T$. 
Now, let $E$ be any edge structure on $B^r \cup \{ \p_1,\ldots, \p_4 \}$, it follows easily from our definition of the edge structure of $\Net_Q^r$ that on the event $\{ B^r(\PPP_T)=B^r\}\cap \{r<r^*\}$, the event $\{ \Net_Q^r =(B^r,E)\}$ is measurable w.r.t. the percolation configuration $\omega_\infty$. \qed


We wish to prove that almost surely the random sequence of networks $\Net_Q^r(\omega_\infty,\PPP_T)$ is eventually stationary as $r\searrow 0$. 

For this, it will be enough to prove the following lemma.

\begin{lemma}\label{l.stab}
Let $Q$ be a fixed quad in $\QUAD_\N$. Let $T>0$ and $\eps>0$ be fixed. 
There is a constant $A=A_Q<\infty$ 
such that 
for any $\bar r>0$ and any $r\in 2^{-\N}$ with $r \leq \bar r$,

\begin{align}\label{e.UBstab}
\Pb{\Net^{r/2}_Q(\omega_\infty,\PPP_T) \neq \Net^{r}_Q(\omega_\infty,\PPP_T)\,, r^*(\PPP_T, Q) >\bar r }
& \le A_Q \, \frac {r^{1/3}}{\bar r^{4/3}}\,.
\end{align}
Note that the constant $A_Q$ only depends on the shape of $Q$ and not on $T$ nor on $\eps$. 
\end{lemma}


Indeed the Lemma implies the following result.
\begin{theorem}\label{th.BC}
There is a measurable scale $r_Q=r_Q(\omega_\infty, \PPP_T(\mu^\eps))\in 2^{-\N}$ which satisfies 

\begin{equation}\label{}
\left
\lbrace
\begin{array}{l}
0< r_Q<r^*(\omega_\infty, \PPP_T) \; a.s. \\
\Net_Q^r(\omega_\infty,\PPP_T) = \Net^{r_Q}(\omega_\infty,\PPP_T)\;\; \forall r<r_Q\, \in 2^{-\N}
\end{array}
\right.
\end{equation}

In particular, the $r$-mesoscopic network $\Net_Q^r(\omega_\infty,\PPP_T)$ a.s. has a limit as $r\to 0$, which is measurable w.r.t. $(\omega_\infty,\PPP_T)$ and which we will denote by $\Net_Q=\Net_Q(\omega_\infty,\PPP_T)$.
\end{theorem}

Let us start by deriving the theorem using Lemma \ref{l.stab}. This is a straightforward use of Borel-Cantelli. 
For any $k\in \N$ and any $\bar k\in \N$, define the event 
\begin{align*}\label{}
A_k^{\bar k}:= \{ \Net^{2^{-\bar k - k-1}}_Q(\omega_\infty,\PPP_T) \neq \Net^{2^{-\bar k - k}}_Q(\omega_\infty,\PPP_T)\,, r^*(\PPP_T, Q) >2^{-\bar k} \}\,.
\end{align*}
Lemma \ref{l.stab} implies that for any $\bar k \in \N$, $\sum_{k\geq 0} \Pb{A_k^{\bar k}} <\infty$. 
Therefore by Borel-Cantelli, there is a measurable $K=K(\bar k, \omega_\infty, \PPP_T)<\infty$ a.s.  such that $A_{K-1}^{\bar k}$ is satisfied and all $A_k^{\bar k}, k\geq K$ are not (we define here $A^{\bar k}_{-1}$ to be of full measure). 
Let $\bar k_0:= \inf\{ \bar k \in \N; 2^{-\bar k} \leq r^{*}(\PPP_T,Q) \}$ which is measurable w.r.t. $\PPP_T$ and is $<\infty$ a.s. (see Definition \ref{d.r*}). With our above notations, the scale 
\[
r_Q:= 2^{-\bar k_0 - K(\bar k_0, \omega_\infty, \PPP_T)}
\]
satisfies the desired conditions of Theorem \ref{th.BC}.  \qed


\medskip
We now turn to the proof of Lemma \ref{l.stab} which will use in a crucial manner the comparison with the discrete setting $\omega_\eta \sim \P_\eta$ established in Proposition \ref{p.discrete}. 
\medskip

\ni
{\bf Proof of Lemma \ref{l.stab}:}
Let us fix $\bar r>0$ and $r<\bar r \in 2^{-\N}$. Since the upper bound in~\eqref{e.UBstab} relies on arm exponents, it will be convenient if not necessary to rely on the discrete setting. In order to work with an actual discrete mesh $\eta \Tg$, we will thus rely on Proposition \ref{p.discrete}. 
\smallskip

As in the proof of Lemma \ref{l.mesNET}, we use the fact that on the event
$\{ r^* >\bar r\}$, there are finitely many possible vertex sets for $\Net_Q^{r/2}(\omega_\infty, \PPP_T)$. In other words, the range of families of $r/2$-squares obtained as $B^{r/2}(\PPP_T)$ with $\PPP_T$ such that $r^*(\PPP_T,Q)>\bar r$ is a finite set $\Gamma$. We will denote by $V = \{B^{r/2}_1,\ldots, B^{r/2}_{N_V}\}$ the elements of $\Gamma$. Note that each $V\in \Gamma$ needs to satisfy $r^*(V,Q)>\bar r$ (with the obvious extension of the quantity $r^*$ defined in Defintion \ref{d.r*} to family of $r$-squares). 
\smallskip

As in Proposition \ref{p.discrete}, let us consider a coupling of $(\omega_\eta)_{\eta>0}$ and $\omega_\infty$ so that $\omega_\eta \sim \P_\eta$ and $\omega_\eta$ converges pointwise in $(\HH,d_\HH)$ towards $\omega_\infty\sim \P_\infty$. Since $\PPP_T$ is a Poisson point process with intensity measure $\mu^\eps(\omega_\infty)$, we have that a.s. $\PPP_T \subset \Piv^\eps(\omega_\infty)$.  In particular if $B^{r/2}(\PPP_T)= V= \{ B^{r/2}_1,\ldots, B^{r/2}_{N_V}\} \in \Gamma$ and if for all $i\in [1,N_V]$, we denote by $A_i^\eps$ the annulus centered around the $r/2$-square $B^{r/2}_i$ of exterior  side-length $\eps$, then $\omega_\infty$ must satisfy a four-arm event in each annuli $A^\eps_i$. Let us then introduce the following event
\[
\A_V=\A_V^\eps = \bigcap_{1\leq i \leq N_V} \A_4(A_i^\eps)\,,
\]
where the events $\A_4(A_i^\eps)$ denote the four-arms events in each annuli $A_i^\eps$. As such, on the event $\{ B^{r/2}(\PPP_T)=V\}$, we must have  $\omega_\infty \in \A_V^\eps$. 

Lemma \ref{l.meas4arm} implies  that for any $V\in \Gamma$
\begin{align}\label{e.DDD1}
\Pb{\omega_\eta \in \A_V^\eps  \md B^{r/2}(\PPP_T(\omega_\infty))=V } \underset{\eta \to 0}{\longrightarrow} 1
\end{align}

Furthermore using Proposition \ref{p.discrete} we have that for any $V\in \Gamma$, 

\begin{align}\label{}
\lim_{\eta\to 0}\Pb{\Net_Q^{r/2}(\omega_\eta, V) = \Net_Q^{r/2}(\omega_\infty, V) } 
= \lim_{\eta\to 0}\Pb{\Net_Q^{r}(\omega_\eta, V) = \Net_Q^{r}(\omega_\infty, V) } =1 \nn
\end{align}
This asymptotic equality enables us to write 
\begin{align}\label{e.DDD3}
\Pb{\Net^{r/2}_Q(\omega_\infty, \PPP_T)\neq \Net^{r}_Q(\omega_\infty, \PPP_T), \bar r < r^*} & \nn \\
& \hskip - 7 cm = \sum_{V\in \Gamma} 
\Pb{\Net^{r/2}_Q(\omega_\infty, \PPP_T)\neq \Net^{r}_Q(\omega_\infty, \PPP_T), B^{r/2}(\PPP_T)=V} \nn \\
& \hskip - 7 cm = \sum_{V\in \Gamma}  \lim_{\eta\to 0} \Pb{\Net^{r/2}_Q(\omega_\eta, V)\neq \Net^{r}_Q(\omega_\eta, V), B^{r/2}(\PPP_T) = V} \nn \\
& \hskip - 7 cm = \sum_{V\in \Gamma}  \lim_{\eta\to 0} \Pb{\Net^{r/2}_Q(\omega_\eta, V)\neq \Net^{r}_Q(\omega_\eta, V),  B^{r/2}(\PPP_T) = V,  \omega_\eta \in \A_V^\eps} \nn \\
& \hskip - 7 cm \le \lim_{\eta\to 0} \Pb{\exists V \in \Gamma \text{ s.t. } \, \Net^{r/2}_Q(\omega_\eta, V)\neq \Net^{r}_Q(\omega_\eta, V)\text{ and }  \omega_\eta \in \A_V^\eps} \,,
\end{align}
where in the third equality we used equation~\eqref{e.DDD1} and for the last inequality we used the fact that all the events $\{ \Net^{r/2}_Q(\omega_\eta, V)\neq \Net^{r}_Q(\omega_\eta, V),  B^{r/2}(\PPP_T) = V,  \omega_\eta \in \A_V^\eps \}$ where $V$ ranges over the set $\Gamma$ are mutually disjoint. Note that in the last probability, one does not average anymore over $(\omega_\infty, \PPP_T)$. 
\medskip

It remains to bound~\eqref{e.DDD3}. 
For this matter, let us analyse the event 
\[
W=W_\eta:= \{ \omega_\eta\in \HH \text{ s.t. }\exists V \in \Gamma \text{ s.t. } \, \Net^{r/2}_Q(\omega_\eta, V)\neq \Net^{r}_Q(\omega_\eta, V)\text{ and }  \omega_\eta \in \A_V^\eps \}\,.
\]
If $\omega_\eta \in W$, then one can find a set of $r/2$-dyadic squares $V=\{ B^{r/2}_1, \ldots, B^{r/2}_{N_V}\}\in \Gamma$ which satisfies the above property. Let us simplify the notation and write instead $V=\{B_1,\ldots, B_N\}$. Let us collect what $V$ and $\omega_\eta$ need to satisfy. From the definition of $r^*$ (Definition \ref{d.r*}) and since $V\in \Gamma$, recall that 
\begin{align}\label{e.C1}
\begin{cases}
\text{for any } i\neq j \in [1,N],\, \dist(B_i,B_j)>5\, r^* > 5\bar r  \\
\text{for any }i \in [1,N],\, \dist(B_i, \p Q) > 5\, r^* > 5 \bar r 
\end{cases}
\end{align}

Furthermore since $\omega_\eta \in W$, and since we assumed the set $V$ realised that event, we have
\begin{align}\label{e.C2}
\begin{cases}
\Net^{r/2}_Q(\omega_\eta, V)\neq \Net^{r}_Q(\omega_\eta, V) \\
\text{for any }i \in [1,N],\, \omega_\eta \in \A_4(A_i^\eps) \subset \A_4(A_i(r/2, \bar r))
\end{cases}\,,
\end{align}
where $A_i(r/2, \bar r)$ is the annulus centred around the $r/2$-square $B_i$ of exterior radius $\bar r$. 
Let us analyse what may happen in order to cause $\Net^{r/2}_Q(\omega_\eta, V)\neq \Net^{r}_Q(\omega_\eta, V)$. For this difference to happen, at least one edge (dual or primal) must differ in the edge structure of $\Net_Q^{r/2}(\omega_\eta, V)$ and $\Net_Q^r(\omega_\eta, V)$. These two networks share the same vertex set, namely $V \cup \{ \p_1,\ldots, \p_4 \}$. Without loss of generality, let us assume that there is at least one primal edge which differs (the dual scenario being treated similarly). There are three types of such primal edges: bulk to bulk ($B_i$ to $B_j$), boundary to bulk ($\p_1$ or $\p_3$ to $B_i$) and boundary to boundary ($\p_1$ to $\p_3$). We will analyse in detail only the first case: bulk to bulk. The two other cases are analysed similarly by taking care of boundary issues. 
\smallskip

\ni
\underline{Bulk to bulk case:}

\ni
Let us assume that a difference occurs on the edge $e=\<{B_i,B_j}$ for some $i\neq j\in [1,N]$. 
I.e. $\omega_\eta \in \{ e=\<{i,j} \in \Net^{r/2}\}\Delta \{e=\<{i,j}\in \Net^r\}$. We still need to distinguish two cases.
\bi
\item Suppose we are in the case where $e\in \Net^r$ and $e\notin \Net^{r/2}$. 
From our Definition \ref{d.mesoN} of $\Net^{u}_Q(\omega,V)$, this implies that one can find a quad $R$ which strictly avoids all $B_k^r, k\notin \{i,j\}$ (as well as the $r$-neighborhood of $\p Q$) and it connects $B_i^r$ with $B_j^r$. Here $B_i^r$ denote the $r$-square which surrounds $B_i=B_i^{r/2}$, in particular $B_i^r \supset B_i^{r/2}$. Since we are on a discrete mesh, this is the same as finding an open path which connects $B_i^r$ and $B_j^r$ and avoids all $B_k^r$. Since $e\notin \Net^{r/2}(\omega_\eta, V)$ and since this is easier to avoid all $B_k, k\notin \{i,j\}$ (as well as the $r/2$-neighbourhood of $\p Q$), it necessarily implies that the open cluster in $\omega_\eta$ which connects $B_i^r$ with $B_j^r$ does not connect $B_i^{r/2}$ with $B_j^{r/2}$. Let us call this cluster $C_{i,j}$. Let us suppose without loss of generality that the loss of connection happens nearby $B_i$. Recall that $\omega_\eta\in \A_4(A_i(r/2,\bar r))$ which means that there are two disjoint open clusters connecting $B_i^{r/2}$ to the square $B_i^{\bar r}$ of radius $\bar r$ centered around $B_i$. Let us call these clusters $C_1$ and $C_2$. To be more precise these two clusters are disjoint only if restricted inside $B_i^{\bar r}$ and $C_1$ and $C_2$ should denote these clusters inside $B_i^{\bar r}$. The three clusters $C_1,C_2,C_{i,j}$ must be disjoint inside $B_i^{\bar r}$ in order that $e \notin \Net^{r/2}$. This implies that in the present case, $\omega_\eta$ needs to satisfy a 6-arm event in the annulus $A_i(r,\bar r)$.

\item Suppose on the other hand that $e\in \Net^{r/2}$ and $e\notin \Net^{r}$. The situation in this case is quite different. Indeed since $e\in \Net^{r/2}$, one can find an open path which connects $B_i^{r/2}$ with $B_j^{r/2}$ and it avoids all $B_k^{r/2}$ as well as the $r/2$-neighborhood of $\p Q$. Let us call $C_{i,j}$ the cluster (not the path) in $\omega_\eta$ which connects these two squares. Since $e\notin \Net^r$ and since besides the restriction coming from $\p Q$ and the other squares $B_k, k\notin \{i,j\}$ it is easier to connect $B_i^r$ with $B_j^r$, it means that at least one of the following scenarios must happen:
\bnum
\item Either there is at least one $k\notin \{i,j\}$ such that restricted on the $\bar r$-square $B_k^{\bar r}$, the cluster $C_{i,j}$ minus the square $B_k^r$ is disconnected. 
\item Or, the cluster $C_{i,j}$ restricted in $Q^{(r)}:=\{x\in Q, \dist(x, \p Q)\geq r\}$ gets disconnected. 
\enum
Let us analyze the first scenario. Since $C_{i,j}$ minus $B_k^r$ gets disconnected in $B_k^{\bar r}$, this means that there is at least one dual cluster $\tilde C$ connecting $B_k^r$ to $B_k^{\bar r}$. Now recall that $\omega_\eta\in \A_4(A_k(r/2, \bar r))$. This means in particular that there are at least two disjoint dual clusters $\tilde C_1$ and $\tilde C_2$ which connect $B_k^{r/2}$ to $B_k^{\bar r}$. We claim that inside the annulus $B_k^{\bar r}\setminus B_k^{r/2}$, the dual clusters $\tilde C, \tilde C_1$ and $\tilde C_2$ must be disjoint. Indeed if one had for example $\tilde C =\tilde C_1$, this would force the cluster $C_{i,j}$ to pass through $B_k^{r/2}$ I.e. restricted on $B_k^{\bar r}$, the cluster $C_{i,j}$ minus $B_k^{r/2}$ would be disconnected as well which would prevent the existence of a path within $C_{i,j}$ which would connect $B_i^{r/2}$ to $B_j^{r/2}$ away from $B_k^{r/2}$. Hence $\tilde C, \tilde C_1$ and $\tilde C_2$ are three disjoint dual clusters. In particular in the present case, $\omega_\eta$ needs to satisfy a 6-arm event in the annulus $A_k(r, \bar r)$.
\medskip

The second scenario is easier to analyze. Let us cover the $r$-neighborhood of $\p Q$ with $N_r=O(1/r)$ $r$-squares $S_1,\ldots, S_{N_r}$. It is easy to check that there must be at least one square $S_k$ such that inside $Q$, $C_{i,j}\setminus S_k$ gets disconnected (proceed for example by removing squares $S_i$ one at a time until $C_{i,j}$ gets disconnected). This can happen only if there is a three arm event for $\omega_\eta$ in $( S_k^{\bar r} \setminus S_k^r ) \cap Q$ (where at this point it should be clear what the notation $S_k^{\bar r}$ stands for). 
\ei
Let us summarize the Bulk to bulk case as follows. If $\omega_\eta \in W$ and falls in this bulk to bulk situation, then 
\bi
\item[(i)] Either one may find a dyadic $r$-square $B_i^r$ inside $Q$ such that $\omega_\eta$ satisfies a 6 arm events in $A_i(r,\bar r)$. 
\item[(ii)] Or one may find a dyadic $r$-square along $\p Q$ for which $\omega_\eta$ satisfies a boundary three-arm event up to distance $\bar r$. 
\ei
The two other cases (boundary to bulk and boundary to boundary) would reach to the same conclusion. 
We will see below that the probability of $(i)\cup (ii)$ is small.

%

Let us then introduce the following event:
\begin{align*}\label{}
G_{r,\bar r} := & \{\omega \in \HH\text{ s.t. } \exists \text{ a dyadic $r$-square } B \text{ s.t. } \omega\in \A_6(B^{\bar r}\setminus B) \} \\
& \cup  
 \{\omega \in \HH\text{ s.t. } \exists \text{ a dyadic $r$-square } B \text{ along $\p Q$ s.t. }\omega \in \A_3((B^{\bar r}\setminus B)\cap Q)\} \\
 := & G_{r, \bar r}^{\mathrm{bulk}} \cup G_{r,\bar r}^{\mathrm{boundary}}\,.
\end{align*}

Indeed, using \ref{e.DDD3} as well as the above analysis, we end up with the following upper bound
\begin{align}\label{}
\Pb{\Net^{r/2}_Q(\omega_\infty, \PPP_T)\neq \Net^{r}_Q(\omega_\infty, \PPP_T), \bar r < r^*}  & \\
& \hskip -5 cm \le \lim_{\eta\to 0} \left( \Pb{\omega_\eta \in  G_{r, \bar r}^{\mathrm{bulk}} } + \Pb{ \omega_\eta \in G_{r, \bar r}^{\mathrm{boundary}}}\right)
\end{align}

\medskip
Let us start with the first term:
\begin{align}\label{}
 \lim_{\eta\to 0} \Pb{\omega_\eta \in  G_{r, \bar r}^{\mathrm{bulk}} } & \le 
 O(r^{-2}) \limsup_{\eta \to 0} \alpha_6^\eta(r,\bar r) \nn \\
 & \le B_Q\, r^{-2}\, (r/\bar r)^{35/12}\,,
\end{align}
using the 6-arm exponent for the triangular lattice obtained in \cite{\SmirnovWerner}. 
\medskip

The second term may be handled as follows: without loss of generality we may assume that $\bar r$ is smaller than $2^{-k}$ where $k$ is such that our quad $Q\in \QUAD_\N$ is in $\QUAD^k$. Following the above notations, there are $O(1/r)$ squares $r$-squares $S_k$ along $\p Q$. The squares for which the probability of a three arm event is the higher are squares near corners of $\p Q$ which have a angular sector inside $Q$ of angle $3\pi /2$. Again using \cite{\SmirnovWerner}, we have that the limsup as $\eta\to 0$ of having a three arm event in such a sector from $r$ to $\bar r$ is bounded above by $C\, (r/\bar r)^{4/3}$. It thus follows that 
\begin{align}\label{}
 \lim_{\eta\to 0} \Pb{\omega_\eta \in  G_{r, \bar r}^{\mathrm{boundary}} } & \le 
 O(r^{-1}) \limsup_{\eta \to 0} \alpha_3^{+,\eta}(r,\bar r, \theta=3\pi/2) \nn \\
 & \le C_Q\, r^{-1}\, (r/\bar r)^{4/3}\,,
\end{align}

All together, we obtain the upper bound 
\begin{align*}\label{}
 \lim_{\eta\to 0} \Pb{\omega_\eta \in  G_{r, \bar r}^{\mathrm{bulk}} } & \le 
 B_Q\, r^{-2}\, (r/\bar r)^{35/12} + C_Q\, r^{-1}\, (r/\bar r)^{4/3} \\
 &\le A_Q \; \frac {r^{1/3}}{\bar r^{4/3}}\,,
\end{align*}
which ends the proof of Lemma \ref{l.stab}.
\qed

Conditioned on $(\omega_\infty, \PPP_T)$, we are now able for each quad $Q\in \QUAD_\N$, to associate in a measurable manner a combinatorial structure, the network $\Net_Q(\omega_\infty, \PPP_T)$, whose aim is to represent how the points in $\PPP_T$ are linked together via $\omega_\infty$ inside $Q$. The purpose of the next section is to use this combinatorial graph in order to define a \cadlag process with values in $\HH$.

\section{Construction of a continuum $\eps$-cut-off dynamics}\label{s.cutoff}

Let $\eps>0$ be a fixed cut-off parameter and $T>0$ be some fixed time-range. We wish to construct a certain cut-off dynamics $\omega_\infty^\eps(t)$ on $\HH$ which will be shown (in Subsection 
\ref{ss.convinSK}) to be the scaling limit in the Skorohod space $\Sk_T$ (see Section \ref{s.skorohod}) of the discrete cut-off dynamics
$\omega_\eta^\eps(t)$.  
The present section is divided as follows: the first five subsections deal with dynamical percolation. More precisely in Subsection \ref{ss.epsdefine} we define the stochastic \cadlag process $t\in [0,T]\mapsto \omega_\infty^\eps(t)$ and in Subsections \ref{ss.remainsH} to \ref{ss.convinSK}, we prove that 
a.s. the process $\omega_\infty^\eps(t)$ remains in $\HH$ for all times $t\in [0,T]$ and is the limit in the Skorohod space $\Sk_T$ of the \cadlag process $t\mapsto \omega_\eta^\eps(t)$. Finally in Subsection \ref{ss.DNC}, we extend the construction to the near-critical case in order to build an $\eps$-cutoff near-critical process $\lambda \in [-L,L]\mapsto \omega_\infty^{\nc,\eps}(\lambda)$. 

\subsection{Construction of the processes}\label{ss.epsdefine}

If one considers the well-defined dynamics $t\mapsto \omega_\eta(t)$, the evolution on the time interval $[0,T]$ of the state of sites which are initially in $\Piv^\eps(\omega_\eta(t=0))$ is governed by the Poisson point process $\PPP_T(\mu^\eps(\omega_\eta))$. When such sites are updated, they become open (or remain open if they were already so) with probability $1/2$ and closed with probability $1/2$. Let us then decompose $\PPP_T$ as follows:
\[
\PPP_T = \PPP_T^+ \cup \PPP_T^-\,,
\]
where $\PPP_T^+$ ($\PPP_T^-$) is the Poisson cloud of points which turn open (closed). 
It is easy to check that $\PPP_T^{\pm}$ are two independent copies of a Poisson point process of intensity measure $\frac 1 2 d\mu^\eps(\omega_\eta)(x) \times 1_{[0,T]}dt$. This motivates the following definition

\begin{definition}[Discrete cut-off dynamics]\label{d.dcutoff}
For any fixed $\eps>0$ and $T>0$, one defines a \cadlag process $t\mapsto \omega_\eta^\eps(t)$ by starting at the initial time $t=0$ with $\omega_\eta^\eps(t=0)=\omega_\eta \sim \P_\eta$ and by updating this initial configuration as time $t\in[0,T]$ increases using the above Poisson point process $\PPP_T(\mu_\eta^\eps(\omega_\eta))=\PPP_T^+\cup \PPP_T^-$ in the obvious manner. 
\end{definition}

Back to the continuum case, 
let us sample $\omega_\infty \sim \P_\infty$ as well as two independent Poisson point processes  $\PPP_T^+$ and $\PPP_T^-$ both with same intensity measure $\frac 1 2 d\mu^\eps(\omega_\infty)(x) \times 1_{[0,T]}dt$. To be consistent with the previous sections, we will still call $\PPP_T= \PPP_T^+ \cup \PPP_T^-$.
From these two sources of randomness, we wish to build, similarly as in Definition \ref{d.dcutoff}, a \cadlag trajectory $t\in [0,T] \mapsto \omega_\infty^\eps(t)\in \HH$.
\medskip

Since the family of quads $\QUAD_\N$ defined earlier in \ref{d.kquads} is dense in $\QUAD_D$, by Theorem \ref{th.compact}, it should in principle be sufficient to define the trajectory $\omega_\infty^\eps(t)$ through the processes (or projections)  $t\mapsto 1_{\boxminus_{Q}}(\omega_\infty^\eps(t))$ for all quads $Q\in \QUAD_\N$.
In the present subsection, we will construct such \cadlag processes simultaneously for all quads $Q\in \QUAD_\N$. We will denote these 0-1 valued processes  by $t\mapsto Z_Q(t)=Z_Q(t,\omega_\infty,\PPP_T)\in \{0,1\}$. (Note that $\eps$ is implicit in the later expression via $\PPP_T$ which depends on the regularization $\eps$). 
The aim of the next subsection will be to show that these projected processes a.s. characterize uniquely a well defined \cadlag process $t \mapsto \omega_\infty^\eps \in \HH$. In other words, we will show that a.s. there exists a unique \cadlag process $t\mapsto \omega_\infty^\eps(t)\in \HH$ such that for all $Q\in \QUAD_\N$, and all 
$t\in [0,T]$, one has 
\[
1_{\boxminus_Q}(\omega_\infty^\eps(t)) = Z_Q(t) \in \{0,1\}\,.
\]

\begin{definition}\label{d.ZQ}
Let $T>0$ and $\eps>0$ be fixed. 
For any $q\in \Q\cap [0,T]$, let $\PPP_q$ be the Poisson point process $\PPP_q(\omega_\infty, \mu^\eps(\omega_\infty))$ defined in Definition \ref{d.PPP}. Note that since $q\leq T$, $\PPP_q$ can be obtained simply by projecting $\PPP_T$ on $1_{\bar D}(x) \times 1_{[0,q]}(t)$.
From Theorem \ref{th.BC}, for each $Q\in \QUAD_\N$, we have at our disposal the network 
$\Net_{Q,q}=\Net_Q(\omega_\infty, \PPP_q)$. Let us define the events 
\begin{align}\label{}
V_1&:= \{ \forall q\in \Q\cap [0,T] \text{ and } \forall Q\in \QUAD_\N,\, \text{the network }\Net_{Q,q}\text{ is  Boolean} \} \\
V_2&:= \{ |\switch_T| <\infty \}\,,
\end{align}
where recall that $\switch_T$ denotes the set of switching times defined in Definition \ref{d.PPP}.
Since $\Q$ and $\QUAD_\N$ are countable, it follows from Corollary \ref{c.boolean} that $\Pb{V_1}=1$. It also follows from Proposition \ref{pr.PPP} that $\Pb{V_2}=1$. As such if $V:=V_1\cap V_2$, one has $\Pb{V}=1$.  
On this event $V$, we denote by $f_{Q,q}$ the  Boolean functions associated to the Boolean networks $\Net_{Q,q}$
 (as in Definition \ref{d.boolean}) and still on the event $V$, for each $Q\in \QUAD_\N$ and each $q\in \Q\cap [0,T]$, we define
\[
Z_Q(q):= f_{Q,q}(\phi_q)\,,
\]
where $\phi_q \in \{0,1\}^{\PPP_q}$ is just the configuration $1_{\PPP_q^+}$. 
Note that the possible discontinuities of the process $q\in \Q\cap [0,T] \mapsto Z_Q(q)$ are necessarily included in $\switch_T$. Since we are on the event $V\subset V_2$, the later process is piecewise constant. 
This enables us to define the random \cadlag process 
 $t\in [0,T] \mapsto Z_Q(t)\in \{0,1\}$ as the unique \cadlag extension of  $q\in \Q\cap [0,T] \mapsto Z_Q(q)$. 
 
On the negligible event $V^c$, we arbitrarily define $Z_Q(q)\equiv 0$. 

In the end, we thus defined for each quad $Q\in\QUAD_\N$ a piecewise constant process $t\in [0,T] \mapsto Z_Q(t)\in \{0,1\}$. Furthermore the set of discontinuities of these processes is always included inside $\switch_T$ whatever $Q\in\QUAD_N$ is (note that this is also the case on $V^c$). 
\end{definition}


\begin{theorem}\label{th.extension}
One can define a \cadlag process $\omega_\infty^\eps : t\in [0,T] \mapsto \HH$ which starts from $\omega_\infty^\eps(t=0)=\omega_\infty \sim \P_\infty$ \margin{Added this} and which is such that a.s. one has for all 
quads $Q\in \QUAD_\N$ and all $t\in [0,T]$:
\[
1_{\boxminus_Q}(\omega_\infty^\eps(t)) = Z_Q(t, \omega_\infty, \PPP_T)\,.
\]
The set of discontinuities of the process $t\mapsto \omega_\infty^\eps(t)\in \HH$ is included in the a.s. finite set $\switch_T\subset [0,T]$.
Furthermore, this process is measurable w.r.t. $(\omega_\infty, \PPP_T)$ and is unique up to indistinguishability among \cadlag processes. 
\end{theorem}

The difficulty in proving this Theorem is to prove that one can simultaneously ``extend'' these projected processes into a single trajectory $t\mapsto \omega_\infty^\eps (t)$ which remains consistent with the processes $Z_q$ and at the same time remains in the space $\HH$. The proof of this Theorem is postponed to Subsection \ref{ss.extension}. Meanwhile we introduce an extension tool which will allow us to prove Theorem \ref{th.extension}. 

\subsection{An extension lemma}\label{ss.remainsH}


In order to define a process as in Theorem \ref{th.extension}, we will need to simultaneously extend  the information provided by the collection of all the processes $Z_Q(\cdot)$. The following Lemma gives sufficient conditions for the existence of a such an extension.

\begin{lemma}\label{l.unique}
Suppose we are given a map $\phi: \QUAD_\N\to \{0,1\}$ satisfying the following constraints:
\bi
\item[(i)] Monotonicity : If $Q,Q'\in \QUAD_\N$, $Q<Q'$ and $\phi(Q')=1$, then $\phi(Q)=1$ as well.
\item[(ii)] Closeness : for any sequence $(Q_n)_{n\geq 0} \in \QUAD_\N$ and any quad  $Q \in \QUAD_\N$ with $Q_n<Q$ for all $n\geq 0$ and $Q_n\to Q$,  if one has $\limsup \phi(Q_n) = 1$, then $\phi(Q)=1$. 
\ei
Then, there exists is a unique element $\omega =\omega_\phi \in \HH$ which extends $\phi$ in the following sense:
for any $Q\in \QUAD_\N$, $Q\in \omega$ if and only if $\phi(Q)=1$. 
\end{lemma}

\ni
{\bf Proof:}
Note that the uniqueness of $\omega_\phi$ follows from Theorem \ref{th.compact}. 
Yet, since the argument for the uniqueness part is straightforward, let us prove it without relying on Theorem \ref{th.compact}: suppose there exist two different configurations $\omega,\omega'\in \HH$ compatible with the map $\phi$ in the above sense. This means that one can find at least one quad $Q \notin \QUAD_\N$ for which one has $\{ \omega\in \boxminus_Q \} \Delta  \{ \omega'\in \boxminus_Q \}$. Since $\QUAD_\N$ is dense in $\QUAD_D$, one can find a sequence $(Q_n)_{n\geq 0}$ in $\QUAD_\N$ which is such that $Q_n \to Q$ and $Q_n<Q$ for all $n\geq 0$. Since either $\omega$ or $\omega'$ belongs to $\boxminus_Q$, 
by assumption $(i)$ we necessarily have $\phi(Q_n)=1$ for all $n\geq 0$. This implies that for all $n\geq 0$, $\omega$ and $\omega'$ both belong to $\boxminus_{Q_n}$. Since $\omega,\omega'$ are in $\HH$, they are closed hereditary subsets which implies that both $\omega$ and $\omega'$ satisfy $\boxminus_Q$ which thus leads to a contradiction.
\medskip

We now turn to the existence of such a configuration $\omega=\omega_\phi$. 
We build the configuration $\omega$ as follows. 
For any $Q\in \QUAD_D\setminus \QUAD_\N$, we can find an increasing sequence of quads $Q_n \in \QUAD_\N$, such that $Q_n < Q$ and $Q_n\to Q$. From assumption $(i)$, the sequence $\phi(Q_n)$ has a limit $l\in \{0,1\}$. We need to check that this limit does not depend on the chosen subsequence. This is straightforward: if $(Q_n')$ is another sequence of quads in $\QUAD_\N$ with $Q_n'<Q$ and $Q_n\to Q$, then it is easy to check that for any $n\geq 0$, there exists $N_n, N'_n \geq 0$ so that 
\begin{align*}\label{}
\begin{cases}
Q_n < Q'_{N'_n} \\
Q'_n < Q_{N_n}
\end{cases}\,,
\end{align*}
which implies by assumption $(i)$ that $\lim \phi(Q_n) = \lim \phi(Q'_n)=l$. Therefore we are able to extend in a consistent way the map $\phi:\QUAD_\N \to \{0,1\}$ to a map $\phi : \QUAD_D \to \{0,1\}$. It remains to show that the configuration $\omega_\phi$ defined by
\[
\omega_\phi:= \{ Q\in \QUAD_D,\, \phi(Q)=1 \}\,,
\]
is a closed hereditary subset and is thus in $\HH$. 
The fact that $\omega_\phi$ is an hereditary subset is straightforward from our construction and from assumption $(i)$. 
To see that it is closed, let $\bar Q_n\in \QUAD_D $ which converges towards $Q\in \QUAD_D$ and suppose all $\bar Q_n$ are in $\omega$. It is easy to check that one can define a sequence of quads $Q_n\in \QUAD_\N$ such that $Q_n<\bar Q_n$ as well as $Q_n<Q$ for all $n\geq 0$ and $Q_n \to Q$. Since $\omega$ is hereditary, we have that $\phi(Q_n)=1$ for all $n\geq 0$. Now, if the limit $Q$ is itself in $\QUAD_\N$, assumption $(ii)$ guarantees that $Q\in \omega$ as well and if $Q\in \QUAD_D\setminus \QUAD_\N$, then by our extension of $\phi$ which does not depend on the chosen sequence of quads, we must have $\phi(Q)=1$ and thus $Q\in \omega$.  \qed

\medskip

At this point, proving Theorem \ref{th.extension} essentially boils down to proving that for almost all $(\omega_\infty, \PPP_T)$, the map
\[ Z_q :  
\begin{array}{ccc}
\QUAD_\N & \longrightarrow  &\{0,1\} \\
 Q & \mapsto  & Z_Q(q)
\end{array}
\]
 satisfies the conditions $(i)$ and $(ii)$ of Lemma \ref{l.unique} for all $q\in \Q\cap [0,T]$ (see the statement of Proposition \ref{pr.extension}).  
Nevertheless, the proof of Theorem \ref{th.extension} will be postponed to Subsection \ref{ss.extension}. Indeed, even though the monotonicity property $(i)$ seems very intuitive, it appears to be quite delicate to prove without relying on a coupling with the discrete 
dynamics $(\omega_\eta^\eps(q))$.
 Indeed, one natural approach is to use the fact that there is an open path in $\Net_{Q',q}$ from $\p_1Q'$ to $\p_3 Q'$. For all $r\in 2^{-\N}$ small enough, one can thus find a set of quads $R_1^r,\ldots R_k^r$ which realise this open path. One might be tempted to claim that such a set of quads also realises an open path in $\Net^r_{Q,q}$ for the smaller quad $Q$ but as is illustrated in Figure \ref{f.careful}, this is not always the case. This is why we postpone the proof of this Proposition to Subsection \ref{ss.extension} and in the mean time, we introduce a coupling between the process $\omega_\eta^\eps(t)$ and what will be the process $\omega_\infty^\eps(t)$. 

 \begin{figure}
 \begin{center}
 \includegraphics[width=0.7\textwidth]{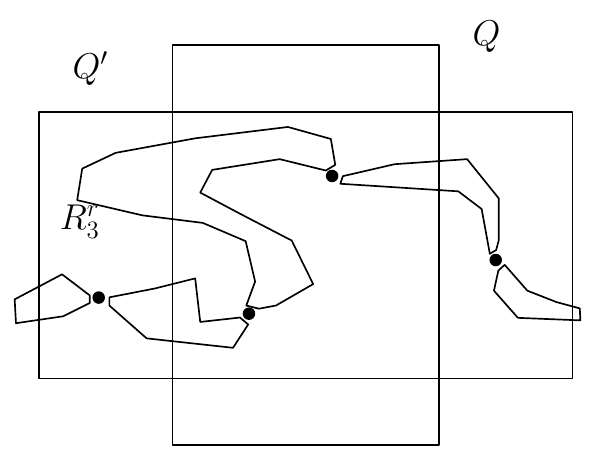}
 \end{center}
 \caption{For the monotony property, the quad $R_3^r$ is troublesome}\label{f.careful}
 \end{figure}
 
 \bigskip

\subsection{A coupling of $\omega_\eta^\eps(t)$ with $(\omega_\infty, \PPP_T(\mu^\eps(\omega_\infty)))$}\label{ss.couplingPPP}

In this subsection, we wish to couple 
$\omega_\eta^\eps(t)$ and $\omega_\infty^\eps(t)$ so that with high probability, they will remain close to each other.
(There is a slight abuse of notation here since we did not yet prove Theorem \ref{th.extension} and thus did not yet define properly $\omega_\infty^\eps(t)$. This will be handled in Subsection \ref{ss.extension}.)  
Recall our main result from \cite{\GPSa}, i.e. Theorem \ref{th.PM} and Corollary \ref{c.PM}.
We have for all fixed $\eps>0$:
\[
(\omega_\eta, \mu^\eps(\omega_\eta)) \overset{(d)} {\longrightarrow} (\omega_\infty, \mu^{\eps}(\omega_\infty))\,.
\] 

The measures $\mu^\eps(\omega_\eta)$ and $\mu^\eps(\omega_\infty)$ are a.s. finite measures on the compact set $\bar D$. 
It is well-known that  the topology of weak-convergence for measures on $\bar D$ is metrizable, 
the so-called Prohorov's metric being one of the possible choices. 
Recall the Prohorov metric on the space $\M= \M(\bar D)$ of finite measures on $\bar D$ is defined as follows:
 for any $\mu,\nu \in \M(\bar D)$, let 
\begin{equation}\label{e.Proho}
d_\M(\mu,\nu):=
\inf \left\{ \eps>0\,, \text{ s.t } \forall \text{ closed set } A \subset \bar D\,, 
\begin{array}{c}
\mu(A^\eps)\le \nu(A)+\eps \\
\textrm{and}\\
\nu(A^\eps) \le \mu(A)+\eps
\end{array}\right\}\,.
\end{equation}

It is well-known (see \cite{\Proho}) that the metric space $(\M(\bar D), d_\M)$ is a complete separable metric space
(in particular, one can apply Skorohod representation theorem).  Furthermore, $\mu_i$ converges weakly towards $\mu\in\M(\bar D)$ if and only if $d_\M(\mu_i,\mu)\to 0$. 

It thus follows from Theorem \ref{th.PM} that one can define a joint coupling of $(\omega_\eta, \mu^\eps(\omega_\eta))$ and $(\omega_\infty, \mu^\eps(\omega_\infty))$ so that a.s. as $\eta\to 0$, 
\begin{align*}\label{}
\begin{cases}
d_\HH(\omega_\eta, \omega_\infty) \to 0  \\
\text{and }\\
d_\M(\mu^\eps(\omega_\eta) , \mu^\eps(\omega_\infty)) \to 0
\end{cases}
\end{align*}

Using this joint coupling, our next step is to couple the Poisson point processes $\PPP_T(\mu^\eps(\omega_\eta))$  and $\PPP_T(\mu^\eps(\omega_\infty))$ so that they are asymptotically close (as $\eta \to 0$). 
We will need the following general lemma. 

\begin{lemma}\label{l.Pcoupling}
We consider the space $\M$ of finite measures on the square $[0,1]^2$ (the extension to our domain $\bar D$ is straightforward). Let $T>0$ be any fixed time.  Suppose $\mu,\nu \in \M$ are such that $d_\M(\mu ,\nu) <\delta$ and $\mu([0,1]^2)<M$, then one can couple $\PPP_T(\mu)$ with $\PPP_T(\nu)$ so that with probability at least 
$1- 12(T+M)\delta^{1/20}$, one has 

\bi  
\item[(i)]  $\#\PPP_T (\mu) = \#\PPP_T(\nu) = k \in \N$ 
\item[(ii)]  If $\PPP_T(\mu) = \{(x_1,t_1), \ldots, (x_k,t_k), 0< t_1< \ldots <t_k <T \}$ and  $\PPP_T(\nu) = \{(y_1,u_1), \ldots, (y_k,u_k), 0<u_1< \ldots <u_k<T \}$, then 
for all $1 \le i \le k$, one has $t_i = u_i$ and $x_i$ and $y_i$ are in the same $r$-square of the grid $G_r$ (see Definition \ref{d.nested}), where
\begin{equation}\label{e.GR}
r =\inf\{ u \in 2^{-\N},\, u\geq 4^{1/20}\delta^{1/20} \}\,.
\end{equation}
In particular, one has $B^{2r}(\PPP_T(\mu))=B^{2r}(\PPP_T(\nu))$. 
\ei
\end{lemma}

\ni
{\bf Proof of Lemma \ref{l.Pcoupling}:}
\ni
Let $r$ be defined as in equation~\eqref{e.GR}.
As in Definition \ref{d.nested}, divide $[0,1]^2$ into $R=r^{-2}$ disjoint squares of side-length $r$ and of the form $[0,r)^2$. 
 Call these squares $S_1, \ldots, S_R$. 
For each square $S_i$, let us call $\mu_i:= \mu(S_i)$ and $\nu_i:= \nu(S_i)$.
In each such square $S_i$, the number of points which will fall in that square for $\PPP_T(\mu)$ is a Poisson variable of mean $T\, \mu_i$. Let us sample independently for all $i\in [1,R]$, $X_i \in \N \sim \mathrm{Poisson}(T\,\mu_i)$.  Now, we have that $\nu_i\le \nu(S_i^{2\delta}) \le \mu(S_i)+ 2\delta=\mu_i+2 \delta$. We now couple $Y_i \sim \mathrm{Poisson}(\nu_i)$ with $X_i$. We distinguish two cases:
\bi
\item[(i)] If $\nu_i< \mu_i$, sample $Y_i$ using a Binomial random variable $\mathrm{Binom}(X_i, \frac{\nu_i}{\mu_i})$.
\item[(ii)] Otherwise $Y_i \sim X_i + \mathrm{Poisson}(T \nu_i-T \mu_i)$.
\ei

We claim that with probability at least $1- (2 R T \delta)$ we have $Y_i \le X_i$ for all $1\le i \le R$. Indeed the probability that $Y_i>X_i$ is bounded from above by $R \times \Pb{\mathrm{Poisson}(2\,T\,\delta) \geq 1} \le 2 R T \delta$. 
\smallskip

We now wish to show that with high probability one has for all $1\leq i \leq R$, $Y_i \geq X_i$. By our coupling, this is already the case for the squares $S_i$ with $\nu_i \geq \mu_i$. Let us analyse what happens for squares $S_i$ in situation $(i)$, i.e. when $\nu_i <\mu_i$. Since $\nu_i / \mu_i$ might be very small, we will divide this set of squares as follows:
\begin{align*}\label{}
\begin{cases}
I^-:= \{ 1\leq i \leq R, \, \mu_i < r^3 \text{ and } \nu_i < \mu_i \} \\
I^+:= \{ 1 \leq i \leq R, \, \mu_i \geq r^3 \text{ and } \nu_i < \mu_i\}
\end{cases}
\end{align*}
Clearly, the squares $S_i$ with $i\in I^-$ contribute very little, since $\sum_{i \in I^-} \mu_i < r^{-2}r^3 =r$. 
On these squares, we have by definition of our coupling $Y_i \leq X_i$. But since $\sum_{i\in I^-} X_i$ is a Poisson variable of parameter bounded by $T \,r$, the probability that $Y_i<X_i$ for at least one $i\in I^-$ is bounded from above by $\Pb{\mathrm{Poisson}(T\, r)\geq 1} \leq T\, r$. It remains to control what happens for  $i\in I^+$. 
In this case since $\nu_i < \mu_i$ we use item $(i)$ to sample $Y_i$. But notice that by our assumption, we have $\mu_i - 2\delta < \nu_i < \mu_i$ which leads to $1- \frac{2\delta}{\mu_i} < \frac {\nu_i}{\mu_i} < 1$ and since in this case $\mu_i \geq r^3$ we have $\frac {\nu_i} {\mu_i} > 1- r^{17}$ ( recall that $\delta < r^{20}/2$ by~\eqref{e.GR}). 
Now, notice that with high probability all $X_i$ are smaller than $M/r^3$, indeed
\begin{align*}\label{}
\Pb{X_i \geq \frac {M} {r^3} } \leq \frac{\Eb{X_i}}{M} \, r^3  \leq T\, r^3\,,
\end{align*}
which implies $\Pb{\exists i, \, X_i \geq M r^{-3}} \leq T\, r$. On the event that all $X_i$ are smaller than $M\, r^{-3}$, 
if $i\in I^+$ since $Y_i \sim \mathrm{Binom}(X_i, \frac{\nu_i}{\mu_i})$, we have 
\begin{align*}\label{}
\Pb{Y_i < X_i} \leq \frac M {r^{3}} r^{17} = M\, r^{14}\,,
\end{align*}
which implies, 
\begin{align*}\label{}
\Pb{\exists i \in I^+,\,  Y_i< X_i} \leq M\, r^{12}\,.
\end{align*}
Summarising the above analysis, we obtain that with probability at least $1- (2\,T \, r^{-2}\, \delta +  2T\, r + M\, r^{12})$, one has $X_i = Y_i\, \forall 1\leq i \leq R$.  Since $r^{20}/8<\delta< r^{20}/2$, we thus obtain the following upper bound on the probability that our event is not satisfied: 
\begin{align*}\label{}
2 T\, r^{-2} \delta + 2 T\, r + M\, r^{12} & \leq T\, r^{18} + 2 \, T \, r + M\, r^{12} \\
&\leq 3 (T+M)\, r  \leq 12 (T+M)\, \delta^{1/20}\,. 
\end{align*}
\qed

As such, one has the following Corollary 
\begin{corollary}\label{c.coupling}
Let us fix $T>0$ and $\eps>0$. One can can couple $(\omega_\eta, \PPP_T^\pm(\mu^\eps(\omega_\eta)))$ and $(\omega_\infty, \PPP_T^\pm(\mu^\eps(\omega_\infty)))$ so that for each $r\in 2^{-\N}$, we have a.s. as $\eta\to 0$:
\bi
\item[(i)] $|\PPP_T^+(\omega_\eta)| = |\PPP_T^+(\omega_\infty)|$ and $|\PPP_T^-(\omega_\eta)| = |\PPP_T^-(\omega_\infty)|$.
\item[(ii)] The switching times $\switch_T(\omega_\eta)$ and $\switch_T(\omega_\infty)$ are identical for $\PPP_T^\pm(\omega_\eta)$ and $\PPP_T^\pm(\omega_\infty)$.
\item[(iii)] $B^r(\PPP_T^\pm(\omega_\eta)) = B^r(\PPP_T^\pm(\omega_\infty))$ (recall Definition \ref{d.nested}).
\ei

This way we obtain a joint coupling of the dynamics $(\omega_\eta^\eps(t))_{\eta>0}$, as defined in Definition \ref{d.dcutoff}, with 
our \cadlag processes $Z_Q(t)$ which are aimed at defining our process $\omega_\infty^\eps(t)$.
\end{corollary}

\ni
{\bf Proof:}
This follows easily from Lemma \ref{l.Pcoupling} with $\mu=\mu^\eps(\omega_\infty)$ and $\eta=\mu^\eps(\omega_\eta)$. One just has to deal with the fact that the first measure $\mu=\mu^\eps(\omega_\infty)$ in the latter Lemma is assumed to have a total mass bounded by $M$. By Proposition \ref{th.mu}, item $(i)$, one indeed has that $\mu^\eps(\bar D) \leq M$  with probability going to one as $M\to \infty$, which is enough for our purpose here. \qed

\subsection{Comparison of $\omega_\eta^\eps(t)$ with $\omega_\infty^\eps(t)$}\label{ss.proofMonotony}

In this subsection, we wish to prove the following proposition.
\begin{proposition}\label{pr.compZ}
Let us consider the coupling from Corollary \ref{c.coupling}. 
For any quad $Q\in\QUAD_\N$, we have 
\begin{align}\label{e.EXCT}
\lim_{\eta\to 0}\Pb{\exists t\in [0,T], \{ \omega_\eta^\eps(t) \in \boxminus_Q \} \Delta \{ \omega_\infty^\eps(t) \in \boxminus_Q\} }  =0\,,
\end{align}
where since we did not yet prove Theorem \ref{th.extension} at this point, the event $\{ \omega_\infty^\eps(t) \in \boxminus_Q \}$ should be understood as $\{ Z_Q(t) =1 \}$. 
\end{proposition}

Note that this Proposition will be very helpful (in fact stronger than what we need) in order to show using Definition \ref{d.skorohod} and item $(i)$ in Proposition \ref{pr.CovrtoCovk} that $d_{\Sk_T}(\omega_\eta^\eps(t), \omega_\infty^\eps(t))$ converges in probability towards 0. 
\bigskip

\ni
{\bf Proof:}
\ni
Recall from Definition \ref{d.ZQ}, that for any $Q\in\QUAD_\N$, the \cadlag process $t\in[0,T]\mapsto Z_Q(t)$ is piecewise constant with a set of discontinuities included in $\switch_T=\switch_T(\omega_\infty)$. Similarly the \cadlag process $t\mapsto \omega_\eta^\eps(t)$ is piecewise constant on $[0,T]$ with a set of discontinuities included in $\switch_T(\omega_\eta)$. Recall that in the joint coupling obtained in Corollary \ref{c.coupling}, one has $\Pb{\switch_T(\omega_\eta)=\switch_T(\omega_\infty)}\to 1$ as $\eta\to 0$. 

Using these facts plus the property that $\switch_T(\omega_\infty)$ is a.s finite (Proposition \ref{pr.PPP}), it is straightforward to check that proving~\eqref{e.EXCT} boils down to proving the following fact: For any $q\in \Q \cap [0,T]$ and any $Q\in \QUAD_\N$, one has 
\begin{align*}\label{}
\lim_{\eta\to 0}\Pb{ \{ \omega_\eta^\eps(q) \in \boxminus_Q \} \Delta \{ \omega_\infty^\eps(q) \in \boxminus_Q\} }  =0\,.
\end{align*}
Let us then fix some $Q\in \QUAD_\N$ and some $q\in \Q\cap [0,T]$.
Let $\alpha \in (0,1)$ be fixed. We wish to show that   
\begin{align*}\label{}
\limsup_{\eta\to 0}\Pb{ \{ \omega_\eta^\eps(q) \in \boxminus_Q \} \Delta \{ \omega_\infty^\eps(q) \in \boxminus_Q\} }  \leq \alpha\,.
\end{align*}
Let us sample $(\omega_\infty, \PPP_q)$ coupled with its discrete analogs.
From Theorem \ref{th.BC}, there is a measurable scale $r_Q=r_Q(\omega_\infty, \PPP_q) \in 2^{-\N}$ such that for all $r\in 2^{-\N},\, r\leq r_Q$, one has $\Net_{Q}(\omega_\infty, \PPP_q) = \Net_Q^r(\omega_\infty, \PPP_q)$. Furthermore the random variable $r^*= r^*(\PPP_q,Q )$ is a.s. positive (see Definition \ref{d.r*}). In particular, one can find $r_\alpha \in 2^{-\N}$ so that 
\begin{align*}\label{}
\Pb{ r_Q \wedge r^* \geq r_\alpha} \geq 1-\alpha/100\,. 
\end{align*}
Let $A_\alpha$ be this event. On the event $A_\alpha$, we have for all $r\leq r_\alpha$, 
\begin{align*}\label{}
\begin{cases}
\Net_Q(\omega_\infty, \PPP_q) = \Net_Q^r(\omega_\infty, \PPP_q) \\
r^*(\PPP_q,Q) \geq r_\alpha
\end{cases}
\end{align*}

Now, in our coupling defined in Corollary \ref{c.coupling}, since $\omega_\eta \to \omega_\infty$ in $\HH$, we have from Proposition \ref{p.discrete} that for any $r\leq r_\alpha$, 
\begin{align*}\label{}
\Pb{ \Net_Q^r(\omega_\infty, X) = \Net_Q^r(\omega_\eta, X) } \to 1\,,
\end{align*}
as $\eta \to 0$ for any fixed (detreministic) $X\subset D$ with $r^*(X,Q)>0$. 
We also have from Corollary \ref{c.coupling} that for any $r\in 2^{-\N}$, 
\begin{align*}\label{}
\Pb{B^r(\PPP_q(\omega_\eta)) = B^r(\PPP_q(\omega_\infty))} \to 1\,,
\end{align*}
as $\eta \to 0$. Since for any fixed $r\in 2^{-\N}$, there are finitely many possible $B^r$ sets, the above two facts plus the way $\Net_Q^r$ is defined in Definition \ref{d.mesoN} imply that for any $r\leq r_\alpha\, \in 2^{-\N}$, we have 

\begin{align}\label{}
\Pb{ \Net_Q^r(\omega_\eta, \PPP_q(\omega_\eta)) = \Net_Q(\omega_\infty, \PPP_q(\omega_\infty)) \md A_\alpha}  \to 1\,,
\end{align}
as $\eta\to 0$. 

Let $r_0 \leq r_\alpha$ be small enough so that there is $\eta_0>0$ such that the probability to have a $2\, r_0$-square in $G_{2\, r_0}$ (recall Definition \ref{d.nested}) with a 6-arms event for $\omega_\eta$ up to radius $r_\alpha/10$ is bounded above by $\alpha/100$ whatever $\eta<\eta_0$ is (see the proof of Lemma \ref{l.stab} where such estimates were used). We will call $B_{r_0}$ the event that there are no such 6-arm events. Finally, let also $\eta_1<\eta_0$ be small enough so that 
for any $\eta<\eta_1$, 
\begin{align}\label{}
\Pb{ \Net_Q^{r_0}(\omega_\eta, \PPP_q(\omega_\eta)) = \Net_Q(\omega_\infty, \PPP_q(\omega_\infty)) \md A_\alpha} \geq 1- \alpha/100\,.
\end{align}
We are now ready to introduce for any $\eta<\eta_1$, the following event 
\[
C:= A_\alpha \cap B_{r_0} \cap \{  \Net_Q^{r_0}(\omega_\eta, \PPP_q(\omega_\eta)) = \Net_Q(\omega_\infty, \PPP_q(\omega_\infty)) \}\,,
\]
whose probability (from the above estimates) is at least $(1-\alpha/100)(1-\alpha/100)-\alpha/100 - \alpha/100$ which is greater than $1- \alpha/10$.  Note that this event depends on the $\eta$- configuration $(\omega_\eta, \PPP_q(\omega_\eta))$
as well as on $(\omega_\infty, \PPP_q(\omega_\infty))$.

Let $\eta<\eta_1$ and suppose we are on the event $C$. We distinguish two cases:
\bnum
\item First case: $\omega_\infty^\eps(q)\in \boxminus_Q$ (or in other words, $Z_Q(q)=1)$. This means that one can find an open path $e_1, \ldots, e_k$ from $\p_1 Q$ to $\p_3 Q$ which only uses vertices in $\PPP_q^+(\omega_\infty)$.  Since we are on the event $C$, we have $\Net^{r_0}_Q(\omega_\eta, \PPP_q(\omega_\eta))= \Net_Q(\omega_\infty, \PPP_q(\omega_\infty))$. By Definition \ref{d.mesoN}, this means that one can find quads $R_1^{r_0}, \ldots, R_k^{r_0}$ which realise the open path $e_1,\ldots, e_k$ and satisfy the conditions of Definition \ref{d.mesoN}. In particular, $\p_1 R_1^{r_0}$ remains $r_0$-away from $\p_1 Q$ (outside of $Q$) and so on. We also have for each $1\leq i \leq k, \omega_\eta \in \boxminus_{R_i^{r_0}}$. To obtain that $\omega_\eta \in \boxminus_Q$, we proceed exactly as in the proof of Lemma \ref{l.stab} by showing that the converse would lead to six-arms events that cannot exist under the above event $C$. We leave the details to the reader. 
\item The second case, i.e. $\omega_\infty^\eps(q)\notin \boxminus_Q$ is treated in the same manner by relying on a dual path $\tilde e_1, \ldots, \tilde e_k$. Note that we need here the fact that $\Net_Q(\omega_\infty, \PPP_q)$ is a.s. Boolean by Corollary \ref{c.boolean}. 
\enum

We thus conclude that if $\eta<\eta_1$, we have 
\begin{align*}\label{}
\Pb{ \{ \omega_\eta^\eps(q) \in \boxminus_Q \} \Delta \{ \omega_\infty^\eps(q) \in \boxminus_Q\} }  \leq \alpha/10\,,
\end{align*}
as desired. 
\qed

We will use this Proposition later to prove that $d_{\Sk_T}(\omega_\eta^\eps(\cdot), \omega_\infty^\eps(\cdot))$
goes to zero in probability as $\eta \to 0$, see Theorem \ref{th.etaskoro}. But, before that, we still need to justify the existence of a \cadlag trajectory $\omega_\infty^\eps(t)$ which extends our projected \cadlag processes $t\mapsto Z_Q(t)$. We will now use the above Proposition \ref{pr.compZ} in order to prove Theorem \ref{th.extension}. 
\bigskip

\subsection{Proof of Theorem \ref{th.extension}}\label{ss.extension}

As explained at the end of Subsection \ref{ss.remainsH}, we start by proving the following Proposition.

\begin{proposition}\label{pr.extension}
For almost all $(\omega_\infty, \PPP_T)$, the following property is satisfied: for all $q\in \Q\cap [0,T]$, the map
\[ Z_q :  
\begin{array}{ccc}
\QUAD_\N & \longrightarrow  &\{0,1\} \\
 Q & \mapsto  & Z_Q(q)
\end{array}
\]
 satisfies the conditions $(i)$ and $(ii)$ of Lemma \ref{l.unique}. 
\end{proposition}

\ni
{\bf Proof:}

\ni
Let us fix $T>0$ and $q\in \Q\cap [0,T]$. We wish to show that the random map $Z_q : \QUAD_\N \to \{0,1\}$ 
 a.s. satisfies the conditions $(i)$ and $(ii)$ of Lemma \ref{l.unique}. 
\smallskip

Let us start with condition $(i)$. Since $\QUAD_\N$ is countable, it is enough to check that property $(i)$ is a.s. satisfied for any fixed $Q,Q' \in \QUAD_\N$ with $Q<Q'$. Suppose $Z_q(Q')=1$, we wish to show that a.s. $Z_q(Q)=1$ as well. We will compare with the $\eta$-lattice using Proposition \ref{pr.compZ}. Let then $(\omega_\eta, \PPP_T(\omega_\eta))$ be coupled with $(\omega_\infty, \PPP_T(\omega_\infty))$ as in Corollary \ref{c.coupling}. By Proposition \ref{pr.compZ}, we have as $\eta\to 0$
\begin{align*}\label{}
\begin{cases}
\lim_{\eta\to 0}\Pb{ \{ \omega_\eta^\eps(q) \in \boxminus_{Q'} \} \Delta \{ Z_q(Q')=1\} }  =0 \\
\lim_{\eta\to 0}\Pb{ \{ \omega_\eta^\eps(q) \in \boxminus_Q \} \Delta \{ Z_q(Q)=1\} }  =0
\end{cases}
\end{align*}

From the above limiting probability, we can write 
\begin{align*}\label{}
\Pb{  Z_q(Q')=1,\, Z_q(Q)=0} & 
= \lim_{\eta \to 0} \Pb{ Z_q(Q')=1,\, \omega_\eta^\eps(q) \in \boxminus_{Q'}, \omega_\eta^\eps(q) \notin \boxminus_{Q}, Z_q(Q)=0} \\
&\leq \lim_{\eta\to 0} \Pb{ \omega_\eta^\eps(q) \in \boxminus_{Q'}, \omega_\eta^\eps(q) \notin \boxminus_{Q} } \\
& = 0\,,
\end{align*}
since $Q<Q'$. This ends the proof of item $(i)$. 

\smallskip

We now turn to the proof of item $(ii)$. Since $\QUAD_\N$ is countable, we may fix one quad $Q\in \QUAD_\N$. 
We wish to prove that the probability that there exists a sequence of quads $Q_n$ with $Q_n<Q$ and $Q_n \to Q$ so that $Z_q(Q_n)=1$ for each $n\geq 1$ but $Z_q(Q)=0$ is equal to zero. 

Suppose our fixed quad $Q$ is in $\QUAD^{k_0}$ (recall Definition \ref{d.kquads}). Similarly to the definition of $\bar Q_k$ in the same Definition \ref{d.kquads}, we will define for each $k\geq k_0$, the quad $\tilde Q_k$ which among all quads $Q'$ in $\QUAD^{k+10}$ satisfying $Q'<Q$ is the largest one. 
Suppose now that a sequence of quads $(Q_n)$ as above exists. Then for each $k\geq k_0$, there is $N=N_k<\infty$ such that for all $n\geq N_k$, one has $\tilde Q_k < Q_n < Q$. In particular, by item $(i)$, one has a.s.  
$Z_q(\tilde Q_k)=1$ since $\tilde Q_k< Q_n$ (for $n$ large enough) and $Z_{q}(Q_n)=1$. 
This implies that there is a negligible event  $W$ ($\Pb{W}=0$) so that 
for any quad $Q\in \QUAD_\N$:
\begin{align*}\label{}
\{ \exists (Q_n)_n \in \QUAD_\N,\, Q_n<Q, \, Q_n \to Q,\, Z_q(Q_n)=1,\, Z_q(Q)=0 \} \\
& \hskip -6 cm  \subset 
\Bigl( \bigcap_{k\geq k_0}\, \{ Z_q(\tilde Q_k)=1,\, Z_q(Q)=0 \}\Bigr) \cup W \,.
\end{align*}

As such it is enough to prove the lemma below in order to show that item $(ii)$ is a.s. satisfied. 
\begin{lemma}\label{}
For any quad $Q\in \QUAD^{k_0}\subset \QUAD_\N$, 
there exists a constant $C=C_Q<\infty$ such that for any $k\geq k_0$ we have
\begin{align}\label{}
\Pb{Z_q(\tilde Q_k)=1\,, Z_q(Q)=0} \leq C\, 2^{-k}\,.
\end{align}
\end{lemma}

To prove this lemma, we proceed as in the proof of item $(i)$ by using the coupling with $(\omega_\eta, \PPP_T(\omega_\eta))$. Using Proposition \ref{pr.compZ}, we have 
\begin{align*}\label{}
\Pb{Z_q(\tilde Q_k)=1\,, Z_q(Q)=0} 
& = \lim_{\eta \to 0} \Pb{Z_q(\tilde Q_k)=1,\, \omega_\eta^\eps(q) \in \boxminus_{\tilde Q_k},\,
\omega_\eta^\eps(q) \notin \boxminus_{Q},\, \,, Z_q(Q)=0} \\
&\le \lim_{\eta \to 0} \Pb{\omega_\eta^\eps(q) \in \boxminus_{\tilde Q_k},\,
\omega_\eta^\eps(q) \notin \boxminus_{Q}}\,.
\end{align*}
Now, it is a standard fact (see for example teh analysis in Section 7.2. in \cite{\GPS} or in Chapter VI in \cite{Buzios}) that the above probability is given by the existence of a three arm event along the $2^{-k-10}$ neighbourhood of $\p_1 Q$ (some analysis needs to be done near the corners of $\p_1 Q$, where only a two-arm events appears, see again Chapter VI in \cite{Buzios} where this is treated in details). As such uniformly in $\eta$ small enough, we obtain an upper bound of the form $O(2^{-k})$.   \qed

 \ni
 {\bf End of proof of Theorem \ref{th.extension}:}
 
 \ni
 Let $A$ be the event that for each $q\in \Q\cap [0,T]$, the map $Z_q= Z_q(\omega_\infty, \PPP_T)$ defined in Proposition \ref{pr.extension} satisfies the conditions $(i)$ and $(ii)$ of Lemma \ref{l.unique}. By Proposition \ref{pr.extension}, we have that $\Pb{A}=1$.
Furthermore, the process $q\mapsto Z_q$ is by construction (see Definition \ref{d.ZQ}) piecewise constant with a set of discontinuities included in the a.s. finite $\switch_T\subset [0,T]$. 
As such, on the event $A\cap \{ |\switch_T| <\infty \}$ and using Lemma \ref{l.unique}, we define the \cadlag process $t\in [0,T]\mapsto \omega_\infty^\eps(t)$ to be the unique \cadlag process in $\HH$ which is compatible with all the \cadlag processes $\{ t\in [0,T] \mapsto Z_t(Q)\}_{Q\in \QUAD_\N}$ defined in  Definition \ref{d.ZQ}.
On $A^c$,  define the \cadlag process $t\mapsto \omega_\infty^\eps(t)$ on $A^c$ to be the constant process $\omega_\infty^\eps(t):=\omega_\infty(t=0)$. \margin{I changed here to make it \cadlag !!!!}


The fact that this process is measurable w.r.t. to $(\omega_\infty,\PPP_T(\mu^\eps(\omega_\infty))$  is due to the fact that one has built $t\mapsto \omega_\infty^\eps(t)$ out of the networks $\Net_Q(\omega_\infty,\PPP_q), q\in Q$, which are themselves limits of the mesoscopic networks $\Net_Q^r(\omega_\infty, \PPP_q)$. Finally, the later networks are 
measurable w.r.t. $(\omega_\infty, \PPP_q)$ using Lemma \ref{l.mesNET}.

The fact that this process is unique up to indistinguishability is obvious (it is a \cadlag process). 
 \qed

This thus ends the proof of Theorem \ref{th.extension}. We now have for any $\eps>0$ and any $T>0$ a well-defined random process 
\begin{align*}\label{}
 t\in [0,T] \mapsto \omega_\infty^\eps(t)\,.
\end{align*}

In the next subsection, we wish to prove that under the above coupling, the trajectories $\omega_\eta^\eps(t)$ 
and $\omega_\infty^\eps(t)$ are very close to each other w.h.p as $\eta\to 0$. 

\subsection{The process $\omega_\eta^\eps(\cdot)$ converges in probability towards $\omega_\infty^\eps(\cdot)$ in the Skorohod space $\Sk_T$.}\label{ss.convinSK}

\begin{theorem}\label{th.etaskoro}
Let $T>0$ and $\eps>0$ be fixed. Under the joint coupling defined in Corollary \ref{c.coupling}, the \cadlag process $t\in[0,T] \mapsto \omega_\eta^\eps(t)$ converges in probability in the Skorohod space $\Sk_T$ (see Definition \ref{d.skorohod}) towards the \cadlag process $t\mapsto \omega_\infty^\eps(t)$ defined in Theorem \ref{th.extension}.
\end{theorem}


\ni
{\bf Proof: }

\ni
We wish to prove that for any $r>0$, 
\begin{align*}\label{}
\lim_{\eta\to 0}\Pb{ d_{\Sk}(\omega_\eta^\eps(\cdot), \omega_\infty^\eps(\cdot))  > r }  =0\,.
\end{align*}


Recall the definition of $d_{\Sk}$ from Definition \ref{d.skorohod}. By the a.s. property $(ii)$ in Corollary $\ref{c.coupling}$, it will be enough to fix $\lambda(s)=\mathrm{id}(s)=s$ so that 
\begin{equation}\label{}
\|\lambda\|: = \sup_{0\leq s<t\leq 1} |\log \frac{\lambda(t)-\lambda(s)}{t-s}| =0 \,. 
\end{equation}
It thus remains to show that for any fixed $r>0 $, one has 
\begin{align}\label{e.unif1}
\Pb{ \exists t\in [0,T],\, d_\HH(\omega_\eta^\eps(t), \omega_\infty^\eps(t)) > r} \underset{\eta\to 0}{\longrightarrow} 0\,.
\end{align}

Since the set $\QUAD^k$ is a finite set of quads, we readily obtain from Proposition \ref{pr.compZ} that for any fixed $k\geq 0$, 
\begin{align}\label{e.meige}
\lim_{\eta\to 0}\Pb{\exists Q\in \QUAD^k \text{ and }\exists t\in [0,T], \{ \omega_\eta^\eps(t) \in \boxminus_Q \} \Delta \{ \omega_\infty^\eps(t) \in \boxminus_Q\} }  =0\,.
\end{align}
Recall the notations from Definition \ref{d.kquads}. We claim that the event 
\begin{align*}\label{}
\mathcal{C}:= \{ \exists t\in [0,T],\, \omega_\infty^\eps(t) \notin \O_k(\omega_\eta^\eps(t)) \text{ and } \omega_\eta^\eps(t) \notin \O_k(\omega_\infty^\eps(t)) \}
\end{align*}
is included in the above event $\left\{ \exists Q\in \QUAD^k \text{ and }\exists t\in [0,T], \{ \omega_\eta^\eps(t) \in \boxminus_Q \} \Delta \{ \omega_\infty^\eps(t) \in \boxminus_Q\} \right\}$. 
Indeed, suppose our joint coupling satisfies the event $\mathcal{C}$. We just need to focus on the fact that $\omega_\infty^\eps(t) \notin \O_k(\omega_\eta^\eps(t))$ for some $t\in [0,T]$. 
This means we can find a quad $Q\in \QUAD^k$ with respect to which $\omega_\infty^\eps(t)$ and $\omega_\eta^\eps(t)$ behave differently. We thus have two cases.
\bnum
\item Either this quad $Q$ is such that $\omega_\eta^\eps(t) \in \boxminus_Q$ and  $\omega_\infty^\eps(t)\notin \boxup_{\hat Q_k}^c$ (recall the notation after Definition \ref{d.kquads}). In particular, this implies that a.s. $\omega_\infty^\eps(t) \notin \boxminus_Q$. We are using here the fact that our process by its construction in Theorem \ref{th.extension} belongs to $\HH$ and is thus hereditary.  In particular the event $\{ \omega_\eta^\eps(t) \in \boxminus_Q \} \Delta \{\omega_\infty^\eps(t)\in \boxminus_Q\}$ holds. 
\item Or this quad $Q$ is such that $\omega_\eta^\eps(t)\in \boxminus_Q^c$ and $\omega_\infty^\eps(t) \in \boxminus_{\bar Q_k}$ (recall the notation after Definition \ref{d.kquads}). In particular, since $\omega_\infty^\eps(t) \in \HH$, and since $\bar Q_k > Q$, we have $\omega_\infty^\eps(t) \in \boxminus_Q$ which implies also here that 
the event  $\{ \omega_\eta^\eps(t) \in \boxminus_Q \} \Delta \{\omega_\infty^\eps(t)\in \boxminus_Q\}$ is realised. 
\enum

We thus obtain from equation~\ref{e.meige} the following estimate that for any fixed $k\in \N$,
\begin{align}\label{e.Kr}
\lim_{\eta\to 0}\Pb{ \exists  t\in [0,T],\, \omega_\infty^\eps(t) \notin \O_k(\omega_\eta^\eps(t)) \text{ and } \omega_\eta^\eps(t) \notin \O_k(\omega_\infty^\eps(t))}  =0\,.
\end{align}

Using Proposition \ref{pr.CovrtoCovk} together with equation~\eqref{e.Kr} (with $k=\k(r)$) we obtain that for any fixed  $r>0$, we have 
\begin{align*}\label{}
\lim_{\eta\to 0}\Pb{ \exists t\in [0,T],\, d_\HH(\omega_\eta^\eps(t),\, \omega_\infty^\eps(t)) > r }  = 0\,.
\end{align*}

This implies (using $\lambda(s)=\mathrm{id}(s) =s $) as desired that for any $r>0$, 
\begin{align*}\label{}
\lim_{\eta\to 0}\Pb{ d_{\Sk}(\omega_\eta^\eps(\cdot), \omega_\infty^\eps(\cdot))  > r }  =0\,.
\end{align*}
\qed

\subsection{The case of the near-critical trajectory $\lambda \mapsto \omega_\infty^{\nc,\eps}(\lambda)$}\label{ss.DNC}

The construction of the near-critical trajectory $\lambda \mapsto \omega_\infty^{\nc,\eps}(\lambda)$ follows the exact same steps as the construction of $t \mapsto \omega_\infty^{\eps}(t)$, except that instead of fixing some $T>0$, we fix some $L>0$ and work on the interval $\lambda\in [-L,L]$. Also, we do not need an analog here of $\PPP_T= \PPP_T^+ \cup \PPP_T^-$ since in this near-critical case, it is enough to consider $\PPP_L=\PPP_L(\mu^\eps(\omega_\infty(0)))$, a Poisson point process on $\bar D \times [-L,L]$ with intensity measure $d\mu^\eps \times d\lambda$. 

Theorem \ref{th.etaskoro} extends readily to this near-critical setting where $\lambda\in[-L,L] \mapsto \omega_\eta^{\nc,\eps}(\lambda)$ converges to $\lambda \mapsto \omega_\infty^{\nc,\eps}(\lambda)$ as $\eta\to 0$ (either in law under the topology of $\Sk_L$ or in probability for a joint coupling on $\Sk_L$ similar to the coupling defined in Corollary \ref{c.coupling}).


\section{Stability property on the discrete level}\label{s.stability}
 
%
%

We wish to prove the following Proposition.

\begin{proposition}\label{pr.stab}
Let $T>0$ be fixed. There exists a continuous function $\psi =\psi_T :  [0,1]\to [0,1]$, with $\psi(0)=0$ such that uniformly in $0<\eta<\eps$,
\[
\Eb{ d_{\Sk_T}(\omega_\eta(\cdot), \omega_\eta^\eps(\cdot)) }\leq  \psi(\eps)\,.
\] 
\end{proposition}

To prove this Proposition, we will need to introduce some notations as well as some preliminary lemmas.

First of all, since this entire section is about discrete configurations $\omega_\eta\in \HH$, we will often omit for convenience the subscript $\eta$ and denote the percolation configurations simply by $\omega$.

\begin{definition}[Importance of a point]\label{}
Given a percolation configuration $\omega=\omega_\eta\in \HH$ and a site $z$,
let $Z(z)=Z_\omega(z)$ denote the maximal radius $r$ such that
the four arm event holds from the hexagon of $z$ to distance $r$ away.  
This is also the maximum $r$ for which changing the value of $\omega(z)$ will
change the white connectivity in $\omega$ between two white points at distance
$r$ away from $z$, or will change the black connectivity between two
black points at distance $r$ away from $z$.
The quantity $Z(z)$ will also be called the
{\bf importance} of $z$ in $\omega$.
\end{definition}


\begin{definition}\label{}
Fix $T>0$. We will denote by $X=X_{\eta,T}$ the random set of sites on $\eta \Tg$ which are updated along the dynamics $t\in[0,T]\mapsto \omega_\eta(t)$. Recall from Definition \ref{d.DP} that this random subset of $\eta \Tg$ is independent of $\omega=\omega_\eta(t=0)$ and each site $z\in\eta \Tg$ is in $X$ independently with probability $q_T:=1-e^{-T r(\eta)} \sim T\, \eta^2/ \alpha_4^\eta(\eta,1)$. 

let $\Omega(\omega,X)$ denote the set of percolation configurations
$\omega'$ such that $\omega'(x)=\omega(x)$ for all $x\notin X$.
Finally, let $\ev A_4(z,r,r')$ denote the $4$-arm event
in the annulus $A(z,r,r')$.
\end{definition}


\begin{lemma}\label{l.main}
Let $T>0$ be fixed. 
Set $r_i:= 2^i\,\eta$, $N:= \lfloor\log_2(1/\eta)\rfloor$.
Let $\ev W_z(i,j)$ denote the event that
there is some $\omega'\in\Omega(\omega,X)$ satisfying $\ev A_4(z,r_i,r_j)$.
Then for every pair of integers $i,j$ satisfying $0\le i<j<N$
and every $z\in\R^2$,
\begin{equation}\label{e.dinduct}
\Pb{\ev W_z(i,j)}
\le C_1\,\alpha_4(r_i,r_j)\,,
\end{equation}
where $C_1=C_1(T)$ is a constant that may depend only on
$T$ (note that here $\P$ includes the extra randomness in the choice of the subset $X$).
\end{lemma}
\begin{proof}
Let $\ev D$ denote the event
that $\omega$ does not satisfy $\ev A_4(z,r_{i+1},r_{j-1})$.
Suppose that $\ev W_z(i,j)\cap \ev D$ holds, and let
$\omega'\in\Omega(\omega,X)$
satisfy $\ev A_4(z,r_i,r_j)$.
Let $Y_0:=X\setminus A(z,r_{i+1},r_{j-1})$,
and let $\{x_1,x_2,\dots,x_m\}$ be some ordering
of $X\cap A(z,r_{i+1},r_{j-1})$.
Let $Y_k=Y_0\cup\{x_1,x_2,\dots,x_k\}$, $k=1,2,\dots,m$,
and let $\omega_k$ be the configuration that agrees with
$\omega'$ on $Y_k$ and is equal to $\omega$ elsewhere.
Then $\omega_0$ does not satisfy $\ev A_4(z,r_{i+1},r_{j-1})$,
and therefore also does not satisfy $\ev A_4(z,r_i,r_j)$.
On the other hand, $\omega_m=\omega'$ satisfies $\ev A_4(z,r_i,r_j)$.
Let $q\in\{1,2,\dots,m\}$ be minimal with the property that
$\ev A_4(z,r_i,r_j)$ holds in $\omega_q$,
and let $n\in\N\cap[i+1,j-2]$ be chosen so that
$x_q\in A(z,r_n,r_{n+1})$.
Then $x_q$ is pivotal in $\omega_q$ for $\ev A_4(z,r_i,r_j)$.
Since $B(x_q,r_{n-1})\subset A(z,r_i,r_j)$,
this implies that $\omega_q$ satisfies
$\ev A_4(x_q,2\,\eta,r_{n-1})$.
Hence, we get the bound
\begin{equation*}\label{e.bs}
\begin{aligned}
&
\Pb{\ev W_z(i,j),\,\ev D}
\le
\sum_{n=i+1}^{j-2}\;\;
\sum_{x\in A(z,r_{n},r_{n+1})}
\Pb{ x\in X,\,\ev W_x ( 1,n-1 ) ,\,\ev W_z(i,j) }.
\end{aligned}
\end{equation*}
Since $\ev W_z(i,j)\subset \ev W_z(i,n-1)\cap\ev W_z(n+2,j)$
and since $B(x,r_{n-1})\subset A(z,r_{n-1},r_{n+2})$,
independence on disjoint sets gives
\begin{align*}
\Pb{ x\in X,\,\ev W_x ( 1,n-1 ) ,\,\ev W_z(i,j) } & 
\\
& \hskip -3 cm \le
\Pb{ x\in X,\,\ev W_x ( 1,n-1 ) ,\,\ev W_z(i,n-1),\,\ev W_z(n+2,j) }
\\
&\hskip -3 cm =
\Ps{x\in X}\,
\Pb{ \ev W_x ( 1,n-1 )}
\Pb{\ev W_z(i,n-1)}\,\Pb{\ev W_z(n+2,j) }.
\end{align*}

Now set $b_i^j:=\sup_z \Pb{\ev W_z(i,j)}$.
The above gives
$$
\Pb{\ev W_z(i,j),\,\ev D}
\le
O(T)
\sum_{n=i+1}^{j-2}
\,(r_n/\eta)^2\,
\eta^2\,\alpha_4^\eta(\eta,1)^{-1}\,b_1^{n-1}\,b_i^{n-1}\,b_{n+2}^j\,.
$$
Since
$
\Pb{\ev W_z(i,j)}
\le \Pb{\neg\ev D}+\Pb{\ev W_z(i,j),\,\ev D)}
$,
The above shows that for some absolute constant
$C_0>0$, we have
\begin{equation}\label{e.recurs}
\begin{aligned}
b_i^j/C_0 &
\le \alpha_4(r_i,r_j) +
T\,\sum_{n=i+1}^{j-2}
r_n^2 \,\alpha_4^\eta(\eta,1)^{-1}\,b_1^{n-1}\,b_i^{n-1}\,b_{n+2}^j
\\ &
\le \alpha_4(r_i,r_j) +T\,\alpha_4^\eta(\eta,1)^{-1}\,\sum_{n=i+1}^{j-1} \,
r_n^2 \,b_1^{n-1}\,b_i^{n-1}\,b_{n+2}^j.
\end{aligned}
\end{equation}
We now claim that~\eqref{e.dinduct} holds with some fixed constant
$C_1=C_1(T)$,
to be later determined. This will be proved by induction on $j$,
and for a fixed $j$ by induction on $j-i$.
In the case where $j-i\le 5$, say, this can be guaranteed by an
appropriate choice of $C_1$.
Therefore, assume that the claim holds for all smaller $j$ and
for the same $j$ with all larger $i$.
The inductive hypothesis can be applied to estimate the right
hand side of~\eqref{e.recurs}, to yield
\begin{align*}
b_i^j\le C_0\, \alpha_4(r_i,r_j)+
T\,C_0\,C_1^3\,\alpha_4^\eta(\eta,1)^{-1}
\sum_{n=i+1}^{j-1}\,r_n^2\,\alpha_4(r_1,r_{n-1})
\,\alpha_4(r_i,r_{n-1})\,\alpha_4(r_{n+2},r_j) \,.
\end{align*}
By the familiar multiplicative properties of $\alpha_4$, \margin{A link here !}
we obtain
\begin{equation}
\label{e.near}
b_i^j\le C_2\,\alpha_4(r_i,r_j)\Bigl(1+T\,
C_1^3
\sum_{n=i+1}^{j-1}\,\frac{r_n^2}{\alpha_4(r_n,1)}
\Bigr),
\end{equation}
for some constant $C_2$.
Since $O(1)\,\alpha_4(r_n,1)>r_n^{2-\epsilon}$
for some constant $\epsilon>0$ (see for example Section 2.2 in \cite{\GPS}), it is clear that when $N-j$ is larger
than some fixed constant $M=M(T)\in\N$,
we have
$$
T\,
(2\,C_2)^3
\sum_{n=i+1}^{j-1}\,\frac{r_n^2}{\alpha_4(r_n,1)} \le 1\,.
$$
This shows that~\eqref{e.near} completes the inductive
step if we choose $C_1=2\,C_2$ and if $N-j>M$.
(Note that in the proof of the induction step when
$N-j>M$, we have not relied on the inductive assumption
in which this condition does not hold.)
To handle the case $N-j\le M$, we just note
that $b_i^j\le b_i^{N-M-1}$,
and the estimate that we have for $b_i^{N-M-1}$ is within
a constant factor (depending on $T$)
of our claimed estimate for $b_i^j$,
since $M$ depends only on $T$.
\end{proof}

Set
$$
Z^X(z):=\sup_{\omega'\in\Omega(\omega,X)}Z_{\omega'}(z)\,.
$$

\begin{lemma}\label{l.change}
For every site $z$ and every $\eps$ and $r$ satisfying
$2\,\eta<\eps<2^4\,\eps<r\le 1$, we have
$$
\Pb{Z^X(z)\ge r,\, Z_\omega(z)\le \eps}\le
O_T(1)\,\eps^2\,\alpha_4(\eta,\eps)\,\alpha_4(r,1)^{-1}\,.
$$
\end{lemma}

The proof uses some of the ideas going into the proof of Lemma~\ref{l.main}
as well as the estimate provided by that lemma.

\begin{proof}
Fix $z$, $\eps$ and $r$ as above.
Suppose that
$Z^X(z)\ge r$ and $Z_\omega(z)\le \eps$ both hold.
Let $\omega'\in\Omega(\omega,X)$ be such that
$Z_{\omega'}(z)\ge r$.
Let $x_1,x_2,\dots,x_m$ be the sites in $\dbox(z,\eps)$ where
$\omega'\ne\omega$.  (We use some arbitrary but fixed rule to choose
$\omega'$ and the sequence $x_j$ among the allowable possibilities.)
For each $j=0,1,\dots,m$, let $\omega_j$ denote the configuration that agrees
with $\omega'$ on every site different from $x_{j+1},x_{j+2},\dots,x_m$,
and agrees with $\omega$ on $x_{j+1},\dots,x_m$.
Then $\omega_m=\omega'$ and $Z_{\omega_0}(z)<\eps$.
Let $k$ be the first $j$ such that $Z_{\omega_j}(z)>r$.

Fix some site $x$ satisfying $r^x:=|z-x|\le\eps$. In order for
$Z^X(z)\ge r$, $Z_\omega(z)\le \eps$ and
$x_k=x$ to hold, the following four events must occur:
 $x\in X$,
 $Z^X(z)\ge r^x/2$,
 $Z^X(x) \ge r^x/2$, and
 $\ev W_z(2+\lceil \log_2 r^x\rceil ,\lfloor \log_2 r\rfloor)$
(using the notation of Lemma~\ref{l.main}).
We have $\Pb{x\in X}= q_T=O(T)\,\eta^2\,\alpha_4(\eta,1)^{-1}$,
while the probabilities of the latter three events
are bounded by Lemma~\ref{l.main}.
Combining these bounds, we get
\begin{multline*}
\Pb{ Z^X(z)\ge 1,\, Z(z)\le \eps,\,x_k=x}
\le
O_T(1)\, \alpha_4(\eta,r^x)^2\, \eta^2\,\alpha_4(\eta,1)^{-1}\,
\alpha_4(r^x,r)
\\
=
O_T(1)\,\alpha_4(\eta,r^x)\,\eta^2\,\alpha_4(r,1)^{-1}\,.
\end{multline*}
Summing this bound over all sites $x$ satisfying $|z-x|\le\eps$
 yields the lemma.
\end{proof}

For any quad $Q\in \QUAD$, if $r>0$ is smaller than the minimal distance from $\p_1 Q$ to $\p_3 Q$,
we will say that $Q$  is {\bf $r$-almost crossed} by $\omega=\omega_\eta\in\HH$, if there is an
open path in the $r$-neighborhood of $Q$ that comes within distance
$r$ of each of the two arcs $\p_1 Q$ and $\p_3 Q$.

\begin{proposition}\label{p.almostCross}
Let $T$ and $X$ be as above, and fix some quad $Q\in \QUAD$.
Let $r>0$ be smaller than the minimal distance between $\p_1Q$ and $\p_3 Q$,
and suppose that $0<\eta<2\,\eta<\eps<2^5\,\eps<r\le 1$.
 Then the probability that there are some
$\omega',\omega''\in\Omega(\omega,X)$ such that
(a) $Q$ is crossed by $\omega'$, (b) $Q$ is not $r$-almost crossed by $\omega''$, and
(c) $\omega'(z)=\omega''(z)$ for every site $z$ satisfying $Z_\omega(z)\ge \eps$
is at most
$$
O_{T,Q}(\eps^2)\,\alpha_4(\eps,1)^{-1}\,\alpha_4(r,1)^{-1}.
$$
\end{proposition}

\begin{proof}
Suppose that there are such $\omega'$ and $\omega''$.
Let $Y$ denote the set of sites whose hexagons are contained in the
$r$-neighborhood of $\p_1 Q\cup\p_3 Q$,
and let $\{x_1,x_2,\dots,x_m\}$ denote the sites not in $Y$ whose hexagons
intersect $Q$.
For $j=0,1,\dots,m$, let $\omega_j$ denote the configuration
that agrees with $\omega'$ on $Y\cup\{x_{j+1},x_{j+2},\dots,x_m\}$,
and agrees with $\omega''$ elsewhere.
Then $Q$ is crossed by $\omega_0$
(since $\omega_0$ agrees with $\omega'$
on all hexagons intersecting $Q$),
but is not $r$-almost crossed by
$\omega_m$ (since $\omega_m$ agrees with $\omega''$ on all hexagons
except those contained in the $r$-neighborhood of $\p_1Q\cup\p_3Q$).
Let $k$ be the least index $j$ such that there is no
$\omega_j$-white path connecting $\p_1Q$ and $\p_3 Q$ within the
$r$-neighborhood of $Q$.
Then $\omega_{k-1}$ and $\omega_k$ differ only in the color of
$x_k$.
Since a flip of $x_k$ modifies the connectivity
between $\p_1 Q$ and $\p_3 Q$ within
the $r$-neighborhood of $Q$, and since
the hexagon of $x_k$ intersects $Q$ and is not
contained in the $r$-neighborhood of $\p_1Q\cup\p_3Q$,
it follows that $Z_{\omega_{k-1}}(x_k)>r/2$.
Consequently, $Z^X(x_k)>r/2$.
If $x$ is any site, then in order to have $x=x_k$, we must
have (i) $Z^X(x)>r/2$, (ii) $Z_\omega(x)<\eps$, (iii) $x\in X$,
and (iv) the hexagon of $x$ intersects $Q$.
There are $O_Q(\eta^{-2})$ sites satisfying (iv).
The event (iii) has probability $q_T$ and
is independent from the intersection of
(i) and (ii), while Lemma~\ref{l.change} bounds the
probability of this intersection.
The proposition now follows easily by summing the
bound we get for $\Ps{x_k=x}$ over all
possible $x$.
\end{proof}

The previous Proposition readily implies the following Lemma. 
\begin{lemma}\label{}
Let $k\in \N$ and $T>0$ be fixed and suppose that  $0<\eta<2\eta< \eps<2^{-k-20}$. Then, the probability that there are some
$\omega',\omega''\in\Omega(\omega,X)$ such that
\bi
\item[(a)] $\omega'\notin \O_k(\omega'')$ (recall Definition \ref{d.open});
\item[(b)] $\omega''\notin \O_k(\omega')$;
\item[(c)] $\omega'(z)=\omega''(z)$ for every site $z$ satisfying $Z_\omega(z)\ge \eps$;
\ei
is at most
$$
O_{T,k}(\eps^2)\,\alpha_4(\eps,1)^{-1}.
$$
\end{lemma}

\ni
{\bf Proof:}
Suppose $\omega' \notin \O_k(\omega'')$ (the second condition $(b)$ is treated the same way). 
We need to analyse 2 cases:
\bnum
\item Either, there is some $Q\in \QUAD^k$, s.t. $\omega' \in \boxminus_Q$ and $\omega'' \in \boxup_{\hat Q_k}$.
By the definition of $\hat Q_k= B_{d_\QUAD}(Q,2^{-k-10})$, this means that $Q$ is not $r$-almost crossed by $\omega''$ with, say $r=2^{-k-20}$. As such, one has from the above Proposition a constant $C_{T,Q}<\infty$ such that  the probability of such a scenario is bounded from above by $C_{T,Q} \; \eps^2 \alpha_4(\eps,1)^{-1} \alpha_4(2^{-k-20},1)^{-1}$.
\item Or, there is some $Q\in\QUAD^k$, s.t. $\omega' \in \boxminus_Q^c$ and $\omega'' \in \boxminus_{\bar Q_k}$. 
Similarly, this means now that $\bar Q_k$ is not $r$-almost crossed by $\omega'$, with $r=2^{-k-20}$. By the symmetry in $\omega',\omega''$ of Proposition \ref{p.almostCross}, there is a constant $\bar C_{T,Q}$ such that the probability of this scenario is bounded from above by $\bar C_{T,Q} \; \eps^2 \alpha_4(\eps,1)^{-1} \alpha_4(2^{-k-20},1)^{-1}$.
\enum

Recall that our domain $D$ is bounded. As such there are finitely many quads $Q\in \QUAD^k$. We thus obtain an upper bound of the form 
\[
2\, \sum_{Q\in \QUAD^k}[C_{T,Q}+\bar C_{T,Q}] \alpha_4(2^{-k-20},1)^{-1} \eps^2 \alpha_4(\eps,1)^{-1} = O_{T,k}(\eps^2) \alpha_4(\eps,1)^{-1}\,,
\]
where the factor 2 handles the second condition $(b)$.  
\qed

We are now able to conclude the proof of Proposition \ref{pr.stab}.

\ni
{\bf Proof of Proposition \ref{pr.stab}}:

Note by our definitions of $t\mapsto \omega_\eta(t)$ and $t\mapsto \omega_\eta^\eps(t)$ in Definitions \ref{d.DP} and \ref{d.dcutoff} that for any time $t\in [0,T]$, the configurations $\omega_\eta(t)$ and $\omega_\eta^\eps(t)$ belong to the  space $\Omega(\omega_\eta(t=0),X)$ introduced above. Furthermore, recall from the definition of the measure $\mu^\eps(\omega_\eta)$ and the construction of $\omega_\eta^\eps(t)$ that for any $t\in[0,T]$, $\omega_\eta(t)(z)=\omega_\eta^\eps(t)(z)$ for all points in $X$ which are initially in $\Piv^\eps(\omega_\eta)$. 
By the definition of $\Piv^\eps$, which relies on an $\eps$-annulus structure, it is easy to check that $\Piv^\eps$ contains the set of all points $z$ which are such that $Z_{\omega_\eta(t=0)}(z)\geq 3\eps$. (See for example remark \ref{r.PIV}). In particular the above Lemma applied with $\tilde \eps= 3\eps$  implies readily that there is a constant $M_{T,K}<\infty$ s.t. for any $2\eta<  \eps<2^{-k-20}$:
\[
\Pb{\exists t\in [0,T],\, K_\HH(\omega_\eta(t), \omega_\eta^\eps(t))<k} < M_{T,k}\, \eps^2 \alpha_4(\eps,1)^{-1}\,,
\]
where the quantity $K_\HH$ was defined in Definition \ref{d.K}. Using this bound with $k=\k(r)$, one obtains for any $2\eta < \eps < 2^{-\k(r)-20}$, 
\begin{align}\label{e.unif2}
\Pb{\sup_{t\in[0,T]} d_{\HH}(\omega_\eta(t), \omega_\eta^\eps(t)) > r} < M_{T,\k(r)}\, \eps^2 \alpha_4(\eps,1)^{-1}\,.
\end{align} 
By the definition of the Skorohod distance $d_{\Sk_T}$ in Definition \ref{d.skorohod}, one thus has
\begin{align*}
\Eb{d_{\Sk_T}(\omega_\eta(\cdot), \omega_\eta^\eps(\cdot))} & \leq r + \diam(\HH) M_{T,\k(r)}\, \eps^2 \alpha_4(\eps,1)^{-1}  \\
&=r + M_{T,\k(r)}\, \eps^2 \alpha_4(\eps,1)^{-1}\,.
\end{align*}

Notice that for any fixed $r>0$, 
\[
\lim_{\eps\to 0} \sup_{0< 2\eta <\eps}  \bigl( r + M_{T,\k(r)}\, \eps^2 \alpha_4(\eps,1)^{-1} \bigr) =r\,.
\]
It is easy to see that this ensures the existence of a continuous function $\psi= \psi_T:[0,1]\to [0,1]$ with $\psi(0)=0$ which is such that  Proposition \ref{pr.stab} holds. (Note that this function is not explicit since it depends on how fast $r\mapsto \k(r)$ diverges). 
\qed

Let us point out that our Lemma \ref{l.main} can be seen as a strengthening of a classical estimate on the stability of the four-arm probabilities in the  near-critical regime, which goes back to Kesten's seminal paper \cite{Kesten}. See also \cite{Nolin,Damron}. Even though we obtain here a strengthening of Kesten's original estimate, our proof is very similar in flavour to the one in \cite{Kesten}, with the important difference that he uses differential inequalities that work well in the monotone coupling, but would break down for the dynamical version. Furthermore, given the stability of the four-arm probability, the above proof can be generalized for alternating $j$-arm events with $j$ even, and also for the one-arm event, since the change in these probabilities is also governed by the pivotal points.

\section{Proof of the main theorem}\label{s.proof}

\subsection{Bounded domain $D$ and  finite time-range $[0,T]$}


We are now ready to prove our main Theorem under the hypothesis we used until now, i.e. $D$ is bounded and one considers dynamical percolation on a finite time-range $[0,T]$. The extensions to $\C$ and $\R_+$ are straightforward and are discussed in the next subsection. 

\begin{theorem}\label{th.MAIN}

The processes defined earlier in Theorem \ref{th.extension} 
\[
\{ t\mapsto \omega_\infty^{\eps}(t)\}_{\eps>0}\,,
\]
converge in probability in $(\Sk_T, d_{\Sk_T})$ as $\eps \to 0$ to a {\bf continuum dynamical percolation}  process: $t\in [0,T] \mapsto \omega_\infty(t)$.  (We will see in Proposition \ref{pr.PROJ} that $\omega_\infty(t)\sim \P_\infty$ for each $t\geq 0$). 

Furthermore, this process $t\in [0,T]\mapsto \omega_\infty(t)$ is the limit in law (under the Skorohod topology on $\Sk_T$) of the discrete dynamical percolation $t\in[0,T] \mapsto \omega_\eta(t)$ as the mesh $\eta\to 0$.  

\end{theorem}

To prove this theorem, we start by constructing the limiting process as an a.s. limit of cut-off processes $t\mapsto \omega_\infty^{\eps_L}(t)$ along a well-chosen sequence $\eps_L$. Namely:

\begin{proposition}\label{}
For a well chosen subsequence $\eps_L \to 0$, the processes $t\mapsto \omega_\infty^{\eps_L}(t)$ converge a.s. for $d_{\Sk_T}$ to a limiting process $t\mapsto \omega_\infty(t)$. 
\end{proposition}

The proof of this Proposition will rely on a coupling with discrete dynamical percolations, but we wish to point out that the process we eventually obtain does not depend on the choice of the coupling, only in principle on the choice of the subsequence $\{\eps_L\}_L$.  
\smallskip

\ni
{\bf Proof:}
Let $\{\eps_L\}_{L\geq 1}$ be a non-increasing sequence converging to 0, to be chosen later. 
for any $L_1<L_2$, by using Theorem \ref{th.etaskoro} (for the coupling defined in  Corollary \ref{c.coupling} with $\eps=\eps_{L_2}<\eps_{L_1}$) together with Proposition \ref{pr.stab} applied successively to $\eps=\eps_{L_1}$ and $\eps=\eps_{L_2}$ and using the triangle inequality, one obtains
\[
\Eb{ d_{\Sk_T}(\omega_\infty^{\eps_{L_1}}(\cdot) , \omega_\infty^{\eps_{L_2}}(\cdot)) } \leq \psi_T(\eps_{L_1})+\psi_T(\eps_{L_2})
\]
Let us fix the sequence $\{\eps_L\}_{L\geq 1}$ so that for any $1\leq L\leq N <\infty$, $\psi_T(\eps_L)\leq 2^{-L}$. 
It follows easily that the sequence $\{ \omega_\infty^{\eps_L}(\cdot) \}_{L\geq 1}$ is a.s. a Cauchy sequence in $\Sk_T$ for $d_{\Sk_T}$. In particular, this defines an a.s. limiting process $t\in [0,T] \mapsto \omega_\infty(t)$. 
\qed

\ni
{\bf Proof of Theorem \ref{th.MAIN}:}
To prove that $t\in[0,T] \mapsto \omega_\eta(t)$ converges in law to the above \cadlag process $t\in [0,T]\mapsto \omega_\infty(t)$, it is enough to show that for any $\delta>0$, one can couple these two processes so that 
\begin{equation}\label{e.TheEnd}
\Eb{d_{\Sk_T}(\omega_\eta(\cdot),\omega_\infty(\cdot))}<\delta\,. 
\end{equation}

This follows easily from the above proof. Indeed, let $L$ be large enough so that $2^{-L}<\delta/10$, then by using the 
coupling defined in  Corollary \ref{c.coupling} with $\eps=\eps_{L}$, one has 
\begin{align*}\label{}
\Eb{d_{\Sk_T}(\omega_\eta(\cdot),\omega_\infty(\cdot))} & \leq 
\Eb{d_{\Sk_T}(\omega_\eta(\cdot),\omega_\eta^{\eps_L}(\cdot))} + \Eb{d_{\Sk_T}(\omega_\eta^{\eps_L}(\cdot),\omega_\infty^{\eps_L}(\cdot))}+ \Eb{d_{\Sk_T}(\omega_\infty^{\eps_L}(\cdot),\omega_\infty(\cdot))} \\
&\leq \psi_T(\eps_L)+ \Eb{d_{\Sk_T}(\omega_\eta^{\eps_L}(\cdot),\omega_\infty^{\eps_L}(\cdot))} + \sum_{N\geq L } \psi_T(N) \\
&\leq 2^{-L} + \Eb{d_{\Sk_T}(\omega_\eta^{\eps_L}(\cdot),\omega_\infty^{\eps_L}(\cdot))} + 2^{-L+1}\,,
\end{align*}
uniformly in $0<2\eta<\eps$. From Theorem \ref{th.etaskoro}, one can choose $\eta$ small enough so that the second term is less than $2^{-L}$ which gives us~\eqref{e.TheEnd}. 

The convergence in probability of the processes $\omega_\infty^\eps(\cdot)$ to $\omega_\infty(\cdot)$ is obtained in the same fashion:
for any $\delta>0$, we wish to show that there is some $\bar\eps>0$ small enough so that for any $\eps<\bar \eps$,
\[
\Pb{d_{\Sk_T}(\omega_\infty^{\eps}(\cdot),\omega_\infty(\cdot))>\delta}<\delta\,.
\] 
Let $\bar \eps$ be such that for any $\eps<\bar \eps$, $\psi_T(\eps)<\delta/2$. Let $L$ be large enough in the above sequence such that $\Eb{d_{\Sk_T}(\omega^{\eps_L}_\infty(\cdot), \omega_\infty(\cdot))}<\delta/2$. By using the exact same argument as above (i.e. coupling with a discrete dynamical configuration), one obtains for any $\eps<\bar \eps$, 
\[
\Eb{\omega^{\eps}_\infty(\cdot), \omega_\infty(\cdot))}<\delta\,,
\]
which implies the desired convergence in probability and thus ends our proof of Theorem \ref{th.MAIN}. \qed

\begin{remark}\label{}
The proof above is somewhat classical. It is very similar for example to the setup of the approximation Theorem 4.28 from Kallenberg's book \cite{Kallenberg}. 
\end{remark}


\subsection{Main theorem in the near-critical case}


\begin{theorem}\label{th.MAINnc}
For any $L>0$, the near-critical ensemble $\lambda\in[-L,L] \mapsto \omega_\eta^\nc(\lambda)$ converges in law (under the Skorohod topology on $\Sk_L$) to a \cadlag process $\lambda\mapsto \omega_\infty^\nc(\lambda)$ as the mesh $\eta\to 0$. 

This limiting process is the limit in probability of the cut-off processes $\lambda\in[-L,L] \mapsto \omega_\infty^{\nc,\eps}(\lambda)$, as $\eps\to 0$.
\end{theorem}

This is proved exactly along the lines of Theorem~\ref{th.MAIN}. Now note that these results 
are not yet satisfactory, because, due to the form of the Skorohod distance $d_{\Sk_L}$, if we fix any $\lambda_0\in(-L,L)$, we cannot conclude from the above Theorem~\ref{th.MAINnc} that $\omega_\eta^\nc(\lambda_0)$ converges in law in $(\HH,d_\HH)$ to $\omega_\infty^\nc(\lambda_0)$ as $\eta\to 0$. We thus need the following theorem, which is not an immediate corollary:

\begin{theorem}\label{th.MARGIN}
For any fixed $\lambda\in \R$, $\omega_\eta^\nc(\lambda)$ converges in law in $(\HH,d_\HH)$ to $\omega_\infty^\nc(\lambda)$, where the ``slice'' $\omega_\infty^\nc(\lambda)$ is extracted from the trajectory obtained in Theorem \ref{th.MAINnc}.  (By taking $L$ sufficiently large, say). 

Furthermore, as in Theorem \ref{th.SScardy} one has 
\begin{align}\label{e.CardyL1}
\lim_{\eta\to 0} \Pb{\omega_\eta^\nc(\lambda) \in \boxminus_Q} = \Pb{\omega_\infty^\nc(\lambda) \in \boxminus_Q}\,.
\end{align}
This may be viewed as a near-critical Cardy's theorem (except we only establish the convergence here, we do not find an explicit formula). We also have:
\begin{align}\label{e.CardyL2}
\lim_{\eta\to 0} \Pb{\omega_\eta^\nc(\lambda) \in \A_j(r,R)} = \Pb{\omega_\infty^\nc(\lambda) \in \A_j(r,R)}\,,
\end{align}
\end{theorem}

\ni
{\bf Proof:} The reason why such a result does not follow readily from Theorem \ref{th.MAINnc} is that there could be some deterministic value of $\lambda$, some $\lambda_0\in\R$ such that there is always a sudden change at that parameter. Of course, such a scenario will not happen, but we do need to prove such a local continuity property:

\begin{proposition}\label{pr.LC}
For any $\lambda_0\in\R$ and any $\alpha>0$, there is some $\delta=\delta(\lambda_0,\alpha)>0$ such that 
\begin{align}\label{e.CONT}
\Pb{\exists \lambda\in(\lambda_0-\delta, \lambda_0+\delta), d_\HH(\omega_\infty^\nc(\lambda_0), \omega_\infty^\nc(\lambda))>\alpha}<\alpha\,.
\end{align}
\end{proposition}

Assume $\lambda_0>0$ and choose $L=2\lambda_0$. 
Since $\lambda \mapsto \omega_\eta^\nc(\lambda)$ converges in law to $\lambda \mapsto \omega_\infty^\nc(\lambda)$ for the topology given by $d_{\Sk_L}$, it is easy to see from the definition of $d_{\Sk_L}$ that it is enough to show that one can find a $\delta=\delta(\lambda_0, \alpha)>0$ sufficiently small so that as $\eta\to 0$, one has 
\begin{align}\label{e.CONT2}
\Pb{\exists \lambda\in(\lambda_0-2\delta, \lambda_0+2\delta), d_\HH(\omega_\eta^\nc(\lambda_0), \omega_\eta^\nc(\lambda))>\alpha/2}<\alpha/2\,.
\end{align}
We leave the details to the reader to recover~\eqref{e.CONT} from~\eqref{e.CONT2} plus the convergence in law of $\omega_\eta^\nc(\cdot)$ to $\omega_\infty^\nc(\cdot)$.

Now, in order to prove ~\eqref{e.CONT2}, recall the definition  of $r\mapsto \k(r)$ from Proposition \ref{pr.CovrtoCovk}.  In particular, it is stronger but sufficient to show that 
\begin{align}\label{e.CONT3}
\Pb{\exists \lambda\in(\lambda_0-2\delta, \lambda_0+2\delta), \omega_\eta^\nc(\lambda) \notin \O_{\k(\alpha/2)}(\omega_\eta^\nc(\lambda_0))} < \alpha/2\,.
\end{align}

In order to prove this, we will use the setup and the notations from Section \ref{s.stability}. In particular let $X=X_{\eta,L}=X_{\eta,2\lambda_0}$ be the set of points which are updated in the interval $[-L,L]$. Note that if $\exists \lambda\in(\lambda_0-2\delta, \lambda_0+2\delta), \omega_\eta^\nc(\lambda) \notin \O_{\k(\alpha/2)}(\omega_\eta^\nc(\lambda_0))$, this means that one can find a point $x\in X_{\eta,L}$ whose label is in $(\lambda_0-2\delta, \lambda_0+2\delta)$ and for which the event $\ev W_x(\eta, 2^{-\k(\alpha/2)-10})$ is satisfied (formally the notation $\ev W_z(i,j)$ used a logarithmic scale but we freely extend the notation to $\ev W_z(r_i,r_j)$ here). Using Lemma \ref{l.main}, the probability of finding at least one such point is dominated by (for $\eta$ sufficiently small):
\begin{align*}
O(\eta^{-2}) \Pb{ x \in X_{\eta,L}\text{ and its label is in } (\lambda_0 -2\delta, \lambda_0 +2 \delta)} C_1(L)\, \alpha_4^\eta(\eta, 2^{-\k(\alpha/2)-10} ) & \\
& \hspace{-12cm} \leq O(\eta^{-2}) O(\delta) \eta^2 \alpha_4^\eta(\eta,1)^{-1} \,  C_1(L)\alpha_4^\eta(\eta,1)\alpha_4(2^{-\k(\alpha/2)-10},1)^{-1} \\
& \hspace{-12cm} \leq C_{L,\alpha} \,\delta\,,
\end{align*}
where $C_{L,\alpha}<\infty$ is a constant which depends only on $L,\alpha$. One can thus find $\delta=\delta(\lambda_0, \alpha)>0$ small enough so that ~\eqref{e.CONT3} holds, thus concluding the proof of Proposition \ref{pr.LC}. \qed

It remains to justify the limits~\eqref{e.CardyL1} and~\eqref{e.CardyL2} in Theorem \ref{th.MARGIN}. It is enough for this to follow the proofs of Corollary 5.2 in \cite{\SchrammSmirnovNoise} and Lemma 2.10 in \cite{\GPSa} by relying when needed on the estimates on near-critical arm-events given by Proposition \ref{pr.ExNC}.  \qed

In fact the proof of Theorem \ref{th.MARGIN}, once adapted to the dynamical setting, easily implies the following interesting and non-trivial fact about the scaling limit of dynamical percolation: 

\begin{proposition}\label{pr.PROJ}
Let $t\mapsto \omega_\infty(t)$ be the process constructed in Theorem \ref{th.MAIN}. Then one has for all $t\geq 0$, 
\[
\omega_\infty(t)\sim \P_\infty\,.
\]
In particular, the process $t\mapsto \omega_\infty(t)$ preserves the measure $\P_\infty$. (This will be important for the simple Markov property in Theorem \ref{th.SMP}). 
\end{proposition}

\subsection{Extension to the full plane and infinite time-range}\label{ss.extensionF}

Extending the above Theorem~\ref{th.MAINnc} to the case of the full plane or to an infinite time-range does not add real additional technicalities. It can be handled using a standard compactification setup. 
For example, one way to proceed  is to consider the following metric on plane percolation configurations: 
\begin{equation}\label{}
d_{\HH_\C}(\omega,\omega'):= \sum_{N\geq 1} 2^{-N} \, d_{\HH_{[-2^N,2^N]}}(\omega,\omega')\,.
\end{equation}
Recall that we assumed for any bounded domain $D$, our distance $d_{\HH_D}$ to be such that $\diam_{d_{\HH_D}}(\HH_D)=1$ so that the above sum is bounded above by one. 

Under the topology given by this metric $d_{\HH_\C}$, it is clear form the above results (Theorem \ref{th.MAIN}) that dynamical and near-critical percolation on the full plane converge to a limiting process for the Skorohod topology on $(\HH_\C, d_{\HH_\C})$. 

Before stating an actual Theorem, let us also extend the setup to an infinite time-range $t\in[0,\infty)$ or $\lambda\in\R$. For this purpose, let us consider as in Lemma \ref{l.SKR} the following Skorohod metric on $\Sk_{\C, (-\infty,\infty)}$ (resp. $\Sk_{\C, [0,\infty)}$), the space of \cadlag  processes from $\R$ (resp. $[0,\infty)$) to $\HH_\C$:

\begin{equation}\label{}
d_{\Sk_{(-\infty,\infty)}}(\omega(\cdot), \tilde\omega(\cdot)) := 
\sum_{k\geq 1} \frac 1 {2^k} d_{\Sk_{\C, [-k,k]}}(\omega(\cdot), \tilde \omega(\cdot)) \,.
\end{equation}

Theorem \ref{th.MAIN} readily implies (since $\Sk_{\C, (-\infty,\infty)}$ and $\Sk_{\C,[0,\infty)}$ are Polish spaces as noted in Lemma \ref{l.SKR}) the following result. 

\begin{theorem}\label{th.MAIN2}
Let $t\mapsto \omega_\eta(t)$ and $\lambda\mapsto \omega_\eta^\nc(\lambda)$ be respectively the dynamical and near-critical percolations (properly renormalized as in Definitions \ref{d.DP} and \ref{d.NCP}) on $\eta \Tg\cap \C=\eta \Tg$.  Then, as the mesh $\eta\to 0$, these processes converge in law respectively to the \cadlag processes $t\mapsto \omega_\infty(t)$ and $\lambda\mapsto \omega_\infty^\nc(\lambda)$ in $\Sk_{\C,[0,\infty)}$ and $\Sk_{\C,(-\infty,\infty)}$ under the  topologies given by $d_{\Sk_{\C,[0,\infty)}}$ and $d_{\Sk_{\C,(-\infty,\infty)}}$.
\end{theorem}

There is one little subtlety which needs to be made more precise here: the construction of the limiting process $t\mapsto \omega_\infty(t)$ (or $\lambda\mapsto \omega_\infty^\nc(\lambda)$). Indeed, in order to prove the existence of this limiting process, one proceeds as in the proof of Theorem \ref{th.MAIN} by approximations using cut-off processes $t\mapsto \omega_\infty^\eps(t)$ except that here the cut-off $\eps$ will play two different roles: focusing on $\eps$-pivotal points as previously and also focusing on the percolation configurations only on the domain $[-1/\eps, 1/\eps]^2$. (Otherwise, one would have infinitely many switches on any interval $[0,T]$). As $\eps\to 0$, these cut-off processes converge in probability to a limiting one as in Theorem \ref{th.MAIN}. 
This is the only additional technicality needed to prove Theorem \ref{th.MAIN2}.

\section{Conformal covariance property, the infinite cluster and correlation length of the n.c.~model}\label{s.CL}

\subsection{Conformal covariance of dynamical and near-critical percolation}

Before stating our result, we need to introduce a slight generalisation of our dynamical and near-critical percolation models originally defined in Definitions~\ref{d.DP} and~\ref{d.NCP}:

\begin{definition}\label{d.INC}
Let $\Omega\subset \C$ be a domain of the plane and let $\phi : \Omega \to (0,\infty)$ be any continuous function. 

We will consider the dynamical percolation process $t\mapsto \omega_\eta^\phi(t)$ which starts at $\omega_\eta^\phi(t=0)\sim \P_\eta$ and for which sites $x\in \eta\, \Tg$ are updated independently of each other with inhomogeneous rate $r^\phi(\eta,x):= \phi(x) \frac{\eta^2} {\alpha_4(\eta,1)}$.  (As such, this dynamical percolation is mixing faster in areas of the domain $\Omega$ where the function $\phi$ is large).

Similarly, we will consider the near-critical coupling $(\omega_\eta^{\nc,\phi}(\lambda))_{\lambda\in \R}$, where $\omega_\eta^{\nc,\phi}(\lambda=0)\sim \P_\eta$, and as $\lambda$ increases, white hexagons $x\in \eta \Tg$ switch to black hexagons at same rate  $r^\phi(\eta,x):= \phi(x) \frac{\eta^2} {\alpha_4(\eta,1)}$. This near-critical percolation $\omega_\eta^{\nc,\phi} (\lambda)$ corresponds exactly to a percolation configuration on $\eta \Tg$ with inhomogeneous parameter $p(x) =p_c + 1-e^{-\lambda\, r^\phi(\eta,x) }\sim p_c+\lambda \phi(x) r(\eta)$. 
\end{definition}

Following the exact same proof as in the rest of the paper, one can define cut-off processes $t\mapsto \omega_\eta^{\phi,\eps}(t)$ by only following the evolution of points in $\Piv^\eps(\omega_\eta(t=0))$. In the same way as before, it can be shown that these processes converge in law (in $(\Sk,d_{\Sk})$) to a process $t\mapsto \omega_\infty^{\phi,\eps}(t)$ and it is straightforward to establish the following analog of Theorem \ref{th.MAIN}:

\begin{theorem}\label{th.INC}
Let $\Omega\subset \C$ be a domain and let $\phi : \Omega \to (0,\infty)$ be any continuous function. Then the processes
\[
\{ t\mapsto \omega_\infty^{\eps,\phi}(t)\}_{\eps>0}
\]
converge in probability in $(\Sk_T, d_{\Sk_T})$ as $\eps \to 0$ to a {\bf continuum dynamical percolation}  process: $t\in [0,T] \mapsto \omega_\infty^\phi(t)$.

Furthermore, this process $t\in [0,T]\mapsto \omega_\infty^\phi(t)$ is the limit in law (under the Skorohod topology on $\Sk_T$) of the discrete dynamical percolation $t\in[0,T] \mapsto \omega_\eta^\phi(t)$ as the mesh $\eta\to 0$. 
\end{theorem}

The same theorem holds with near-critical instead of dynamical percolation. We are now ready to state our main conformal covariance result:

\begin{theorem}\label{th.cc}
Assume that $f: \Omega \longrightarrow \tilde \Omega$ is a conformal map with $|f'|$ being bounded away from zero and infinity. (For instance, a conformal map between so-called Dini-smooth domains is always like this, see \cite[Theorem 10.2]{Pommerenke}.) Then, if $\omega_\infty(\cdot)$ (resp. $\omega_\infty^\nc(\cdot)$) is a continuous dynamical percolation (resp.~n.c.~percolation), then the image of these processes by $f$, i.e., the \cadlag processes $t\mapsto f(\omega_\infty(t))$ (resp.~$\lambda \mapsto f(\omega_\infty^\nc(\lambda))$) have the same law as the following processes defined on $\tilde \Omega$:
\bnum
\item $t \mapsto \omega_\infty^{\phi}(t)$ in the dynamical case
\item $\lambda \mapsto \omega_\infty^{\nc,\phi}(\lambda)$ in the near-critical case, 
\enum
where the function $\phi$ on $\tilde \Omega$ is defined  by, 
\[
\phi(f(z)):= |f'(z)|^{-3/4}\,, \forall z\in \Omega\,.
\]
\end{theorem}

\begin{remark}
If $\omega\in \HH_\Omega$, the configuration image $f(\omega)\in\HH_{\tilde \Omega}$ is well-defined. See the end of Subsection 2.3 in \cite{\GPSa} for a discussion why the measure $\P_\infty=\P_{\infty,\Omega}$ is conformally invariant. 
\end{remark}

\begin{corollary}\label{c.SCALE}
\ni
\bi
\item The scaling limits of dynamical and near-critical percolation on $\eta \Tg$ as $\eta\to 0$ are rotationally invariant.
\item They also have a form of scaling invariance which can be stated as follows. For any scaling parameter $\alpha>0$ and any $\omega\in \HH$, we will denote by $\alpha\cdot \omega$ the image by $z\mapsto \alpha\, z$ of the configuration $\omega$. With these notations, we have the following identities in law:
\bnum
\item \[
\Bigl( \lambda \mapsto \alpha\cdot \omega_\infty^\nc(\lambda) \Bigr) \overset{(d)}{=} \Bigl( \lambda \mapsto \omega_\infty^\nc(\alpha^{-3/4} \lambda) \Bigr)
\]
\item
\[
\Bigl( t\geq 0 \mapsto \alpha\cdot \omega_\infty(t)\Bigr) \overset{(d)}{=} \Bigl( t \mapsto \omega_\infty(\alpha^{-3/4} t) \Bigr)
\]
\enum
\ei
\end{corollary}

\ni
{\bf Proof of Theorem \ref{th.cc}.}
We start with the following lemma:

\begin{lemma}\label{l.UC}
Let $f: \Omega \to \tilde \Omega$ be a conformal map with $|f'|$ bounded away from zero and infinity.
Let $\Sk$ and $\tilde \Sk$ be respectively the space of \cadlag trajectories in $\HH_\Omega$ and $\HH_{\tilde \Omega}$ endowed with the Skorohod distance defined in Lemma~\ref{l.SKR}. Then (with a slight abuse of notation), the map 
\[
\begin{array}{llcl}
f  : & \Sk & \to & \tilde \Sk \\
& \omega(\cdot) &\mapsto& f(\omega(\cdot))
\end{array}
\]
is uniformly continuous. 
\end{lemma}

\ni
{\bf Proof.} 
Let us prove the lemma for a finite time-range $\Sk_T$ for any $T>0$. The extension to the infinite time-range is only technical. 
Let $\alpha>0$. Suppose $d_{\Sk_T}(\omega(\cdot), \omega'(\cdot))<\alpha$. One can thus find a reparametrization $\phi : [0,T] \to [0,T]$, such that $\| \phi \| < \alpha$ and $\sup_{t\in [0,T]} d_\HH(\omega(t), \omega'(\phi(t)))<\alpha$.    
Now, by assumption, we have two constants $c_1,c_2\in(0,\infty)$ such that
\begin{align}\label{e.c1}
c_1< \inf_\Omega |f'(z)| \leq \sup_\Omega |f'(z)| <c_2\,. 
\end{align}
Using Section~\ref{s.unif} together with the above bounds, one can show that $f(\omega(t))$ and $f(\omega'(\phi(t)))$ are also close. Indeed, the map (still with an abuse of notation)
\[
\begin{array}{llcl}
f  : & \HH_\Omega & \to & \HH_{\tilde \Omega} \\
& \omega &\mapsto& f(\omega)
\end{array}
\]
is continuous and thus uniformly continuous, since $(\HH_\Omega,d_\HH)$ is compact. If $\alpha\mapsto g(\alpha)$ denotes its modulus of continuity, we thus have 
\[
\sup_{t\in[0,T]} d_\HH(f(\omega(t)), f(\omega'(\phi(t)))) < g(\alpha)\,.
\]
Since $\| \phi \| < \alpha$, we have shown that $d_\Sk(\omega(\cdot), \omega'(\cdot))<\alpha$ implies 
$d_{\Sk_T}(f(\omega(\cdot)), f(\omega'(\cdot)))<\alpha+ g(\alpha)$ which ends the proof of the lemma (modulo the easy extension to the infinite time-range $\Sk_{[0,\infty)}$). \qed

This lemma is useful for the following reason: 
we have from Theorem \ref{th.MAIN} that $\omega_\infty^\eps(\cdot)$ converges in probability in $\Sk$ towards $\omega_\infty(\cdot)$ as $\eps\to 0$. By the above lemma, this implies that $f(\omega_\infty^\eps(\cdot))$ converges in probability in $\tilde \Sk$ to $f(\omega_\infty(\cdot))$ as $\eps\to 0$.

It remains to show (by the uniqueness of the limit in probability) that $f(\omega_\infty^\eps(\cdot))$ also converges in probability to the process $\omega_\infty^\phi(\cdot)$ defined in Theorem \ref{th.cc} on the domain $\tilde \Omega$. 
 
For this, recall that the cut-off dynamics $\omega_\infty^\eps(\cdot)$ is based on the set of $\eps$-pivotal points $\Piv^\eps(\omega_\infty)$ defined using the grid $\eps\Z^2$. On $\tilde \Omega$, let us consider the image of the grid $\eps \Z^2$ by the conformal map $f$. Call this grid $F_\eps$. Let $\Piv_f^\eps=\Piv_f^\eps(\tilde \omega_\infty)$ denote the set of $F_\eps$-important points for $\tilde \omega_\infty=f(\omega_\infty)$, a sample of a continuum critical percolation on $\tilde \Omega$.
It follows from the construction of $\omega_\infty^\eps(\cdot)$ that $f(\omega_\infty^\eps(\cdot))$ is exactly the \cadlag process which starts at $\tilde \omega_\infty$ and is updated according to a Poisson point process $\tilde \PPP$ of intensity measure $f_*(\mu^\eps(\omega_\infty)) \times dt$. Using the following two facts, one can conclude that $f(\omega_\infty^\eps(\cdot))$ and $\omega_\infty^{\eps,\phi}(\cdot)$ have the same limit in probability as $\eps\to 0$ (which thus concludes the proof):
\bnum
\item Theorem 6.1 in \cite{\GPSa} shows that the push-forward measure $f_*(\mu^\eps(\omega_\infty))$ satisfies for any point $z\in \Omega$:
\[
\frac{d f_*(\mu^\eps(\omega_\infty))}{d \mu^{\eps,f}(\tilde \omega)}(f(z)) = |f'(z)|^{-3/4}\,,
\]
where $\mu^{\eps,f}=\mu^{\eps,f}(\tilde \omega_\infty)$ stands for the measure on the $F_\eps$-important points of $\tilde \omega_\infty$. This item makes the link with the process $\omega_\infty^{\phi}(\cdot)$ in the statement of the theorem, with $\phi(f(z)):=|f'(z)|^{-3/4}$. 
\item From equation (\ref{e.c1}), one can easily check that
\[
\Piv^{10 c_2 \eps} (\tilde \omega_\infty) \subset \Piv_f^\eps (\tilde \omega_\infty) \subset \Piv^{c_1 \eps/10} (\tilde \omega_\infty)\,.
\]
By going back to the discrete and using the stability Section \ref{s.stability}, this shows that the cut-off dynamics $f(\omega_\infty^\eps(\cdot))$ defined on the distorted scale $F_\eps$ and the cut-off dynamics $\omega_\infty^{\eps,\phi}$ which is defined on a proper $\eps$-square grid, have the same limit as $\eps\to 0$. 
\enum
This finishes the proof of Theorem~\ref{th.cc}.\qed

\subsection{Infinite cluster and correlation length}\label{ss.CL}

\begin{theorem}\label{th.CL}
For any $\lambda>0$, there is a.s. an infinite cluster in $\omega_\infty^\nc(\lambda)$ in the sense that, for any $r>0$,
\begin{align}\label{e.inftycluster}
\lim_{\eta\to 0} \Pb{\omega_\eta^\nc(\lambda) \in \A_1(r,\infty)} = \PB{\omega_\infty^\nc(\lambda) \in \bigcap_{R>r} \A_1(r,R)}\,,
\end{align}
and this probability tends to 1 as $r\to\infty$, hence one can find some random $r>0$ such that the event on the right hand side occurs. 

Furthermore, as in the discrete model, one can define a notion of {\bf correlation length} for $\omega^\nc_\infty(\lambda), \lambda>0$. 
In fact, let us give two different such definitions: for any $\lambda>0$, define
\begin{align}\label{}
\begin{cases}
L_1(\lambda) &:= \inf \{r>0: \Pb{\omega_\infty^\nc(\lambda)\in \bigcap_{R>r} \A_1(r,R)} >1/2 \} \\
L_2(\lambda)&:=  \inf \{r>0: \Pb{\omega_\infty^\nc(\lambda)\text{ crosses } [0,2r]\times [0,r] } >0.99 \}\,.
\end{cases}
\end{align}
These correlation lengths have the following behaviour:  there exist two constants $c_1,c_2\in(0,\infty)$, s.t.
\begin{align}\label{}
\begin{cases}
L_1(\lambda) = c_1 \lambda^{-4/3} \\
L_2(\lambda)=  c_2 \lambda^{-4/3}\,.
\end{cases}
\end{align}
\end{theorem}

\ni
{\bf Proof.} 
Recall from Theorem \ref{th.MARGIN} that  for any $0<r<R$, 
\begin{align}\label{e.LimR}
\Pb{\omega_\infty(\lambda)\in \A_1(r,R)} = \lim_{\eta\to 0} \Pb{\omega_\eta(\lambda)\in \A_1(r,R)}\,.
\end{align}
However, we do not know this convergence for the infinite intersection of events in~(\ref{e.inftycluster}), hence we need to work a little bit.

Using the notations of Kesten, on the non-renormalized lattice $\Tg$, let $L_\eps(p)$ be the {\bf correlation length} defined as the smallest scale $n\geq 0$ such that the probability under $\P_p$ to cross the rectangle $[0,2n]\times[0,n]$ is larger than $1-\eps$. Kesten's \cite{Kesten} implies that for any $\eps,\eps'>0$, then as $p\to p_c$, 
\[
L_\eps(p)\asymp L_{\eps'}(p) \asymp L(p):= \inf \{N\geq 1: N^2 \alpha_4(N) \geq 1/|p-p_c|\}\,.
\] 
See also the survey \cite{Nolin}. Furthermore, it is well known that for any $\delta>0$, one can find $\eps>0$ such that for any $p>p_c$,  $\P_p\bigl[ \omega\in \A_1(L_\eps(p), \infty)\bigr]>1-\delta$. See for example \cite{Bollobas}. From these results, together with the large probability of having an open circuit in any annulus of large conformal modulus even at criticality, we also get that, for any $\eps, \delta, a>0$, if we take $b>b_0(\eps,a,\delta)$ large enough, then
\begin{align}\label{e.abinfty}
\P_p\bigl[ \omega\in  \A_1\big(a L_\eps(p), b L_\eps(p)\big) \setminus \A_1\big(a L_\eps(p), \infty\big)\bigr] < \delta\,.
\end{align}

One can introduce the same notion of correlation length in the setting of our near-critical coupling (see Definition \ref{d.NCP}), except the lattice is now renormalized. More precisely, for any $\eps,\eta,\lambda$, define 
\begin{align*}
\tilde L_{\eps,\eta}(\lambda):= \inf \{r>0:\, \Pb{\omega_\eta^\nc(\lambda) \text{ crosses the rectangle }  [0,2r]\times [0,r]} > 1-\eps\}. 
\end{align*}
By our choice of rescaling in Definition \ref{d.NCP}, the above results from Kesten readily translate as follows: for any values of $\eps,\lambda >0$, one has 
\begin{align}\label{e.CorL}
0< \liminf_{\eta\to 0}  \tilde L_{\eps,\eta}(\lambda) \leq \limsup_{\eta\to 0} \tilde L_{\eps,\eta}(\lambda) <\infty\,.
\end{align}
Furthermore, for  any $\delta>0$, one can also choose $\eps$ small enough so that for any $\eta\in(0,1]$:
\[
\Pb{\omega_\eta^\nc(\lambda) \in \A_1(\tilde L_{\eps,\eta}(\lambda),\infty)}>1-\delta\,,
\] 
or using~\eqref{e.CorL}, for $r>0$ large enough,
\begin{align}\label{e.rinfty}
\Pb{\omega_\eta^\nc(\lambda) \in \A_1(r,\infty)}>1-\delta\,.
\end{align}
Similarly to (\ref{e.abinfty}), we also get that for any $\delta, r>0$, if $R>R_0(\lambda,r,\delta)$ is large enough, then
\begin{align}\label{e.rRinfty}
\Pb{ \omega_\eta^\nc(\lambda)\in  \A_1(r, R) \setminus \A_1(r, \infty)} < \delta\,,
\end{align}
for all $\eta>0$ small enough.

Now, this finite $R$ approximation~(\ref{e.rRinfty}), together with~\eqref{e.LimR}, imply~(\ref{e.inftycluster}). That the probability tends to 1 follows from~(\ref{e.rinfty}).

The above arguments clearly show that the  correlation lengths $L_1(\lambda)$ and $L_2(\lambda)$ are finite and nonzero. The exact formulas for them follow from the scaling covariance result in Corollary \ref{c.SCALE}. Indeed, one needs to scale $\omega_\infty^\nc(\lambda)$ by a factor $\lambda^{4/3}$ in order to obtain the same law as $\omega_\infty^\nc(\lambda=1)$.  \qed

%
%

\ni
{\bf Proof of Corollary~\ref{c.BCL}.} The correlation length $L(p)$ that we use here is basically the inverse of the rate function $r(\eta)$ defined in~\eqref{e.r}, except that we do not know that $r(\eta)$ is monotone, hence the ``inverse'' is a little loosely defined. Nevertheless, there is a ratio limit theorem for $\alpha_4(n)$ in \cite[Proposition 4.7]{GPS2a}, saying that, for any $t>0$ fixed, 
$$
\lim_{\eta\to 0} \frac{r(t\eta)}{r(\eta)}=t^{3/4}\,,
$$
which immediately implies that 
\begin{equation}\label{e.qinverse}
\lim_{n\to 0} \eta L(r(\eta)) = 1\qquad\textrm{and}\qquad \lim_{p\to p_c} \frac{r(1/L(p))}{|p-p_c|} =1\,.
\end{equation} 

The configuration $\omega_\eta^\nc(\lambda)$, as $\eta\to 0$, is just percolation $\omega_p$ at density $p-p_c \sim \frac{1}{2}\lambda r(\eta)$. Therefore, when we consider percolation $\omega_p$ on a lattice scaled down by $L(p)$, that is, when we take $\eta=1/L(p)$, then~\eqref{e.qinverse} says that $r(\eta) \sim |p-p_c|$ as $p\to p_c$, hence $\frac{1}{L(p)}\omega_p$ is close to $\omega_\eta^\nc(\lambda=2)$ for $p\to p_c+$, and to  $\omega_\eta^\nc(\lambda=-2)$ for $p\to p_c-$, as claimed.
\qed

\section{Markov property, associated semigroup}\label{s.markov}

\subsection{Simple Markov property for $t\mapsto \omega_\infty(t)$}

We wish to prove the following simple Markov property.

\begin{theorem}\label{th.SMP}
The scaling limit of dynamical percolation is a simple Markov process with values in $(\HH, d_\HH)$. Furthermore this process is reversible w.r.t. the measure $\P_\infty$. 

As such, one obtains a semi-group $(P_t)_{t\geq 0}$ on $\mathcal{B}(\HH)$, the space of bounded Borel measurable functions on $(\HH, d_\HH)$.
\end{theorem}

%

\ni
{\bf Proof:}
\ni

Fix $0<s<t$. We wish to prove that 
\begin{align*}\label{}
\mathcal{L}[\omega_\infty(t) \md (\omega_\infty(u))_{0\leq u \leq s}] = \mathcal{L}[\omega_\infty(t) \md \omega_\infty(s)]\,.
\end{align*}

To prove this identity in law, we will build the limiting process $\omega_\infty(\cdot)$ in a way which is well suited to the above conditioning. Instead of building our process using the critical ``slice'' $\omega_\infty(t=0)$, we will shift things so that one builds our process from the slice $\omega_\infty(s)$ which by Proposition \ref{pr.PROJ} is known to satisfy as well  $\omega_\infty(s) \sim \P_\infty$.  
One proceeds as follows (the details are omitted):

\bnum
\item We sample $\omega_\infty(s)\sim \P_\infty$. 
\item We choose $\eps>0$ very small and consider $\mu^\eps=\mu^\eps(\omega_\infty(s))$ the measure on the $\eps$-pivotal points of $\omega_\infty(s)$ we used continuously so far. 
\item Knowing $\mu^\eps$, we sample the Poisson Point processes $\PPP_{[s,t]}$ on $D\times[s,t]$ and $\PPP_{[0,s]}$ on $D\times [0,s]$ independently of each other and respectively with intensity measures given by $\mu^\eps(dx) du 1_{[s,t]}$ and $\mu^\eps(dx) du 1_{[0,s]}$. 
\item Using these two PPPs, we proceed as in Section \ref{s.cutoff} to construct a \cadlag trajectory $u\in[s,t]\mapsto \omega_\infty^\eps(u)$ and a ``c\`agl\`ad'' trajectory $u\in[0,s]\mapsto \omega_\infty^\eps(s-u)$. 
\item From the above construction, note that conditionally on $\omega_\infty(s)$, these two processes are conditionally independent. (This results from the fact that $\mu^\eps$ is measurable w.r.t. $\omega_\infty(s)\sim \P_\infty$). 
\item As in the proof of Theorem \ref{th.MAIN}, we obtain a limiting process $u\in[0,t] \mapsto \bar \omega_\infty(u)$ as $\eps$ to 0. The convergence is in probability in $\Sk_t$ as $\eps\to 0$ and the above conditional independence property survives as $\eps\to 0$. In particular, one has 
\begin{align}\label{e.SM}
\mathcal{L}[\bar \omega_\infty(t) \md (\bar \omega_\infty(u))_{0\leq u \leq s}] & = \mathcal{L}[\bar \omega_\infty(t) \md \bar \omega_\infty(s)] = \mathcal{L}[\bar \omega_\infty(t) \md  \omega_\infty(s)] 
\end{align}
\item As in Theorem \ref{th.MAIN}, this process $\bar \omega_\infty(\cdot)$ is the limiting law as $\eta\to 0$ of $u\in[0,t] \mapsto \omega_\eta(u)$. By the uniqueness of the limit, one has as a process in $\Sk_t$, $\bar \omega_\infty(\cdot) \overset{(d)}{=} \omega_\infty(\cdot)$. 
In particular, property~\eqref{e.SM} is satisfied for the process $\omega_\infty(\cdot)$ which implies the desired simple Markov property. 
\enum

In order to obtain a proper Markov process together with its semi-group $(P_t)_{t\geq 0}$ on $\mathcal{B}(\HH)$, one needs to be a bit more careful and define a random \cadlag process starting from any possible initial configuration $\omega\in\HH$. So far, it is implicit in Theorems \ref{th.extension}, \ref{th.MAIN} and Proposition \ref{pr.PROJ} that we only defined a random \cadlag process {\bf almost surely} in the initial configuration $\omega_\infty(0) \sim \P_\infty$ (for example, it could be that the measure $\mu^\eps=\mu^\eps(\omega_\infty)$ is infinite which is an event of measure 0 and is thus included in the event $A^c$ introduced in the proof of Theorem \ref{th.extension}). Formally, let $B\subset \HH$ be the set of initial configurations $\omega$ such that {\bf almost surely} in the additional randomness required to sample the Poisson Point Processes $\PPP(\mu^\eps(\omega))$, the random trajectory $t\mapsto \omega_\infty^{\eps_L}(\omega)$ converges  to a limiting \cadlag process in $\Sk$ (where the sequence $\eps_L$ is the one used in the proof of Theorem \ref{th.MAIN} to construct our limiting process). By the proofs of Theorems \ref{th.extension} and \ref{th.MAIN}, we have that $\P_\infty [B]=1$. As in the proof of Theorem~\ref{th.extension}, if the initial configuration $\omega$ is in $B^c \subset \HH$, then we define our random process $\omega_\infty(t)$ to be the constant process equal to $\omega$.  Since we know from the above argument that starting from $\omega_\infty(t=0)$, we have $\omega_\infty(t) \sim \P_\infty$, this construction implies that $\Pb{\omega_\infty(t)\in B^c} = \P_\infty[B^c]=0$, which is enough for the simple Markov property and the existence of a semi-group.  Note that in order to prove a strong Markov property, one would need to check (in particular) that this set $B$ is polar for the dynamics $t\geq 0\mapsto \omega_\infty(t)$ starting at $\P_\infty$. See Remark \ref{r.feller} and the question below.
\qed

\begin{remark}\label{}
Note that this proof uses in an essential manner the invariance of $\P_\eta$, $\P_\infty$ as well as our way of producing a trajectory $\omega_\infty(\cdot)$ in a measurable manner w.r.t. an initial slice $\omega_\infty(t=0)$.
\end{remark}

\subsection{Simple Markov property for $\lambda \mapsto \omega_\infty^\nc(\lambda)$}
It is tempting to claim that the simple Markov property is satisfied in the same fashion by the near-critical process $\lambda \mapsto \omega_\infty^\nc(\lambda)$ since so far the dynamical and near-critical regimes did share the same level of difficulty. This is no longer the case here. The additional difficulty in the near-critical case is due to the fact that the law $\P_\infty$ is not invariant along the process $\lambda \mapsto \omega_\infty^\nc(\lambda)$. In particular, the above proof for the simple Markov property of $t\mapsto \omega_\infty(t)$ does not work in the near-critical setting. Nevertheless, we can prove that the simple Markov property holds for $\omega_\infty^\nc(\cdot)$, namely:
\begin{theorem}\label{th.mnc}
The scaling limit of near-critical percolation $\lambda\mapsto \omega_\infty^\nc(\lambda)$ is a simple Markov process with values in $(\HH, d_\HH)$. 
\end{theorem}
This scaling limit may seem to be an inhomogeneous Markov process in $\HH$. It is not (the asymmetry comes from the non-reversible nature of the near-critical dynamics). 

\begin{theorem}\label{th.inh}
$\lambda\mapsto \omega_\infty^\nc(\lambda)$ is a {\bf homogeneous} non-reversible Markov process.
\end{theorem}


\ni
{\bf Proof:}

For any fixed $-\infty<\lambda_1 <\lambda_2<\infty$, we wish to show that
\begin{align*}\label{}
\mathcal{L}[\omega_\infty^\nc(\lambda_2) \md (\omega_\infty^\nc(\lambda))_{\lambda \leq \lambda_1}] = \mathcal{L}[\omega_\infty^\nc(\lambda_2) \md \omega_\infty^\nc(\lambda_1)]\,.
\end{align*}
The strategy we wish to follow is the same as the one used for dynamical percolation, i.e. to build the process $\omega_\infty^\nc(\cdot)$ from the near-critical slice $\omega_\infty^\nc(\lambda_1)$ instead of from the critical one $\omega_\infty^\nc(\lambda=0)$. The same approach works but several non-trivial steps need to be checked/adapted.
\bnum
\item Now that $\omega_\infty^\nc(\lambda_1)$ is well defined, one can sample such a near-critical slice. We will denote $\omega_\infty^\nc (\lambda) \sim \P_{\lambda,\infty}$. 
\item We need an analog of the measure $\mu^\eps$ which was defined in a measurable manner w.r.t. $\omega_\infty\sim \P_\infty$ except here that $\omega_\infty^\nc(\lambda_1)$ follows a different law. This means that the work done in \cite{\GPSa} to build a {\bf pivotal measure} needs to be extended to the near-critical regime. 
We will show in Theorem \ref{th.PMNC} that for any $\lambda_1<\lambda_2 <\ldots <\lambda_n$, this is the {\bf same} measurable function $\omega\in(\HH,d_\HH) \mapsto \mu^\eps(\omega)$ which gives the appropriate pivotal measures respectively for the measures $\P_{\lambda_1,\infty},\ldots, \P_{\lambda_n,\infty}$. In this sense, Theorem \ref{th.PMNC} implies that $\lambda\mapsto \omega_\infty^\nc(\lambda)$
 is indeed an {\bf homogeneous} non-reversible Markov process as stated in Theorem \ref{th.inh}. 
\item Once the work from \cite{\GPSa} is extended thanks to Theorem \ref{th.PMNC}, it remains to check that all the proofs of the present paper do extend to this regime. The main things to be checked are the stability Section \ref{s.stability} as well as the arguments from the discrete used everywhere in Sections \ref{s.network} and \ref{s.cutoff}. It is is easy to check that Proposition \ref{pr.ExNC} enables to extend these sections to the near-critical regime. 
\item With these extensions at hand, the proof used for the simple Markov property of the dynamical percolation $\omega_\infty(\cdot)$ works in the same manner. 
\enum

To prove Theorems \ref{th.mnc} and \ref{th.inh} following the same strategy as for dynamical percolation, we are thus left with the following two statements. 
\begin{theorem}\label{th.PMNC}
For any $-\infty< \lambda_1<\lambda_2 <\ldots< \lambda_n <\infty$ and any $\eps>0$, one can define a measure $\mu^\eps$ which is Borel measurable w.r.t. $\omega\in(\HH,d_\HH)$ and which is such that for any $\lambda \in \{ \lambda_i\}_{1\leq i \leq n}$, one has 
\begin{align*}\label{}
(\omega_\eta^\nc(\lambda),\mu^\eps(\omega_\eta^\nc(\lambda))) \overset{(d)}{\longrightarrow} 
(\omega_\infty^\nc(\lambda), \mu^\eps(\omega_\infty^\nc(\lambda)))\,,
\end{align*}
as the mesh $\eta\to 0$. 
\end{theorem}

\begin{proposition}\label{pr.ExNC}
For any $\lambda\in \R$, 
\ni
\bnum
\item  As $\eta\to 0$, the {\bf separation of arms} phenomenon holds for $\omega_\eta(\lambda)$ up to scales of order $O(1)$. In particular, for any scale $R>0$, there is a constant $C_R\in(0,\infty)$ such that for any $0<r_1<r_2<r_3<R$ and uniformly in $\eta<r_1$, one has 
\begin{align*}
C_R^{-1} \Pb{\omega_\eta^\nc(\lambda)\in\A_4(r_1,r_3)} & \leq   \Pb{\omega_\eta^\nc(\lambda)\in\A_4(r_1,r_2)}\Pb{\omega_\eta^\nc(\lambda)\in\A_4(r_2,r_3)} \\
& \leq C_R \Pb{\omega_\eta^\nc(\lambda)\in\A_4(r_1,r_3)}
\end{align*}
\item There is an $\eps>0$ (independent of $\lambda$) such that for any $R>0$, there is a constant $C=C_{R,\lambda}<\infty$ such that for any $0<r<R$, one has  uniformly in $\eta<r$: 
\[
\Pb{\omega_\eta(\lambda)\in \A_6(r,R)} \leq C\, (r/R)^{2+\eps}\,,
\]
and the probability of a three--arm event for $\omega_\eta(\lambda)$ in $\Hyp$ between radii $r$ and $R$ is bounded above by $C\, (r/R)^2$ uniformly in $\eta<r$.  
\enum
\end{proposition}

In fact, Theorem \ref{th.PMNC} will rely partly on Proposition \ref{pr.ExNC}. Hence we start with a sketch of proof of the latter proposition.
\medskip

\ni
{\bf Sketch of proof of Proposition \ref{pr.ExNC}.}
Item 1 follows from the fact that if one fixes a macroscopic scale $R>0$ as well as a fixed  $\lambda\in \R$, then the RSW Theorem holds for rectangles of diameter bounded by $R$. This can be seen for example by using the results on the correlation length by Kesten (see the discussion in Subsection \ref{ss.CL}). Once we have a RSW Theorem, then separation of arms as well as the quasi-mutliplicativity property can be established below the scale $R$ by now classical arguments. See for example \cite{PC,Nolin}.

There are two ways to see why item 2 holds: either by generalizing Lemma \ref{l.main} to the case of these arm-events, or by using the fact that a RSW Theorem classically implies that plane-5-arm exponent and half-plane 3-arm exponent are equal to 2. The 6-arm estimate then follows from Reimer's inequality.  \qed

\ni
{\bf Sketch of proof of Theorem \ref{th.PMNC}.} As in \cite{\GPSa}, we fix an annulus $A$ and we wish to construct a measure $\mu^A=\mu^A(\omega_\infty^\nc(\lambda))$ which is the limit in law of the counting measures for $\omega_\eta^\nc(\lambda)$ on its $A$-pivotal points. When $\lambda=0$ (i.e. the critical case), this measure was well approximated  (in the $L^2$ sense) on the discrete lattice $\eta \Tg$ by a deterministic constant times the number $Y$ of mesoscopic squares of size $\eps$ which intersect the set of $A$-pivotal points. (Note that the parameter $\eps$ does not play the same role in \cite{\GPSa} and in the present paper). This deterministic constant was given by 
\[
\beta=\beta(\eta,\eps):= \Eb{x_0 \md \A_0(2\eps, 1)}\,,
\]
where we use here the same notations as in \cite{\GPSa}. 
The same strategy/proof as in \cite{\GPSa} applies in the present near-critical case except that some work is needed to identify the deterministic constant $\beta_\lambda=\beta_\lambda(\eta,\eps)$ when $\lambda\neq 0$. Two issues in particular need to be addressed:
\bnum
\item First of all, in order to obtain the same measurable map $\omega\mapsto \mu^A(\omega)$, whatever $\lambda$ is, one needs to show that as $\eps$ and $\eta/\eps$ go to zero, one has 
\begin{align}\label{e,I1}
\beta_\lambda(\eta,\eps) \sim \beta(\eta,\eps)\,.
\end{align}
The fact that the proportional factor $\beta_\lambda$ is asymptotically identical to the critical case ensures that the measurable map $\mu^A(\cdot)$ does not depend on $\lambda$. 
\item One of the main technical problems that arises in \cite{\GPSa} comes form the fact that there is an additional conditioning that one needs to handle and the proportional factor that is eventually used is the following one (we rely here on notations from \cite{\GPS} to which we refer):
\[
\hat \beta= \hat \beta(\eta,\eps):= \Eb{x_0 \md \A_0(2\eps, 1), U_0 =1}\,.
\]
Lemma 4.7 in \cite{\GPSa} shows that $\beta=\hat \beta(1+o(1))$ as $\eps$ and $\eta/\eps$ go to zero. Unfortunately the proof of Lemma 4.7 does not apply in our case since it relies on a color-switching argument which only works if $\lambda=0$. 
\enum

From the above construction, Theorem \ref{th.PMNC} is proved exactly as the main theorem in \cite{\GPSa} (and using when needed the estimates from Proposition~\ref{pr.ExNC}), assuming that the following  lemma holds:

\begin{lemma}\label{l.I1}
Let $\lambda\in \R$ be fixed. With the same notations as in \cite{\GPSa}, let 
\[
\begin{cases}
\beta_\lambda= \beta_\lambda(\eta,\eps):= \Eb{x_0(\omega_\eta^\nc(\lambda)) \md \A_0(2\eps,1)}\\
\hat\beta_\lambda= \hat\beta_\lambda(\eta,\eps):= \Eb{x_0(\omega_\eta^\nc(\lambda)) \md \A_0(2\eps,1), U_0=1}
\end{cases}
\]
Then as $\eps$ and $\eta/\eps$ tend to zero, we have
\[
\beta \sim \beta_\lambda \sim \hat \beta_\lambda\,.
\]
\end{lemma}

\ni
{\bf Proof.} 
Let us fix $\lambda\in \R$, and assume without loss of generality that $\lambda>0$. 
Let us start with the first equivalent, $\beta_0 \sim \beta_\lambda$. Using the same technology as in \cite{\GPSa}, i.e. by relying on a coupling argument based on a near-critical RSW, it is easy to show that as $\eps$ and $\eta/\eps$ go to zero, one has:
\[
\Eb{x_0(\omega_\eta^\nc(\lambda))\md \A_0(2\eps, 1)} \sim \Eb{x_0(\omega_\eta^\nc(\lambda))\md \A_0(2\eps, \sqrt{\eps})}\,.
\]
The reason being that one has many logarithmic scales between radii $2\eps$ and $\sqrt{\eps}$ in order to couple the two conditional measures. 
The same argument shows that as $\eps, \eta/\eps \to 0$, 
\begin{align}\label{e.U0}
\Eb{x_0(\omega_\eta^\nc(\lambda))\md \A_0(2\eps, 1), U_0=1} \sim \Eb{x_0(\omega_\eta^\nc(\lambda))\md \A_0(2\eps, \sqrt{\eps}),U_0=1}\,.
\end{align}
This step is important: it explains why $\beta\sim \beta_\lambda$ as the $\lambda$ near-critical effect is almost invisible on the scale $\A_0(2\eps, \sqrt{\eps})$ (as it will be shown in Lemma~\ref{l.appr} below). We will analyse separately the numerators and denominators
\[
\Eb{x_0(\omega_\eta^\nc(\lambda))\md \A_0(2\eps, \sqrt{\eps})} = \frac{ \Eb{x_0(\omega_\eta^\nc(\lambda))\cap \A_0(2\eps, \sqrt{\eps})}} 
{\Pb{\omega_\eta^\nc(\lambda)\in\A_0(2\eps, \sqrt{\eps})}}\,.
\]
As is shown in Lemma 4.12. in \cite{\GPSa}, the numerator is well-approximated by 
\begin{align}\label{e.appr}
\eps^2/\eta^2 \frac{ \alpha_\square^\lambda(\eta, \sqrt{\eps}) }{ \alpha_\square^\lambda(2\eps, \sqrt{\eps})}\,.
\end{align}
See the notations in \cite{\GPSa}. The only ingredient to prove this estimate is the fact that the three-arm exponent in $\Hyp$ is 2 and this is still the case when $\lambda\neq 0$ by Proposition \ref{pr.ExNC}.  We now wish to show the following Lemma.
\begin{lemma}\label{l.appr}
For any fixed $\lambda>0$, 
\[
\alpha_\square^\lambda(\eta, \sqrt{\eps}) \sim \alpha_\square^0(\eta, \sqrt{\eps})\,,
\]
as $\eps, \eta/\eps \to 0$.
\end{lemma}

\ni
{\bf Proof.}
There are two ways to see why this lemma holds:
\bnum
\item One way is to notice that it follows from the proof of Lemma \ref{l.main}. Indeed this Lemma already gives that $\alpha_\square^\lambda(\eta,\sqrt{\eps}) \asymp \alpha_\square^0(\eta,\sqrt{\eps})$, but it is easy to check that if one stops the double induction $0< i <j $ at a level $j$ so that $r_j = 2^j\, \eta  \asymp \sqrt{\eta} \ll 1$, then as $\eta<\eps \to 0 $ one obtains constants in Lemma \ref{l.main} as close to 1 as one wishes. 
\item One may also use the differential inequalities from Kesten (see \cite{Kesten,PC}) and use the fact that the correlation length at level $\lambda>0$ is much larger than $\sqrt{\eps}$. 
\enum
\qed

Now, exactly as in this last lemma, one also has 
\[
\alpha_\square^\lambda(2\eps, \sqrt{\eps}) \sim \alpha_\square^0(2\eps, \sqrt{\eps})\,,
\]
as $\eps, \eta/\eps \to 0$. Thus one obtains the first asymptotic relation $\beta\sim \beta_\lambda$ in Lemma~\ref{l.I1}. For the second asymptotic relation, since we already have from Lemma~4.7 in \cite{\GPSa} that $\hat \beta \sim \beta$, we only need to check that $\hat \beta_\lambda(\eta,\eps)\sim \hat \beta_0(\eta,\eps)$ as $\eps, \eta/\eps\to 0$. This is done in the same manner as $\beta\sim \beta_\lambda$, i.e., by first relying on an approximation such as~\eqref{e.appr}, and then using the proof of Lemma~\ref{l.main} (the additional condition that $U_0=1$ is handled while following the proof of Lemma \ref{l.main}, namely if from $\omega_\eta(0)$ to $\omega_\eta(\lambda)$ one passes from $U_0=0$ to $U_1=1$, it also means that a pivotal point has been used, as in the double induction proof of Lemma \ref{l.main}). This finishes the proof of Lemma~\ref{l.I1} and Theorem~\ref{th.PMNC}. \qed


%
%
%

\begin{remark}\label{r.feller}
Let us end this section by pointing out that the simple Markov processes $t\mapsto \omega_\infty(t)$ and $\lambda\mapsto \omega_\infty(\lambda)$ are not Feller processes! Indeed it is not hard to build two configurations $\omega \sim \P_\infty$ and $\omega'\in \HH$ with $d_\HH(\omega,\omega')\ll 1$ which are such that their pivotal measures are very far apart which then induces very different dynamics starting from these initial points. Not being Feller does not exclude the possibility of being a strong Markov property, but it certainly makes it harder to prove:
\end{remark}

\begin{question}
Is $t\mapsto \omega_\infty(t)$ a {\bf strong} Markov process?
\end{question}

\section{Noise sensitivity and exceptional times}\label{s.ET}

We will start by establishing in Subsection \ref{ss.NS} an analog in our continuous setting of the noise-sensitivity results obtained for dynamical percolation in \cite{\GPS}.  We will then use the noise-sensitivity of the process $t\mapsto \omega_\infty(t)$ in order to obtain the a.s. existence of exceptional times for which there is an infinite cluster in $\omega_\infty(t)$.

%
\subsection{Noise sensitivity for $t\mapsto \omega_\infty(t)$}\label{ss.NS}

\begin{theorem}\label{th.NS1}
For any $Q\in \QUAD_\N$, there is a constant $C=C_Q<\infty$ s.t. for any $t\geq 0$, 
\begin{equation*}
\cov [1_{\boxminus_Q}(\omega_\infty(0)), 1_{\boxminus_Q}(\omega_\infty(t))] \leq C_Q\, t^{-2/3}\,. 
\end{equation*}
\end{theorem}

\ni
{\bf Proof:}
\ni
From the estimates (7.6) or (8.7) in \cite{\GPS}, one easily obtains that 
\begin{equation*}
\limsup_{\eta\to 0} \cov [1_{\boxminus_Q}(\omega_\eta(0)), 1_{\boxminus_Q}(\omega_\eta(t))] \asymp t^{-2/3}\,. 
\end{equation*}

Now, following the same proof as in Theorem \ref{th.MARGIN} and Proposition \ref{pr.LC}, one can prove that  
\[
\Pb{\omega_\eta(0) \text{ and } \omega_\eta(t) \in \boxminus_Q} 
\underset{\eta\to 0}{\longrightarrow} 
\Pb{\omega_\infty(0) \text{ and } \omega_\infty(t) \in \boxminus_Q}\,.
\]
Note that the difficulty here, as in Theorem \ref{th.MARGIN}, is to handle the possibility that there would be a sudden change before/after $t$, which would be almost invisible under the Skorohod distance from Definition \ref{d.skorohod}. \qed


Similarly, one has the following radial decorrelation result:
\begin{theorem}\label{th.NS2}
For any $0<r<R$, let $f_{r,R}$ be the indicator function of the event $\A_1(r,R)$ defined in Subsection \ref{ss.armevents}. There is a constant $C<\infty$ s.t. for all $0<r<R$ and for any $ t< r^{-3/4}$, 
\begin{align}
\Eb{f_{r,R}(\omega_\infty(0)) f_{r,R}(\omega_\infty(t))}  
& \leq C\,   \alpha_1(r, t^{-4/3}) \alpha_1(  t^{-4/3} ,R)^2 \nn \\
&  \leq C\, r^{-5/48}  t^{-5/36}\, \, \alpha_1(r,R)^2 \label{e.CORradial}
\end{align}

\end{theorem}

This is proved along the same lines as the proof of Theorem \ref{th.NS1}. This relies in particular on a two-scale version of Theorem 7.3. from \cite{\GPS}.

\begin{remark}\label{}
Theorem \ref{th.NS1} 
hints that the Markov process $t\mapsto \omega_\infty(t)$ should be {\bf ergodic}. We believe that this should indeed be the case, but in order to prove its ergodicity, we would need to control the decorrelation (or noise-sensitivity) of events like $\boxminus_{Q_1}\cap \boxminus_{Q_2}^c$ and the latter ones are not monotone events which prevents us from using the results and techniques from \cite{\GPS}. 
\end{remark}

\subsection{Exceptional times at the scaling limit}

\begin{theorem}\label{}
Almost surely, there exist exceptional times $t$, such that there is an infinite cluster in $\omega_\infty(t)$.  Furthermore, if $\mathcal{E} \subset (0,\infty)$ denotes the random set of such exceptional times, then $\mathcal{E}$ is almost surely of Hausdorff dimension 31/36. 
\end{theorem}

\ni
{\bf Proof:}
\ni
In this proof, we will denote the radial event $f_{r=1,R}$ from Theorem \ref{th.NS2} simply by $f_R$.
Let $X_R:= \int_0^1 f_R(\omega_\infty(s)) ds$. By definition, we have $\Eb{X_R}= \alpha_1(1,R)\asymp R^{-5/48}$. As in \cite{\SS,\GPS}, one has 
\begin{align*}\label{}
\Eb{X_R^2} & \leq 2 \int_0^1 \Eb{f_R(\omega_\infty(0)) f_R(\omega_\infty(s))} ds \\ 
& \leq 2C\, (\int_{0}^1 s^{-5/36} \, ds)  \alpha_1(1,R)^2  \\
& \leq  \tilde C \; \Eb{X_R}^2\,.
\end{align*}
By the standard second moment method, $\liminf_{R\to\infty} \Pb{X_R>0} >0$. Since the events $\{ X_R >0 \}$ are decreasing in $R$, by countable additivity, one obtains 
\begin{equation}\label{e.ET}
\Pb{ \cap_R \{ X_R >0 \} } >0\,. 
\end{equation}
If for each radius $R$, the random set of times $\left\{t\in [0,1], \text{ s.t. } \omega_\infty(t) \in \A_1(1,R) \right\}$ was a.s. a compact set, then the estimate ~\eqref{e.ET} would readily imply the existence of exceptional times. Unfortunately, our process $t\mapsto \omega_\infty(t)$ is c\`adl\`ag. One can still conclude using a similar trick as in the case of (discrete) dynamical percolation: since $t\mapsto \omega_\infty(t)$ is \cadlag, let $\{ t_i^+\}_{i\geq 1}\cup \{t_i^-\}_{i\geq 1}$ denote its countable set of discontinuities in $[0,1]$, where each discontinuity is marked $+$ if $\omega_\infty(t-) \leq \omega_\infty(t+)$ (i.e. the pivotal point responsible for the discontinuity turned open) and is marked $-$ otherwise. 
Let us consider the trajectory $t \mapsto \tilde\omega_\infty(t)$ which is identical to $t\mapsto \omega_\infty(t)$ outside of $\bigcup t_i^-$ and on $\bigcup t_i^-$, is defined by $\tilde \omega_\infty(u):= \lim_{\delta\to 0} \omega_\infty(u-\delta)$. For this process, one has decreasing compact sets as $R$ increases and the above proof leads to the existence of exceptional times for $t\mapsto \tilde \omega_\infty(t)$. Since by construction $\omega_\infty(t) \leq \tilde \omega_\infty(t)$, it could still be that there are exceptional times for $\tilde \omega_\infty(\cdot)$ but not for our process $\omega_\infty(\cdot)$. 
The purpose of Lemma 3.2.  in \cite{HPS} is to overcome this problem in the classical (discrete) model. It turns out that one can adapt the proof of this Lemma 3.2. to our present setting as follows. 
Divide the plane $\R^2$ into disjoint squares $Q_{n,m}=[n,n+1)\times [m,m+1)$. Let $\{ t^{m,n}_j\} \subset \bigcup \{t_i^-\}$ be the set of discontinuities which correspond to a pivotal point in $Q_{n,m}$. 
 Since the event of having an infinite cluster in $\omega_\infty(t)$ is independent of what happens in each fixed square $Q_{n,m}$, by countable additivity of $\{ t^{m,n}_j\}$, one concludes that a.s. there are no times $t^{m,n}_j$ s.t. $\tilde \omega_\infty(t^{m,n}_j)$ has an infinite cluster. This implies that if $t\mapsto \tilde \omega_\infty(t)$ has exceptional times, then all of these a.s.  arise outside of the discontinuity points.

%

 Finally, the fact that $\mathcal{E}$ is a.s.~of Hausdorff dimension $31/36$ follows in a classical way from the $t^{-5/36}$ estimate in the correlation bound~\eqref{e.CORradial} as it is explained for example in \cite{SchrammSteif} or in \cite{\GPS}. \qed
 
It was pointed out at the end of \cite{dynexit} that although the dimension of exceptional times coincides for the discrete and the continuum dynamical percolation processes, the tail behaviour of the time until the first exceptional time seems to be different: it is proved to be exponentially small for the discrete process, but is conjecturally only subexponential in the scaling limit.\margin{added this paragraph}

\section{Miscellaneous: gradient percolation, near-critical singularity, Loewner drift}\label{s.MIS}

\subsection{Gradient percolation}\label{ss.gradient}

In \cite{gradient}, the author considers the following gradient percolation model: in the domain $[0,1]^2$, consider an inhomogeneous percolation model on $\frac 1 n \Tg \cap [0,1]^2$ with parameter $p(z):=\Im(z), z\in [0,1]^2$. As the mesh $\frac 1 n \to 0$, it is straightforward to check that there is an interface between open and closed hexagons which localizes near the horizontal line $y=p_c=1/2$. This interface between the two phases is called the {\bf front}. Various critical exponents of this front are studied in \cite{gradient}; in particular, its typical distance from the midline was proved to be $f(n)/n=n^{4/7+o(1)}/n$; the exact definition for $f(n)$ should be 
\begin{equation}\label{e.retafn}
\frac{f(n)}{n} = r\left(\frac{1}{f(n)}\right)\,.
\end{equation} 
It is furthermore conjectured in \cite{gradient} that the front properly renormalized should have an interesting scaling limit, which is what we wish to discuss now. More than just the front itself, we can also prove the existence of a scaling limit for the entire gradient percolation configuration. See Theorem~\ref{th.gradient} below. Then, the front itself will be a measurable function of this scaling limit, in the same way as the $\SLE_6$ trace is measurable w.r.t. $\omega_\infty \in \HH$, as proved in \cite[Corollary 2.13]{GPS2a}.

Before stating a theorem, just like in near-critical percolation, one needs to renormalize gradient percolation in a suitable  manner:

\begin{definition}
For each $\eta>0$, let $\omega_\eta^{\gr}$ be the percolation model on $\eta \Tg$ with inhomogeneous parameter $p(z) =p_c + -1/2 \vee \Im(z) r(\eta) \wedge 1/2 $. 
This $\omega_\eta^\gr$ is exactly a scaled and centered copy of the above gradient percolation with $\eta=1/f(n)$, as follows from~(\ref{e.retafn}).
\end{definition}

\begin{theorem}[Scaling limit of gradient percolation]\label{th.gradient}
There is a random variable $\omega_\infty^\gr\in \HH=\HH_\C$, the {\bf continuum gradient percolation}, so that 
\[
\omega_\eta^\gr \overset{(d)}{\longrightarrow} \omega_\infty^\gr\,,
\]
as the mesh $\eta \to 0$. Furthermore, this gradient percolation $\omega_\infty^\gr$ corresponds to the inhomogeneous near-critical $\omega_\infty^{\nc,\phi}(\lambda=1)$ with $\phi(z):=\Im(z)$ defined in Definition \ref{d.INC}. 
\end{theorem}

\ni
{\bf Proof.}
The proof is rather straightforward at this stage of the paper: 
 it is enough to notice that $\omega_\eta^\gr$ is well-approximated by $\omega_\eta^{\nc,\phi}(\lambda=1)$ with $\phi(z)=\Im(z)$ and then to rely on the near-critical version of Theorem \ref{th.INC} as well as on an easy generalisation of Proposition \ref{pr.LC} to deal with the convergence of the latter. Details are omitted.  \qed


\ni
We end this subsection with the {\bf measurability of the front}. First of all, from the proof of Theorem \ref{th.CL}, we have that a.s.~for $\omega_\infty^\gr$, there is an infinite cluster in the upper half-plane and a dual infinite cluster in the lower half plane, which suggests an interface (or front) $\gamma_\infty$ between these two. Indeed, for any $0<r<R$, consider the subset of the plane 
\[
F_{r,R}:= \{ z\in \C, \omega_\infty^\gr \in \A_2(z,r,R) \}\,.
\]
In other words, $F_{r,R}$ is the set of points in $\C$ which have 2-arm (one dual, one primal) in the Euclidean annulus $A(x,r,R)$. 
Now, it is not hard to show using the proof of Theorem~\ref{th.CL} that the set 
\[
\gamma_\infty := \bigcap_{0<r<R<\infty} \closure{F_{r,R}}\,,
\]
is non-empty, and just like in \cite[Corollary 2.13]{GPS2a}, it is measurable w.r.t.~the gradient percolation scaling limit. We stress here that this is only the set of points in the front, without an ordering that would give the front as a curve.
See also \cite[Question 2.14]{GPS2a}.

\subsection{Singularity of $\omega_\infty(\lambda)$ w.r.t. $\omega_\infty(0)$ }\label{ss.sing}

The main result in \cite{NolinWerner} may be stated as follows:
\begin{theorem}[\cite{NolinWerner}]\label{th.NW}
Let $\lambda\neq 0$. Consider the interface $\gamma_\eta(\lambda)$ in the upper half-plane $\eta \Tg \cap \Hyp$ for the near-critical configuration $\omega_\eta^\nc(\lambda)$. Then, any subsequential scaling limit for $\{ \gamma_\eta(\lambda) \}_\eta$ is singular w.r.t. the $\SLE_6$ measure, the scaling limit of $\gamma_\eta(0)$.
\end{theorem}

Note that in this paper we obtain a scaling limit for $\omega_\eta^\nc(\lambda)$, and exactly as in \cite[Corollary 2.13]{GPS2a}, the trace of the interface in the upper half plane is measurable w.r.t.~this scaling limit. However, we did not prove the measurability of the interface as a curve (see \cite[Question 2.14]{\GPSa}), hence the subsequential limits in the above theorem are not exactly known yet to be an actual limit.\margin{I'm not sure that singularity of traces is also proved in NW, so I wrote this paragraph like this.} On the other hand, in the spirit of this theorem, we prove the following singularity result:

\begin{theorem}\label{}
Let $\lambda \neq 0$, then the near-critical continuum percolation $\omega_\infty^\nc(\lambda)$ is singular w.r.t. $\omega_\infty=\omega_\infty(0)\sim\P_\infty$. 
\end{theorem}

\begin{remark}\label{}
\ni\bnum
\item Note that such a result does not imply Theorem \ref{th.NW}. Indeed it could well be that $\omega_\infty^\nc(\lambda)$ and $\omega_\infty^\nc(0)$ are singular but their interfaces would look ``similar''.  In this sense, the singularity result provided by Theorem \ref{th.NW} is much finer than ours.
\item This singularity result has been proved independently and prior to our work by Simon Aumann in \cite{aumann}, but with a seemingly more complicated approach. 
\enum
\end{remark}

\ni
{\bf Proof.}
We wish to find a measurable event $A$ so that $\Pb{\omega_\infty^\nc(\lambda)\in A , \omega_\infty^\nc(0)\notin A} =1$. Let us start with the following lemma:

\begin{lemma}\label{l.SQ}
Let $\lambda>0$ be fixed. Let us denote by $\boxminus_u$ the crossing event which corresponds to the quad $Q_u:=[0,u]^2$. There is a constant $c=c_\lambda>0$ such that for any $u\in(0,1]$:
\[
\Pb{\omega_\infty^\nc(\lambda)\in \boxminus_u} \geq 1/2+ c u^{3/4}\,.
\]
\end{lemma}

\ni
{\bf Proof of Lemma \ref{l.SQ}.}
Using Theorem \ref{th.MARGIN}, we have that 
\[
\Pb{\omega_\infty^\nc(\lambda)\in \boxminus_u} = \lim_{\eta \to 0} \Pb{\omega_\eta^\nc(\lambda)\in \boxminus_u}\,.
\]
Now, by using the standard monotone coupling, one has uniformly as $u\to 0$ and  $\eta/u \to 0$:
\begin{align*}
& \hspace{-1 cm}\Pb{\omega_\eta^\nc(\lambda)\in \boxminus_u} - \Pb{\omega_\eta(0) \in \boxminus_u}  \\
& \geq  (1-e^{-\lambda}) \eta^2 \alpha_4^\eta(\eta,1)^{-1} \sum_{x\in \eta \Tg \cap Q_u} \frac 1 2 \P_{\lambda=0} \bigl[x \text{ is pivotal for } \boxminus_u\bigr]  \\
& \geq  (1-e^{-\lambda}) \Omega(1) \eta^2 \alpha_4^\eta(\eta,1)^{-1}  u^2 \eta^{-2} \alpha_4^\eta(\eta,u) \\
& \geq  (1-e^{-\lambda}) \Omega(1) u^2 \alpha_4^\eta(u,1)^{-1} \\ 
& \geq  C (1-e^{-\lambda}) u^{3/4}\,,
\end{align*}
where $C>0$ is some universal constant. The second inequality is obtained by classical separation of arms phenomenon plus RSW (see for example Chapter VI in \cite{Buzios}). The third inequality relies on the multiplicativity property. The last one uses \cite{MR1879816}. 
\qed

Consider now the square $[0,1]^2$ and for each $n\geq 1$, divide this square into $n^2$ squares of side length $1/n$. For each such square $Q$, one has by the previous lemma, 
\[
\Pb{\omega_\infty^\nc(\lambda) \in \boxminus_Q}= \Pb{\omega_\infty^\nc(\lambda) \in \boxminus_{1/n}} \geq c_\lambda\, n^{-3/4}\,. 
\]

Let $A_n$ be the event that there are at least $\frac{n^2} 2 + \frac {c_\lambda} 2 n^{5/4}$ squares in the above $1/n$-grid which are crossed horizontally. Since these events are independent, by a classical Hoeffding inequality, one can find a constant $a_\lambda>0$ so that 
\[
\begin{cases}
\P_\lambda[A_n] \geq 1- a_\lambda^{-1} \exp(-a_\lambda \, \sqrt{n}) \\
\P_0[A_n] \leq a_\lambda^{-1} \exp(-a_\lambda \, \sqrt{n})
\end{cases}
\] 
Clearly, the event $A:= \bigcup_{N\geq 1} \bigcap_{n\geq N} A_n$ is measurable, and by the Borel-Cantelli lemma it satisfies $\Pb{\omega_\infty^\nc(\lambda)\in A , \omega_\infty^\nc(0)\notin A} =1$, as desired. 
\qed

\subsection{A conjecture on the Loewner drift}\label{ss.off}

We present here a conjectural SDE for the driving function of the so-called {\bf massive chordal $\SLE_6$}: the Loewner chain of the scaling limit of the interface in near-critical percolation at $p = p_c + \lambda r(\eta)$ in the upper half plane, with open hexagons on the left  boundary and closed ones on the right. As we mentioned in Subsection~\ref{ss.links}, a general discussion of massive $\SLE_\kappa$'s, with focus on some special values of $\kappa$ other than 6, can be found in \cite{MS}.

Since zooming in spatially is equivalent to moving $\lambda$ closer to 0, we expect the driving function to be of the form 
\begin{equation}\label{e.NCLoew}
d W_t = \sqrt{6} \, d B_t + d A_t\,,
\end{equation}
where $B_t$ is Brownian motion and $A_t$ is a monotone drift, increasing for $\lambda>0$, decreasing for $\lambda<0$. In other words we expect $W_t$ to be a \textbf{submartingale} when  $\lambda>0$. This property does not seem to be obvious, and will be analyzed in \cite{SubM}. We conjecture the following precise form for the monotone drift $A_t$:  
\begin{equation}\label{e.NCdrift}
\begin{aligned}
d A_t & = c' \, \lambda \, |d\gamma_t|^{3/4}\,  |dW_t| = c'' \, \lambda \, |d\gamma_t|^{3/4} \, |dt|^{1/2}\,,
\end{aligned}
\end{equation}
where $|d\gamma_t|$ stands for the infinitesimal Euclidean increment length performed by the curve $\gamma_t$. Prior to proving a scaling limit of massive $\SLE_6$ towards this Loewner chain, making sense of a Loewner chain with such a degenerate drift already appears like a challenging mathematical problem. The intuition behind this conjecture will be discussed in more depth in \cite{SubM}. 

\addcontentsline{toc}{section}{Bibliography}


\ \\
{\bf Christophe Garban}\\
Universit\'e Lyon 1\\
\url{http://math.univ-lyon1.fr/~garban/}\\
Partially supported by the ANR grant MAC2 10-BLAN-0123.\\
\\
{\bf G\'abor Pete}\\
Alfr\'ed R\'enyi Institute of Mathematics, Hungarian Academy of Sciences, Budapest,\\
and Institute of Mathematics, Budapest University of Technology and Economics\\
\url{http://www.math.bme.hu/~gabor}\\
Partially supported by the Hungarian National Research, Development and Innovation Office, NKFIH grant K109684, and the MTA R\'enyi ``Lend\"ulet'' Limits of  Structures Research Group\\
\\
{\bf Oded Schramm} (December 10, 1961 -- September 1, 2008)\\
Microsoft Research\\ 
\url{http://research.microsoft.com/en-us/um/people/schramm/}\\

\end{document}